\documentclass[english]{myarticle}

\usepackage{enumitem}
\setlist[enumerate]{leftmargin=.5in}
\setlist[itemize]{leftmargin=.5in}

\title{Characteristic Operators and Spectral Properties of Periodic Evolutionary Systems}

\author{Bram Lentjes \thanks{Department of Mathematics, Hasselt University, Diepenbeek Campus, 3590 Diepenbeek, Belgium \email{(bram.lentjes@uhasselt.be)}}
\and Babette A. J. de Wolff\thanks{Department of Mathematics, University of Hamburg, 20146 Hamburg, Germany \email{(babette.de.wolff@uni-hamburg.de)}}}

\begin{document}
\maketitle
\begin{abstract}
In this paper, we introduce the notion of a characteristic operator for closable linear operators and explore their connected spectral properties via equivalence. Additionally, we develop an explicit scheme for constructing characteristic operators for a broad class of closable linear operators which are commonly encountered in periodic evolution equations. Our findings are illustrated through examples involving classical delay differential equations, delay differential equations with infinite delay and mixed functional differential equations. Notably, we resolve an open problem concerning the discrete spectral structure of the Floquet exponents for this latter class of differential equations. This work can be regarded as a natural and significant extension of the powerful framework developed by Kaashoek and Verduyn Lunel \cite{Kaashoek1992} on characteristic matrices and spectral properties induced by autonomous evolution equations.
\end{abstract}

% REQUIRED

% REQUIRED

\begin{sloppypar}
\section{Introduction} \label{sec:introduction}
Determining the stability of an invariant set is a fundamental task in the field of dynamical systems. Nowadays, various methods are available to tackle this problem, but spectral theory offers a particularly powerful approach. For example, consider the autonomous ordinary differential equation
\begin{equation} \tag{ODE} \label{intro:ODE}
    \dot{x}(t) = f(x(t)),
\end{equation}
where $x(t) \in \mathbb{C}^n$ and $f : \mathbb{C}^n \to \mathbb{C}^n$ denotes a sufficiently smooth vector field. Let $\overline{x} \in \mathbb{C}^n$ be an equilibrium of \eqref{intro:ODE} and recall from the classical Lyapunov stability theorem \cite{Lyapunov1892}, also known as the principle of linearized stability, that the stability of $\overline{x}$ is largely determined by the location of the spectrum $\sigma(A)$ of the matrix $A = Df(\overline{x})$ in $\mathbb{C}$. Now, let $\gamma : \mathbb{R} \to \mathbb{C}^n$ be a $T$-periodic solution of \eqref{intro:ODE} for some (minimal) period $T > 0$, and consider the (limit) cycle $\Gamma = \gamma(\mathbb{R})$ in $\mathbb{C}^n$. At first glance, one would expect that the stability of $\Gamma$ is now somehow determined by the family of spectra $\{\sigma(A(t)) : t \in \mathbb{R}\}$, where $A(t) = Df(\gamma(t))$. However, the simple counterexamples \cite{Coppel1978,Wu1974} illustrate that in general no principle of linearized stability holds in terms of $\sigma(A(t))$. Instead, one can compute the spectrum $\sigma(U(T,0))$ of the monodromy matrix $U(T,0)$, where $U$ denotes the fundamental matrix solution of the linearization of \eqref{intro:ODE} around $\Gamma$. A \emph{Floquet multiplier} $\lambda \in \sigma(U(T,0))$ is of the form $\lambda = e^{\sigma T}$, where $\sigma$ is called a \emph{Floquet exponent}, and it is well-known that the location of the Floquet exponents in $\mathbb{C}$ largely determines the stability of $\Gamma$, analogous to the principle of linearized stability for equilibria \cite{Hale2009}. Although this approach provides a satisfying answer for stability in terms of spectra, one still has to compute the solution operator $U$, which is often not explicitly available. In contrast, the stability of the equilibrium $\overline{x}$ could be directly determined in terms of the known matrix $A$. To address this issue, a recent observation in \cite{Article2}, obtained by recasting a result of Iooss (and Adelmeyer) \cite{Iooss1988,Iooss1999} in a functional analytic framework, introduces the closed linear operator $\mathcal{A} : \mathcal{D}(\mathcal{A}) \subseteq C_T(\mathbb{R},\mathbb{C}^n) \to C_T(\mathbb{R},\mathbb{C}^n)$ defined by
\begin{equation*}
    \mathcal{D}(\mathcal{A}) = C_T^1(\mathbb{R},\mathbb{C}^n), 
\qquad 
(\mathcal{A}\varphi)(t) = A(t)\varphi(t) - \dot{\varphi}(t).
\end{equation*}
Here, $C_T(\mathbb{R},\mathbb{C}^n)$ denotes the Banach space of $\mathbb{C}^n$-valued $T$-periodic functions on $\mathbb{R}$ equipped with the supremum norm, and the dot indicates differentiation with respect to $t$. With this formulation, one obtains the important characterization
\begin{equation} \label{intro:spectraequal}
    \sigma(\mathcal{A}) = \sigma_p(\mathcal{A})= \{ \sigma \in \mathbb{C} : \sigma \text{ is a Floquet exponent} \},
\end{equation}
where $\sigma_p(\mathcal{A})$ denotes the point spectrum of $\mathcal{A}$. Moreover, spectral properties of $\mathcal{A}$ have been fundamental in the study of codimension one \cite{Kuznetsov2005} and two \cite{Witte2014,Witte2013} bifurcations of limit cycles in ODEs. For example, the (generalized) eigenfunctions of $\mathcal{A}$ associated with the center subspace form a $T$-periodic basis along $\Gamma$, which in turn provides a natural coordinate system on the nearby periodic center manifold, see \cite{Iooss1988,Iooss1999,LentjesCMODE,Kuznetsov2023a} for more information. Explicit computational formulas for the critical normal form coefficients are expressed in these $T$-periodic (generalized) eigenfunctions and determine the nature and direction (degenerate, subcritical or supercritical) of these bifurcations. These formulas are implemented in several numerical bifurcation analysis software packages such as \verb|MatCont| \cite{Dhooge2003,Dhooge2008} and \verb|BifurcationKit| \cite{Veltz2020}, and are widely used by researchers across a broad range of scientific disciplines.

For infinite-dimensional dynamical systems, determining the stability of an invariant set is in general more challenging due to its infinite-dimensional nature. For certain subclasses, it is possible to make a dimension reduction so that the stability of an invariant set can be determined more easily. For example, consider the autonomous delay differential equation
\begin{equation} \tag{DDE} \label{intro:DDE}
    \dot{x}(t) = F(x_t), \quad t \geq 0,
\end{equation}
where $x(t) \in \mathbb{C}^n$ and $x_t \in X \coloneqq C([-h,0],\mathbb{C}^n)$ represents the \emph{history} of the unknown $x$ at time $t$ defined by $x_t(\theta) \coloneqq x(t+\theta)$ for all $\theta \in [-h,0]$. Here, $0 < h < \infty$ denotes the upper bound of (finite) delays and the infinite-dimensional \emph{state space} $X$ becomes a Banach space when equipped with the supremum norm. To study the stability of an equilibrium $\overline{x} \in X$ of \eqref{intro:DDE}, we are interested in the stability of the trivial solution of the linear DDE
\begin{equation} \tag{LDDE}\label{intro:LDDE}
    \dot{y}(t) = Ly_t, \quad t \geq 0,
\end{equation}
where $L = DF(\overline{x}) : X \to \mathbb{C}^n$ is a bounded linear operator. The solutions of \eqref{intro:LDDE} are generated by a strongly continuous semigroup on $X$ whose (infinitesimal) generator $A : \mathcal{D}(A) \subseteq X \to X$ reads
\begin{equation*}
    \mathcal{D}(A) = \{ \varphi \in X : \varphi' \in X, \ \varphi'(0) = L\varphi \}, \quad A \varphi = \varphi',
\end{equation*}
see \cite{Diekmann1995,Hale1993,Engel2000} for more information. From the same references, it follows that $A$ is a closed linear operator whose spectrum $\sigma(A)$ largely determines the stability of $\overline{x}$, ensuring that the principle of linearized stability holds for \eqref{intro:LDDE}. Computing $\sigma(A)$ may seem tedious at first, but Kaashoek and Verduyn Lunel provided in \cite{Kaashoek1992} a powerful framework to reduce the rather difficult spectral problem from the infinite-dimensional state space $C([-h,0],\mathbb{C}^n)$ towards a much simpler problem in its finite-dimensional codomain $\mathbb{C}^n$. In particular, they construct an equivalence after extension between the generator $A$ and the holomorphic matrix-valued function $\Delta : \mathbb{C} \to \mathbb{C}^{n \times n}$ with action
\begin{equation} \label{intro:Characmatrix}
    \Delta(z)q = zq - L[\theta  \mapsto e^{z \theta}q], \quad \forall q \in \mathbb{C}^n,
\end{equation}
of the form
\begin{equation} \label{intro:equivalence}
        \begin{pmatrix}
            \Delta(z) & 0 \\
            0 & I_{X}
        \end{pmatrix}
        =
        F(z)
        \begin{pmatrix}
            I_{\mathbb{C}^n}& 0 \\
            0 & zI - A
        \end{pmatrix}
        E(z), \quad \forall z \in \mathbb{C}.
\end{equation}
Here, $\Delta$ can be naturally obtained by substituting eigensolutions of the form $t \mapsto e^{zt}q$ into \eqref{intro:LDDE}, and $E$ and $F$ are holomorphic operator-valued functions whose values are bounded and bijective mappings between suitable Banach spaces. A direct consequence of \eqref{intro:equivalence} is that $\sigma(A)$ consists solely of isolated eigenvalues of finite type and satisfies
\begin{equation} \label{intro:spectraDDE}
    \sigma(A) = \sigma_p(A) = \{ z \in \mathbb{C} : \det \Delta(z) = 0 \}.
\end{equation}
Moreover, the equivalence \eqref{intro:equivalence} provides naturally an explicit representation of the resolvent operator $z \mapsto (zI - A)^{-1}$ in terms of $z \mapsto \Delta(z)^{-1}$. This leads to a concise proof of the so-called ``folk theorem in autonomous functional differential equations'' by reformulating it as a multiplicity theorem. It also yields an explicit description of the (generalized) eigenfunctions of $A$ in terms of the Jordan chains of $\Delta$, and ultimately, all elementary solutions of \eqref{intro:LDDE} can be expressed in terms of these Jordan chains, see \cite{Kaashoek1992,Hale1993,Diekmann1995} for further details and applications. In this way, the matrix-valued function $\Delta$ characterizes the spectral properties of the linear operator $A$, and is therefore called a \emph{characteristic matrix} for $A$. Moreover, the connection between $A$ and $\Delta$ serves as the foundation of (numerical) bifurcation analysis for \eqref{intro:DDE} near equilibria, see \cite{Bosschaert2020,Bosschaert2024a,Janssens2010,Diekmann1995} for further details as well as its implementation in numerical bifurcation analysis software packages such as \verb|DDE-BifTool| \cite{Sieber2014,Engelborghs2002,Krauskopf2022} and \verb|BifurcationKit|.

To formulate an analogous result to that for periodic orbits in the ODE-setting, let $\gamma$ be a $T$-periodic solution of \eqref{intro:DDE} and consider the (limit) cycle $\Gamma = \{\gamma_t \in X : t \in \mathbb{R} \}$ in $X$. To determine the stability of $\Gamma$, we are interested in the stability of the trivial solution of the periodic linear DDE
\begin{equation} \tag{LDDE$_T$} \label{intro:LDDET}
    \dot{y}(t) = L(t)y_t, \quad t \geq 0,
\end{equation}
where $L(t) = DF(\gamma_t) : X \to \mathbb{C}^n$ is a bounded linear operator. Recall from \cite{Diekmann1995,Hale1993} that its solutions are generated by a strongly continuous forward evolutionary system $U$ on $X$ whose generalized (infinitesimal) generator (at time $t$) $A(t) : \mathcal{D}(A(t)) \subseteq X \to X$ reads
\begin{equation*}
    \mathcal{D}(A(t)) = \{\varphi \in X : \varphi' \in X, \ \varphi'(0) = L(t)\varphi\}, \quad A(t)\varphi = \varphi',
\end{equation*}
see \cite{Article2,Article1,Clement1988} for more information. In this infinite-dimensional setting, the monodromy operator $U(T,0)$ becomes an iteratively compact bounded linear operator on $X$ for which the principle of linearized stability can be proven to hold as well \cite{Diekmann1995}. The natural question that arises is whether a characteristic matrix approach, as illustrated above for \eqref{intro:LDDE}, can also be exploited to analyse the spectral structure of \eqref{intro:LDDET}. 

The first results in this direction are due to Hale and Verduyn Lunel, who introduced in \cite{Hale1993} a characteristic matrix for $U(T,0)$ in the case of discrete delays under the assumption that the maximal delay $h$ is an integer multiple of the period $T$. The proof of this result relies on the previously mentioned framework of characteristic matrices developed in \cite{Kaashoek1992}. Later, Kaashoek and Verduyn Lunel (re-)proved in \cite{Kaashoek2022} the same statement using the general framework of characteristic matrices for classes of compact operators. Recently, these results were extended to systems with symmetry, relevant to equivariant Pyragas control \cite{Pyragas1992}, where a characteristic matrix can be constructed when the delay is related to a spatio-temporal symmetry of the solution \cite{Wolff2022a}. All these results demonstrate that characteristic matrices provide an effective tool for analysing the Floquet spectrum of \eqref{intro:LDDET}. However, they are confined to situations in which a specific relation between the period $T$ and the maximal delay $h$ is imposed. A first step towards removing these restrictions was undertaken by Szalai et al. \cite{Szalai2006}, who constructed a characteristic matrix $\Delta_n(z) \in \mathbb{C}^{n \times n}$ for the (inverse of the) Floquet multipliers of \eqref{intro:LDDET} without imposing additional assumptions on the relation between $h$ and $T$. Later, Sieber and Szalai \cite{Sieber2011} observed that $\Delta_n$ has poles in $\mathbb{C}$, and that these poles may coincide with the (inverse of the) Floquet multipliers, so that not all multipliers are captured by the constructed characteristic matrix. To remedy this, Sieber and Szalai modified $\Delta_n(z)$ to a characteristic matrix $\Delta_{nk}(z) \in \mathbb{C}^{nk \times nk}$ such that $\Delta_{nk}$ has no poles inside a disk of finite radius (dependent on $k$) in $\mathbb{C}$, but may still have poles outside this disk. Hence, the constructed characteristic matrix $\Delta_{nk}$ is guaranteed to capture the (inverse of the) Floquet multipliers within the given disk, but may fail to capture (inverted) multipliers outside this disk.

All in all, to date there is no construction of a characteristic matrix that captures the complete Floquet spectrum of \eqref{intro:LDDET} without imposing additional assumptions on the relation between the maximal delay $h$ and the period $T$. The present paper aims to contribute precisely in this direction. The price we pay for capturing all Floquet multipliers/exponents is that the resulting characteristic object is no longer matrix-valued, as in the previous constructions, but operator-valued. In other words, we construct a \emph{characteristic operator} that encodes the full Floquet spectrum of \eqref{intro:LDDET}. Our approach to constructing this characteristic operator is based upon the very recent observation from \cite{Bosschaert2025,Lentjes2026} that the spectral equality \eqref{intro:spectraequal} also holds for \eqref{intro:LDDET}, where now the non-closed but closable linear operator $\mathcal{A} : \mathcal{D}(\mathcal{A}) \subseteq C_T(\mathbb{R},X) \to C_T(\mathbb{R},X)$ takes the form
\begin{equation} \label{eq:introcurlyA}
    \mathcal{D}(\mathcal{A}) = \{\varphi \in C_T^1(\mathbb{R},X) : \varphi(t) \in \mathcal{D}(A(t)) \mbox{ for all } t \in \mathbb{R} \}, \quad (\mathcal{A}\varphi)(t) = A(t)\varphi(t) - \dot{\varphi}(t).
\end{equation}
Moreover, the authors of \cite{Bosschaert2025,Lentjes2026} established various interesting spectral relations between $\mathcal{A}$ and the holomorphic operator-valued function $\Delta$ with action
\begin{equation} \label{intro:Characoperator}
    (\Delta(z)q)(t) = \dot{q}(t) + zq(t) - L(t)[\theta \mapsto e^{z \theta}q(t+\theta)], \quad \forall q \in C_T^1(\mathbb{R},\mathbb{C}^n),
\end{equation}
which is naturally obtained by substituting Floquet eigensolutions of the form $t \mapsto e^{zt}q(t)$ into \eqref{intro:LDDET} with $z \in \mathbb{C}$. These results were subsequently applied in \cite{Bosschaert2025} to study codimension one bifurcations of limit cycles for \eqref{intro:DDE}, where the characteristic operator $\Delta$ from \eqref{intro:Characoperator} for $\mathcal{A}$ played an analogous role as the characteristic matrix $\Delta$ from \eqref{intro:Characmatrix} for $A$ in the analysis of bifurcations of equilibria \cite{Bosschaert2020}. In view of these results, it is natural to expect an equivalence after extension between $\mathcal{A}$ and $\Delta$ of the form \eqref{intro:equivalence}. In this paper, we confirm this expectation. However, establishing this equivalence, along with the associated spectral results, proved to be significantly more interesting and technical than we initially expected.

In order to contextualize our results, we point out that the existence of a characteristic operator, rather than a characteristic matrix, to capture all Floquet multipliers of \eqref{intro:LDDET} in the general case, is in line with the previously known results. In the construction by Sieber and Szalai \cite{Sieber2011}, the dimension of the characteristic matrix $\Delta_{nk}(z) \in \mathbb{C}^{nk \times nk}$ increases unboundedly when the radius (dependent on $k$) of the disk, on which the (inverted) Floquet multipliers are captured, is increased. Relatedly, the monograph \cite{Kaashoek2022} contains several examples where the maximal delay $h$ and period $T$ of \eqref{intro:LDDET} are rationally related. In these examples, the dimension of the characteristic matrix increases as the ratio $T/h$ becomes more irrational. We also note that similar observations were previously reported by Verduyn Lunel in \cite{Lunel2001}, Just in \cite{Just2000}, and Skubachevskii and Walter in \cite{Skubachevskii2006}. This all suggests that, if we aim to capture the full Floquet spectrum of \eqref{intro:LDDET} without any additional assumptions on the relation between the delay and the period, it is natural to work in the framework of characteristic operators rather than that of characteristic matrices.

Although our reduction of the spectral problem to a characteristic operator does not yield a dimension reduction in the classical sense (of characteristic matrices), it nevertheless leads to a substantial simplification of the spectral problem: whereas the original spectral problem for $\mathcal{A}$ is posed on $C_T(\mathbb{R},X)$ with $X = C([-h,0],\mathbb{C}^n)$, the characteristic operator $\Delta$ acts on the simpler space $C_T(\mathbb{R},\mathbb{C}^n)$. Moreover, working with the characteristic operator $\Delta$ rather than with the operator $\mathcal{A}$ has both computational and analytical advantages. Computationally, it turns out that the critical normal form coefficients of all codimension one bifurcations of limit cycles in \eqref{intro:DDE} can be expressed in terms of the characteristic operator $\Delta$, see \cite{Bosschaert2025} for further details. The key observation is that, analogous to the finite-dimensional ODE-setting described above, the (generalized) eigenfunctions of $\mathcal{A}$ associated with the center subspace form a $T$-periodic basis along the limit cycle $\Gamma$, see \cite{Article2} for further details. Most importantly, these (generalized) eigenfunctions of $\mathcal{A}$ can be expressed in terms of the Jordan chains of $\Delta$. In this framework, the connection between $\mathcal{A}$ and $\Delta$ serves as the foundation of (numerical) bifurcation analysis of \eqref{intro:DDE} near limit cycles, see \cite{Bosschaert2025,Article2,Article1} for further details as well as its implementation into the numerical bifurcation analysis software package \verb|PeriodicNormalizationDDEs| \cite{Bosschaert2024c}, which is compatible with \verb|BifurcationKit|. Analytically, the characteristic operator framework we present in this article is not only limited to \eqref{intro:LDDET} but also applies to a much broader class of periodic evolution equations, such as DDEs with infinite delay and functional differential equations of mixed type. In these contexts, we use the existence of a characteristic operator to derive new spectral results for these classes of differential equations. In particular, we establish, under relatively mild conditions, the existence of a spectral gap around the imaginary axis for periodic functional differential equations of mixed type, which was an open question in the literature before.

\subsection{Challenges and overview} 
\label{subsec:aims}
In the upcoming construction of a characteristic operator $\Delta$ for a linear operator $\mathcal{A}$, we encounter several functional analytical challenges.

The first challenge stems from the fact that the linear operator $\mathcal{A}$ is in general not closed, but only closable. In the analysis of its spectrum, it becomes clear that the conventional definition of the spectrum $\sigma(\mathcal{A})$ of $\mathcal{A}$ cannot be applied (\cref{remark:definitionresolvent}). Therefore, we present in \cref{subsec:spectralprop} a spectral theory, with an emphasis on the point spectrum and the isolated points in the spectrum, for closable linear operators. We start by analyzing the point spectrum algebraically using Jordan chains, then proceed to studying isolated spectral points analytically via contour integrals. However, the resulting spectral projections only identify the invariant subspaces of the closure, thereby failing to fully capture the distinct spectral structure of the original closable linear operator (\cref{remark:isolatedpoints}).

The second challenge arises from the need to define a suitable notion of equivalence for closable linear operators, as the literature is mainly concerned about a bounded equivalence between closed (or bounded) linear operators (\cref{remark:equivalencetypes}). With this extended definition of equivalence, we can still naturally define Jordan chains and prove their important one-to-one correspondences via equivalence, see \cref{subsec:equivalence} for this construction. Furthermore, we demonstrate that specific spectral numbers remain invariant under this equivalence.

The third challenge lies in proving that certain spectral properties are preserved under the equivalence between a closable linear operator $\mathcal{A}$ defined on a Banach space $X$ and a characteristic operator $\Delta$ whose values are defined on a (possibly different) Banach space $Y$. This task will be accomplished in \cref{subsec:char operators} under three additional (spectral) hypotheses: \ref{hyp:SH1}, \ref{hyp:SH2} and \ref{hyp:SH3}. As a consequence, all our results from \cref{sec:characoperators} are a natural extension of the mentioned work by Kaashoek and Verduyn Lunel on characteristic matrices, see \cref{remark:Yfinitedim} for further details.

In \cref{sec:construction}, we also develop a general scheme to construct for a given linear operator $\mathcal{A} : \mathcal{D}(\mathcal{A}) \subseteq \mathcal{F}_T(\mathbb{R},X) \to \mathcal{F}_T(\mathbb{R},X)$ of the form
\begin{equation} \label{intro:curlyAgeneral}
    \mathcal{D}(\mathcal{A}) = \{ \varphi \in \mathcal{F}_T(\mathbb{R},X) : \varphi \in \mathcal{D}(D), \ MD\varphi = K \varphi \}, \quad \mathcal{A}\varphi = D\varphi,
\end{equation}
a characteristic operator $\Delta$ with values $\Delta(z) : \mathcal{D}(\Delta(z)) \subseteq \mathcal{F}_T(\mathbb{R},Y) \to \mathcal{F}_T(\mathbb{R},Y)$, where $z \in \Omega \subseteq \mathbb{C}$. Here, $D$ and $K$ are closable linear operators while $M$ is a bounded linear operator. In applications, $D$ can be regarded as a maximal operator while $K$ and $M$ can be interpreted as generalized boundary-value operators. For $Z \in \{X,Y\}$, the Banach space $\mathcal{F}_T(\mathbb{R},Z)$ stands for an arbitrary function space consisting of $T$-periodic $Z$-valued functions defined on $\mathbb{R}$. Since many periodic evolution equations can be expressed in the form \eqref{intro:curlyAgeneral}, this scheme can be a applied to a wide range of periodic evolutionary systems, demonstrating its robust applicability. However, we already mention that this construction is significantly more technical than the original construction of Kaashoek and Verduyn Lunel, mainly due to the fact that the linear operator $\mathcal{A}$ is not necessarily closed and $\Delta$ is not necessarily matrix-valued, which prevents a strong spectral characterization like \eqref{intro:spectraDDE} in terms of point spectra and determinants. However, the characteristic operator still simplifies a challenging spectral problem in $\mathcal{F}_T(\mathbb{R},X)$ by reducing it towards a simpler problem in $\mathcal{F}_T(\mathbb{R},Y)$. In our upcoming applications (\cref{sec:applications}), the state space $X$ will typically be of the form $\mathcal{F}(I,Y)$ so that the spectral reduction is clear. Analogous to the characteristic matrix approach, we provide an explicit representation of the resolvent operator $z \mapsto (zI-\mathcal{A})^{-1}$ in terms of $z \mapsto \Delta(z)^{-1}$ (\cref{cor:spectralrelations}), a proof of the ``folk theorem in periodic functional differential equations'' by rephrasing it as an abstract multiplicity theorem (\cref{cor:multiplicity}), an explicit representation of the (generalized) eigenfunctions of $\mathcal{A}$ in terms of the Jordan chains of $\Delta$ (\cref{cor:Jordanchains}), and eventually all elementary solutions for a large class of periodic linear functional differential equations can be computed in terms of these Jordan chains (\cref{cor:FloquetsolDDE}). Moreover, we illustrate in \cref{remark:autonomousDDE} that the characteristic matrix from Kaashoek and Verduyn Lunel for linear autonomous systems can be obtained from our construction in two equivalent ways: setting naturally the period $T=0$ or restricting to the subspace of constant functions in $\mathcal{F}_T(\mathbb{R},Z)$, where this latter interpretation may be more insightful for \eqref{intro:LDDET} with $T/h \in \mathbb{Q}$, see \cref{remark:charoperatorTh} for further details. In addition, we show in \cref{subsec:periodicspectral} that the underlying $T$-periodic structure of $\mathcal{A}$ and $\Delta$ give rise to a $\frac{2 \pi i}{T} \mathbb{Z}$-periodic pattern in their spectrum. As a result, the spectral analysis of these linear operators can be further simplified by restricting it to a (subset of the) horizontal strip $\{z \in \mathbb{C} : \Im(z) \in (-\frac{\pi}{T}, \frac{\pi}{T}] \}$ rather than considering the entire complex plane.

In \cref{sec:applications} we apply our results to functional differential equations \eqref{eq:FDE}. As the class of periodic linear FDEs is rather large, we will illustrate the results from \cref{sec:construction} using three representative subclasses that frequently arise in applications:

\textbf{1) Classical DDEs (\cref{subsec:classicalDDEs}):} This subclass, introduced as our motivating example in \eqref{intro:LDDET}, naturally warrants inclusion here. We present detailed proofs for the results in this case, whereas the proofs for the next two subclasses follow similar reasoning and are therefore partly omitted. 

\textbf{2) iDDEs (\cref{subsec:infinitedelay}):} The subclass of DDEs with infinite delay is particularly interesting because the validity of the principle of linearized stability depends on the selected state space $X$. Furthermore, we note that exactly the same challenges arise, and can be addressed, as in the construction of a characteristic matrix for autonomous iDDEs \cite[Section II.1.3]{Kaashoek1992}. 

\textbf{3) Mixed FDEs (\cref{subsec:MFDEs}):} This subclass of FDEs highlights the strength of the characteristic operator framework, particularly since a forward evolutionary system $U$ is generally unavailable. Consequently, we affirm a conjecture by Hupkes and Verduyn Lunel \cite[Hypothesis (HF)]{Hupkes2008} concerning the discrete spectral structure of the Floquet exponents, see \cref{remark:MFDEspectral} for further details.

In \cref{sec:conclusion}, we summarize and discuss additional examples of periodic evolution equations on which our construction could be applied. Moreover, we list some conjectures and open problems that arise naturally from our results and may guide further research in this area.

\begin{remark}
We aim to engage two types of audiences: those interested in operator theory and those focused on applications in dynamics. For the former, we suggest exploring the theory presented in \cref{sec:characoperators}, while for the latter, we recommend starting with the applications in \cref{sec:applications} and later referring to the general scheme outlined in \cref{sec:construction}. \hfill $\lozenge$
\end{remark}

\section{Characteristic operators for closable linear operators} \label{sec:characoperators}

\subsection{Spectral properties of closable linear operators} \label{subsec:spectralprop}
We first introduce notation that we will use throughout the rest of the manuscript. Let $Y$ and $Z$ be normed (but not necessarily complete) spaces. The set of all linear operators from (a linear subspace of) $Y$ to $Z$ is denoted by $L(Y,Z)$ and we write $L(Z) \coloneqq L(Z,Z)$. Likewise, the set of all bounded linear operators from $Y$ to $Z$ is denoted by $\mathcal{L}(Y,Z)$ and we write $\mathcal{L}(Z) \coloneqq \mathcal{L}(Z,Z)$. Note that $\mathcal{L}(Y,Z)$ becomes a normed space when equipped with the operator norm, forms a Banach space when $Z$ is a Banach space, and forms a Banach algebra under composition when $Y=Z$ are Banach spaces.

Let $Y$ and $Z$ be Banach spaces and consider a linear operator $A: \mathcal{D}(A) \subseteq Y \to Z$ with \emph{domain} $\mathcal{D}(A)$, a linear subspace of $Y$. We call $\mathcal{N}(A) \coloneqq \{ \varphi \in \mathcal{D}(A) : A \varphi = 0 \}$ the \emph{kernel} of $A$, $\mathcal{R}(A) \coloneqq \{ A\varphi \in Z : \varphi \in \mathcal{D}(A) \}$ the \emph{range} of $A$ and $\Gamma(A) \coloneqq \{ (A\varphi, \varphi) \in Z \times Y : \varphi \in \mathcal{D}(A) \}$ the \emph{graph} of $A$. Throughout the rest of this paper, we will frequently use the notion of a closed and closable linear operator.

\begin{definition} \label{def:closedandclosable}
We call the linear operator $A$ \emph{closed} if its graph $\Gamma(A)$ is a closed subset of $Z \times Y$. Moreover, we call the linear operator $A$ \emph{closable} if $A$ admits a closed linear extension, meaning that there exists a closed linear operator $B : \mathcal{D}(B) \subseteq Y \to Z$ satisfying $\mathcal{D}(A) \subseteq \mathcal{D}(B)$ and $B \varphi = A \varphi$ for all $\varphi \in \mathcal{D}(A)$. \hfill $\lozenge$
\end{definition}

Let $A: \mathcal{D}(A) \subseteq Y \to Z$ be a closable linear operator, then there exists a smallest closed linear extension $\overline{A}$ of $A$, called the \emph{closure} of $A$, which can be characterized by $\Gamma(\overline{A}) = \overline{\Gamma(A)}$, see \cite[Proposition XIV.1.3]{Gohberg1990} for a detailed proof. As a direct consequence of the preceding definition and characterization, we obtain the following result concerning the connection between the domain, kernel and range of $A$ and $\overline{A}$. It is important to note that the inclusions in \eqref{eq:closurelemma} might be strict and that the topological closure is taken with respect to the norm on the associated Banach space.
\begin{lemma} \label{lemma:inclusions}
The following inclusions hold:
\begin{equation} \label{eq:closurelemma}
    \mathcal{D}(A) \subseteq \mathcal{D}(\overline{A}) \subseteq \overline{\mathcal{D}(A)}, \quad\mathcal{N}(A) \subseteq \overline{\mathcal{N}(A)} \subseteq \mathcal{N}(\overline{A}), \quad \mathcal{R}(A) \subseteq \mathcal{R}(\overline{A}) \subseteq \overline{\mathcal{R}(A)}.
\end{equation}
\end{lemma}
\begin{proof}
We first prove the statement on the domains. Here, the first inclusion is clear. To prove the second inclusion, let $\varphi \in \mathcal{D}(\overline{A})$ be given. Then $(\overline{A}\varphi,\varphi) \in \Gamma(\overline{A}) = \overline{\Gamma(A)}$. Thus, there exists a sequence $(\varphi_m)_m$ in $\mathcal{D}(A)$ such that $(A\varphi_m, \varphi_m) \to (\overline{A}\varphi, \varphi)$ as $m \to \infty$. This implies $\varphi_m \to \varphi$ as $m \to \infty$, meaning that $\varphi \in \overline{\mathcal{D}(A)}$.

We next prove the statement on the kernels. Here, the first inclusion is clear. To prove the second inclusion, let $\varphi \in \overline{\mathcal{N}(A)}$ be given. Then there exists a sequence $(\varphi_m)_m$ in $\mathcal{N}(A)$ such that $\varphi_m \to \varphi$ as $m \to \infty$ and thus $A\varphi_m = 0$ for all $m \in \mathbb{N}$. Hence, $(A\varphi_m,\varphi_m) \to (0,\varphi) \in \overline{\Gamma(A)} = \Gamma(\overline{A})$ as $m \to \infty$. Hence, $\varphi \in \mathcal{D}(\overline{A})$ and $\overline{A}\varphi = 0$ so that $\varphi \in \mathcal{N}(\overline{A})$.

We finally prove the statement on the ranges. Here, the first inclusion is clear. To prove the second inclusion, let $\psi \in \mathcal{R}(\overline{A})$ be given. Then there exists a $\varphi \in \mathcal{D}(\overline{A})$ such that $\overline{A}\varphi = \psi$ and thus $(\psi,\varphi) \in \Gamma(\overline{A}) = \overline{\Gamma(A)}$. Hence, there exists a sequence $(\varphi_m)_m$ in $\mathcal{D}(A)$ such that $\varphi_m \to \varphi$ and $A\varphi_m \to \overline{A}\varphi = \psi$ as $m \to \infty$. We conclude that $\psi \in \overline{\mathcal{R}(A)}$.  
\end{proof}
So if we write (with somewhat unusual, but later useful notation) $Y_A \coloneqq \mathcal{D}(\overline{A}) \subseteq Y$, then this normed space becomes a Banach space with respect to the \emph{graph norm} $\| \cdot \|_{\overline{A}} \coloneqq \| \cdot \| + \|\overline{A} \cdot \|$. In particular, $\overline{A}$ becomes a continuous (bounded) linear operator on $Y_A$ and $\mathcal{D}(A)$ is dense in $Y_A$ in the graph norm. Using the earlier introduced notation, we have that $A \in L(Y,Z)$ and $\overline{A} \in \mathcal{L}(Y_A,Z)$.

We now proceed to characterize the resolvent set and spectrum of the closable linear operator $A$. To this end, we assume that the Banach spaces $Y$ and $Z$ are complex and identical, and we denote this common space by $X$. As mentioned in the introduction, the classical notion of the spectrum is for our purposes not suitable for closable linear operators. Instead, we will work with the following extended notion of the spectrum introduced in \cite[Chapter 5]{Taylor1986}.

\begin{definition} \label{def:resolvent}
Let $X$ be a complex Banach space and $A : \mathcal{D}(A) \subseteq X \to X$ a closable linear operator. A complex number $z$ belongs to the \emph{resolvent set} $\rho(A)$ of $A$ if the operator $zI - A$ is injective, has dense range $\mathcal{R}(zI-A)$ in $X$, and the \emph{resolvent} of $A$ at $z$ 
defined by $R(z,A) \coloneqq (zI-A)^{-1}: \mathcal{R}(zI-A) \to \mathcal{D}(A)$ is a bounded linear operator. In this case, we say that the linear operator $zI - A$ is \emph{invertible}.

The \emph{spectrum} $\sigma(A)$ of $A$ is defined to be the complement of $\rho(A)$ in $\mathbb{C}$, and the \emph{point spectrum} $\sigma_p(A) \subseteq \sigma(A)$ of $A$ is the set of those $\sigma \in \mathbb{C}$ such that $\sigma I - A$ is not injective, i.e. $A \varphi = \sigma \varphi$ for some nonzero \emph{eigenvector} $\varphi \in \mathcal{D}(A)$. If $\sigma \in \sigma_p(A)$, we will also call $\sigma$ an \emph{eigenvalue} of $A$. \hfill $\lozenge$
\end{definition}

\begin{remark} \label{remark:definitionresolvent}
We emphasize that the above definition of the resolvent set for closable linear operators extends the classical definition (\cite[Definition IV.1.1]{Engel2000}) for closed linear operators, which requires only that $zI - A$ is a linear bijection. In our setting, the surjectivity condition is relaxed to the requirement that $zI - A$ has dense range, while additionally insisting that $R(z, A)$ is bounded. If $A$ is a closed (and thus closable) linear operator with $z \in \rho(A)$, then it can be shown by an application of the closed graph theorem that $zI - A$ is actually surjective and its inverse is automatically bounded \cite[Theorem 4.2.E]{Taylor1986}. Therefore, for closed linear operators, \cref{def:resolvent} agrees with the classical one.

The reason we cannot work with the usual definition is that most of our operators of interest are in general not closed but only closable (\cref{prop:DAnotclosed}), and the domains of their resolvents are not necessarily the full space $X$ (\cref{prop:rangenotequal}). With the usual standard definition, recall from \cite[Proposition XIV.1.2]{Gohberg1990} that $\rho(A) = \emptyset$ for a non-closed operator $A$, which would be very impractical for our purposes. \hfill $\lozenge$
\end{remark}

For the remainder of this section, we fix a complex Banach space $X$ and a closable linear operator $A : \mathcal{D}(A) \subseteq X \to X$. The first fundamental result on the resolvent and spectrum of $A$ is stated below. Since it follows directly from \cite[Theorems 5.1.A and 5.1.B]{Taylor1986}, the proof is omitted.

\begin{proposition} \label{prop:resolventopen}
The resolvent set $\rho(A)$ is open in $\mathbb{C}$, and hence the spectrum $\sigma(A)$ is closed. In particular, if $z_2 \in \rho(A)$ and $z_1 \in \mathbb{C}$ is such that $|z_2-z_1| < 1/ \|R(z_2,A)\|$, then $z_1 \in \rho(A)$.
\end{proposition}

The next step is to analyse in detail the relationship between the spectral objects associated with $A$ and $\overline{A}$.

\begin{lemma} \label{lemma:spectraequalclosure}
There holds $\rho(A) = \rho(\overline{A})$ and in particular
\begin{equation} \label{eq:resolventclosure}
    R(z,A) = R(z,\overline{A})|_{\mathcal{R}(zI-A)}, \quad \forall z \in \rho(A).
\end{equation}
Consequently, $\sigma(A) = \sigma(\overline{A})$ and there holds $\sigma_p(A) \subseteq \sigma_p(\overline{A})$. Moreover, if $\sigma_p(A)\cap\Omega=\sigma(A)\cap\Omega$ for some $\Omega \subseteq \mathbb{C}$, then $\sigma_p(A) \cap \Omega = \sigma_p(\overline{A}) \cap \Omega$.
\end{lemma}
\begin{proof}
We first prove that $\rho(A) \subseteq \rho(\overline{A})$, and let therefore $z \in \rho(A)$ be given. Since $\mathcal{D}(R(z,A)) = \mathcal{R}(zI - A)$ is dense in $X$, the bounded linear operator $R(z,A)$ extends uniquely to a bounded linear operator $\overline{R(z,A)} \in \mathcal{L}(X)$. Since $\Gamma(\overline{A}) = \overline{\Gamma(A)}$, one can check that $\overline{R(z,A)} = R(z,\overline{A})$, meaning that $z \in \rho(\overline{A})$. To prove the other inclusion, let $z \in \rho(\overline{A})$ be given. Note that $zI-A$ is injective as the linear extension $zI-\overline{A}$ is injective. To show that $\mathcal{R}(zI-A)$ is dense in $X$, we apply \cref{lemma:inclusions} onto the linear operator $zI-A$, which is closable as $A$ is closable and $I$ is bounded (\cref{lemma:closablesum}). This yields $X = \mathcal{R}(zI-\overline{A}) \subseteq \overline{\mathcal{R}(zI-A)}$ and thus $\overline{\mathcal{R}(zI-A)} = X$. Clearly, $R(z,A)$ is bounded since $R(z,\overline{A})$ is a bounded linear extension. We conclude that $z \in \rho(A)$.

To prove \eqref{eq:resolventclosure}, it remains to show that $R(z,\overline{A})|_{\mathcal{R}(zI-A)}$ maps into $\mathcal{D}(A)$. Indeed, if $\psi \in \mathcal{R}(zI-A)$, then there exists $\varphi \in \mathcal{D}(A)$ such that $(zI-A)\varphi = \psi$. Since $\overline{A}$ extends $A$, we have $(zI-\overline{A})\varphi = \psi$, which implies $R(z,\overline{A})\psi = \varphi \in \mathcal{D}(A)$.

To prove the claim regarding the point spectrum, let $\sigma \in \sigma_p(A)$ be given. The result now follows by applying the kernel statement of \cref{lemma:inclusions} to the closable linear operator $\sigma I - A$.

The last claim follows from $\sigma_p(\overline{A}) \cap \Omega \subseteq \sigma(\overline{A}) \cap \Omega = \sigma(A) \cap \Omega = \sigma_p(A) \cap \Omega \subseteq \sigma_p(\overline{A}) \cap \Omega$. Here, the first equality follows from the first part of the proof, and the second from the additional assumption.
\end{proof}

It is important to observe that the inclusion of point spectra in \cref{lemma:spectraequalclosure} may be strict since there exist closable linear operators that are injective but whose closures are not, see \cite[Section 19.2.5]{Mortad2022} for further details. However, one can prove that spectral values in $\sigma_p(\overline{A}) \setminus \sigma_p(A)$ actually lie in the approximate point spectrum of $A$, though we will not pursue this direction further. Recall that this situation cannot arise if $\sigma_p(A) \cap \Omega = \sigma(A) \cap \Omega$ for a subset $\Omega \subseteq \mathbb{C}$, see the applications presented in \cref{sec:applications}.

To further characterize the point spectrum of $A$, let $\sigma \in \sigma_p(A)$ be given. The linear subspace $\mathcal{N}(\sigma I - A)$ of $X$ is called the \emph{eigenspace} of $A$ at $\sigma$, and its dimension $m_g(\sigma, A) \in \mathbb{N} \cup \{\infty\}$ is referred to as the \emph{geometric multiplicity} of $A$ at $\sigma$. Consider the following closed linear subspaces
\begin{equation*}
    E_\sigma(A) \coloneqq \overline{\bigcup_{l \geq 1} \mathcal{N}((\sigma I - A)^l)}, \quad Q_\sigma(A) \coloneqq \overline{\bigcap_{l \geq 1} \mathcal{R}((\sigma I - A)^l)},
\end{equation*}
of $X$. Here, $E_\sigma(A)$ is called the \emph{generalized eigenspace} of $A$ at $\sigma$ and $Q_\sigma(A)$ is called the \emph{complementary generalized eigenspace} of $A$ at $\sigma$. The dimension $m_a(\sigma,A) \in \mathbb{N} \cup \{ \infty \}$ of $E_\sigma(A)$ is referred to as the \emph{algebraic multiplicity} of $A$ at $\sigma$. Observe that $E_\sigma(A)$ is the smallest closed linear subspace of $X$ that contains the ascending chain of linear subspaces $\mathcal{N}((\sigma I - A)^l)$ for all integers $l \geq 1$. Moreover, $m_g(\sigma,A) \leq m_a(\sigma,A)$ and the condition $m_g(\sigma,A)=\infty$ implies that $m_a(\sigma,A)=\infty$. For the case $m_g(\sigma,A) < \infty$, a more detailed description of $E_\sigma(A)$ requires the notion of Jordan chains for closable linear operators.

\begin{definition} \label{def:Jordanchain}
Let $X$ be a complex Banach space and $A : \mathcal{D}(A) \subseteq X \to X$ a closable linear operator. An ordered set $\{\varphi_0,\dots,\varphi_{k-1}\}$ of vectors in $\mathcal{D}(A)$ is called a \emph{Jordan chain} of $A$ at $\sigma \in \mathbb{C}$ if $\varphi_0 \neq 0$ and
\begin{equation*}
    (A-\sigma I) \varphi_i = 
    \begin{dcases}
        0, \quad &i=0, \\
        \varphi_{i-1}, \quad &i=1,\dots,k-1.
    \end{dcases}
\end{equation*}
The vectors $\varphi_1,\dots,\varphi_{k-1}$ are called \emph{generalized eigenvectors} of $A$ at $\sigma$. The number $k \in \mathbb{N}$ is called the \emph{length} of the chain. If there exists a maximal length of the chain starting with $\varphi_0$, then we call this length the \emph{rank} of $\varphi_0$. If there is no maximal length, then we say that $\varphi_0$ has \emph{infinite rank}. \hfill $\lozenge$
\end{definition}

Let us assume that $p = m_g(\sigma,A) < \infty$. Following the algebraic procedure of Gohberg and Sigal \cite{Gohberg1971}, we may arrange the Jordan chains as follows. Let $\varphi_{1,0},\dots,\varphi_{p,0}$ be a basis of $\mathcal{N}(\sigma I - A)$ chosen such that the ranks $k_j \in \mathbb{N} \cup \{\infty\}$ of the eigenvector $\varphi_{j,0}$ are maximized sequentially. The (possibly infinite) integers $k_1 \leq \dots \leq k_p$, which are independent of the choice of such a basis, represent the sizes of the Jordan blocks and are referred to as the \emph{partial multiplicities} of $\sigma$. We define the largest partial multiplicity $k_p \eqqcolon k(\sigma,A) \in \mathbb{N} \cup \{\infty\}$ to be the \emph{ascent} of $A$ at $\sigma$. In this framework, the algebraic multiplicity satisfies $m_a(\sigma,A) = \sum_{j=1}^p k_j$. If in addition $m_a(\sigma,A) < \infty$, then $\sigma$ is called an \emph{eigenvalue of finite type} of $A$ since in that case $m_g(\sigma,A)$, $m_a(\sigma,A)$ and $k(\sigma,A)$ are all finite. In particular, $k(\sigma,A)$ is the smallest integer such that 
\begin{equation} \label{eq:EsigmaQsigma}
    E_\sigma(A) = \mathcal{N}((\sigma I - A)^{k(\sigma,A)}), \quad Q_\sigma(A) \subseteq \overline{\mathcal{R}((\sigma I - A)^{k(\sigma,A)})},
\end{equation}
where we recall that $E_\sigma(A)$ is closed as it is finite-dimensional. Note that the second inclusion might be strict (\cite[Problem 5.4.2]{Taylor1986}). Moreover, $E_\sigma(A)$ admits a basis of the form
\begin{equation} \label{eq:canonicalbasis}
\varphi_{1,0},\dots,\varphi_{1,k_1-1},\dots,\varphi_{p,0},\dots,\varphi_{p,k_p-1},
\end{equation}
which is called a \emph{canonical basis of (generalized) eigenvectors} of $A$ at $\sigma$. With respect to this specific basis, the matrix representation of $A|_{E_\sigma(A)} \in \mathcal{L}(E_{\sigma}(A))$ is a Jordan matrix with $\sigma$ on the main diagonal. The partitioning of the vectors in \eqref{eq:canonicalbasis} corresponds precisely to the decomposition of this Jordan matrix into individual Jordan blocks. The following lemma clarifies the relationship between these notions for $A$ and $\overline{A}$.

\begin{lemma} \label{lemma:mgakineq}
If $\sigma \in \sigma_p(A)$, then $m_g(\sigma,A) \leq m_g(\sigma,\overline{A})$, $m_a(\sigma,A) \leq m_a(\sigma,\overline{A})$ and $ k(\sigma,A) \leq k(\sigma,\overline{A})$, and we have the inequalities
\begin{align} 
\begin{split} \label{eq:chainineq}
    &m_g(\sigma,A) + k(\sigma,A) - 1 \leq m_a(\sigma,A) \leq m_g(\sigma,A) k(\sigma,A), \\
    &m_g(\sigma,\overline{A}) + k(\sigma,\overline{A}) - 1 \leq m_a(\sigma,\overline{A}) \leq m_g(\sigma,\overline{A}) k(\sigma,\overline{A}).  
\end{split}
\end{align}   
\end{lemma}
\begin{proof}
Let us first recall from \cref{lemma:spectraequalclosure} that $\sigma$ is also in $\sigma_p(\overline{A})$. The first claim follows from the fact that $\mathcal{N}(\sigma I - A) \subseteq \mathcal{N}(\sigma I - \overline{A})$ due to \cref{lemma:inclusions}. To prove the second claim, an inductive argument shows that $\mathcal{N}((\sigma I -A)^l) \subseteq \mathcal{N}((\sigma I - \overline{A})^l)$ for all integers $l \geq 1$. As the union and closure operations preserve the order of inclusion, we obtain $E_\sigma(A) \subseteq E_\sigma(\overline{A})$, which proves the second claim. Since any Jordan chain of $A$ at $\sigma$ is also a Jordan chain of $\overline{A}$ at $\sigma$, the third claim holds. 

We now prove \eqref{eq:chainineq} for $A$. First, if $m_g(\sigma,A) = \infty$, then $m_a(\sigma,A) = \infty$ and $k(\sigma,A) \geq 1$ so that \eqref{eq:chainineq} holds. Second, if $m_g(\sigma,A) < \infty$ while $m_a(\sigma,A) = \infty$, then necessarily $k(\sigma,A) = \infty$, and \eqref{eq:chainineq} holds since $m_g(\sigma,A) \geq 1$. Third, if $m_g(\sigma,A) < \infty$ but $k(\sigma,A) = \infty$, then $m_a(\sigma,A) = \infty$ and thus \eqref{eq:chainineq} is again immediate. It therefore remains to consider the case where $\sigma$ is of finite type. Set $p = m_g(\sigma,A) < \infty$ and recall that we have ordered the canonical basis of (generalized) eigenvectors of $A$ at $\sigma$ in such a way that $k(\sigma,A) = k_p$. Then
\begin{align*}
m_g(\sigma,A) + k(\sigma,A) - 1
&= \underbrace{1 + \cdots + 1 }_{p-1\ \text{times} }  +  k_p
\leq \sum_{j=1}^p k_j
= m_a(\sigma,A),
\end{align*}
where the inequality follows from the fact that $k_j \geq 1$ for all $j = 1,\dots,p$. This proves the first inequality in \eqref{eq:chainineq}. The second inequality follows from $m_a(\sigma,A) = \sum_{j=1}^p k_j \leq pk_p = m_g(\sigma,A) k(\sigma,A)$, as $k_p$ is maximal. Since $\overline{A}$ is closed (and thus closable), the second chain of inequalities in \eqref{eq:chainineq} holds as well.
\end{proof}

With the framework for the point spectrum of the closable linear operator $A$ established, we extend our analysis to the full spectrum $\sigma(A)$. This requires first a thorough investigation of the resolvent set $\rho(A)$ and the associated \emph{resolvent operator} $R(\cdot,A)$ of $A$. In order to develop a coherent spectral theory for closable linear operators, paralleling to the classical framework available for closed linear operators, one needs to make additional assumptions on the behaviour of $R(\cdot,A)$. In particular, we would like to use a \emph{resolvent identity} of the form:
\begin{equation*}
    R(z_1,A)-R(z_2,A) = (z_2-z_1)R(z_1,A)R(z_2,A), \quad \forall z_1,z_2 \in \rho(A).
\end{equation*}
However, note that the expressions in this identity are not necessarily well-defined, since i) on the left-hand side of the equation, the domains $\mathcal{D}(R(z_i, A)) = \mathcal{R}(z_iI - A)$ might be different for $i \in \{1,2\}$, and ii) on the right-hand side of the equation, the operator $R(z_2, A)$, which has range $\mathcal{R}(R(z_2, A)) = \mathcal{D}(A)$, might not map into the domain $\mathcal{D}(R(z_1, A)) = \mathcal{R}(z_1I - A)$. In order to bypass these issues, we introduce the following (spectral) hypothesis:
\begin{enumerate}[label=({SH}{{\arabic*}})]
    \item \label{hyp:SH1} The domain $\mathcal{D}(A) \subseteq \mathcal{R}(zI-A)$ for all $z \in \rho(A)$.
\end{enumerate}
Under the assumption of \ref{hyp:SH1}, we show in the next result that $\mathcal{D}(R(z,A)) = \mathcal{R}(zI - A)$ is independent of $z \in \rho(A)$. This ensures that both sides of the resolvent identity are well-defined and makes the equality itself true. It is worth noting that \ref{hyp:SH1} represents a deviation from the treatment in \cite[Section 5.1]{Taylor1986}. Specifically, the analogue of the following result in that work (\cite[Theorem 5.1.C]{Taylor1986}) is established under the stronger assumption that $\mathcal{R}(zI-A) = X$ whenever $z \in \rho(A)$.

\begin{proposition} \label{prop:resolventproperties}
Assume that $A$ satisfies \ref{hyp:SH1} and let $z_1, z_2 \in \rho(A)$ be given. Then
\begin{equation} \label{eq:rangeequality}
    \mathcal{R}(z_1I - A) = \mathcal{R}(z_2I - A),
\end{equation}
and the resolvent identity
\begin{equation} \label{eq:resolventeq}
R(z_1,A)-R(z_2,A) = (z_2-z_1)R(z_1,A)R(z_2,A),
\end{equation}
holds. Moreover, the resolvent operator $R(\cdot,A)$ of $A$ is analytic on $\rho(A)$ since
\begin{equation} \label{eq:resolventseries}
    R(z_1,A) = \sum_{l=0}^{\infty} (z_2-z_1)^l R(z_2,A)^{l+1},
\end{equation}
for all $z_2 \in \rho(A)$ and $z_1 \in \mathbb{C}$ such that $|z_1 - z_2| < 1/ \|R(z_2,A)\|$.
\end{proposition}
\begin{proof}
To prove \eqref{eq:rangeequality}, we first show that $\mathcal{R}(z_1I - A) \subseteq \mathcal{R}(z_2I - A)$ for $z_{1},z_{2} \in \rho(A)$. The other inclusion then follows by interchanging $z_1$ and $z_2$. Fix $\psi \in \mathcal{R}(z_1I- A)$, then there exists a $\varphi \in \mathcal{D}(A)$ such that $(z_1I-A)\varphi = \psi$. We next observe that $(z_2 I - A) \varphi = \psi + (z_2 - z_1) \varphi$ or equivalently 
\begin{equation*}
\psi = (z_2 I - A) \varphi- (z_2 - z_1) \varphi.
\end{equation*}
The first term on the right-hand side is an element of $\mathcal{R}(z_2 I - A)$, and the second term is an element of $\mathcal{D}(A) \subseteq \mathcal{R}(z_2 I - A)$ by \ref{hyp:SH1}. So $\psi \in \mathcal{R}(z_2 I - A)$, and thus $\mathcal{R}(z_1 I - A) \subseteq \mathcal{R}(z_2 I - A)$, as claimed. 

To prove \eqref{eq:resolventeq}, we first observe that $(z_2 - z_1)I = (z_2 I - A) - (z_1 I - A)$. Composing this identity from the left with $R(z_1, A)$, which is well-defined as $z_2 I - A$ maps into 
$\mathcal{D}(R(z_2, A)) = \mathcal{R}(z_2 I - A) = \mathcal{R}(z_1 I - A)$ by \eqref{eq:rangeequality}, and composing the resulting identity from the right by $R(z_2, A)$ yields the result. 

To prove \eqref{eq:resolventseries}, consider $z_2 \in \rho(A)$ and $z_1 \in \mathbb{C}$ satisfying $|z_1 - z_2| < 1/ \|R(z_2,A)\|$. Then \cref{prop:resolventopen} tells us that $z_1 \in \rho(A)$. By subsequently applying \eqref{eq:resolventeq}, induction yields the expression
\begin{equation} \label{eq:resolventexp}
    R(z_1,A) = \sum_{l=0}^{m} (z_2-z_1)^l R(z_2,A)^{l+1} + (z_2-z_1)^{m+1} R(z_1,A)R(z_2,A)^{m+1}, \quad \forall m \in \mathbb{N},
\end{equation}
which is well-defined due to \ref{hyp:SH1} and \eqref{eq:rangeequality}. Note that the remainder term in \eqref{eq:resolventexp} vanishes (in norm) as $m \to \infty$ since $\|R(z_1,A)\| [ |z_1-z_2| \|R(z_2,A)\|]^{m+1} \to 0$ as $m \to \infty$ since $|z_1 - z_2| \|R(z_2,A)\| < 1$. This proves \eqref{eq:resolventseries} and so the proof is complete.
\end{proof}

With the framework for the point spectrum and the resolvent set of a closable linear operator $A$ established, we extend our analysis to isolated points of $\sigma(A)$ using \ref{hyp:SH1}. Our goal is to recover a spectral decomposition of $X$ into two $A$-invariant linear subspaces similar to the classical setting of closed linear operators, see for example the results in \cite[Section XV.2]{Gohberg1990}. However, this turns out to be ambitious. Indeed, since $A$ is not necessarily closed, we cannot guarantee that the residue of $A$ at an isolated eigenvalue $\sigma$ of $A$ is a projection operator, as will be detailed below. Instead, we decompose with respect to invariant subspaces of $\overline{A}$ and subsequently try to establish the connection with the corresponding subspaces related to $A$. We start with the following result, which lifts the classical construction of a Laurent series for $R(z,\overline{A})$ towards $R(z,A)$ around an isolated point in $\sigma(A)$.

\begin{lemma} \label{lemma:Dunford}
Assume that $A$ satisfies \ref{hyp:SH1}. If $\sigma$ is an isolated point of $\sigma(A)$, then the resolvent of $A$ at $\sigma$ can be locally expanded as a Laurent series
\begin{equation} \label{eq:Dunford}
    R(z,A) = \sum_{l = -\infty}^{\infty} (z - \sigma)^l R_l(\sigma,A), \quad R_l(\sigma,A) = \frac{1}{2 \pi i} \oint_{\Gamma_{\sigma}} (z-\sigma)^{-(l+1)}R(z,A) dz,
\end{equation}
whenever $0 < |z-\sigma| < \delta$ for some sufficiently small $\delta$ and $\Gamma_{\sigma}$ denotes a positively oriented boundary of a disc centered at $\sigma$ with radius strictly smaller than $\delta$. 
\end{lemma}
\begin{proof}
We first recall from \cref{lemma:spectraequalclosure} that $\sigma(A) = \sigma(\overline{A})$ and hence $\sigma$ is also an isolated point of $\sigma(\overline{A})$. Since $\overline{A}$ is a closed linear operator, the results from \cite[Section IV.1.17]{Engel2000} tell us that
\begin{equation} \label{eq:Dunfordclosure}
    R(z,\overline{A}) = \sum_{l = -\infty}^{\infty} (z - \sigma)^l R_l(\sigma,\overline{A}), \quad R_l(\sigma,\overline{A}) \coloneqq \frac{1}{2 \pi i} \oint_{\Gamma_{\sigma}} (z-\sigma)^{-(l+1)}R(z,\overline{A}) dz,
\end{equation}
for all $z \in \mathbb{C}$ sufficiently close to $\sigma$. Due to \eqref{eq:resolventclosure}, it is clear that $R_l(\sigma,A) = R_l(\sigma,\overline{A})|_{\mathcal{R}(zI-A)}$ for all $l \in \mathbb{Z}$. We next want to prove the expression for $R_l(\sigma, A)$ claimed in \eqref{eq:Dunford}. To that end, denote by $W_l(\sigma,A)$ the integral term in \eqref{eq:Dunford}, so that we want to prove the equality $W_l(\sigma, A) = R_l(\sigma, A)$. We observe that this equality follows
if we can prove that $\mathcal{D}(W_l(\sigma,A)) = \mathcal{R}(zI-A)$. To show this, note that $\Gamma_\sigma \subset \rho(A)$ and so \cref{prop:resolventproperties} tells us that $\mathcal{R}(wI - A)$ is independent of $w \in \Gamma_\sigma$. Hence, $\mathcal{D}(W_l(\sigma,A)) = \cap_{w \in \Gamma_\sigma} \mathcal{R}(wI-A) = \mathcal{R}(zI-A)$ for all $l \in \mathbb{Z}$, which proves the claim. Restricting both sides of the first equality in \eqref{eq:Dunfordclosure} to $\mathcal{R}(zI-A)$ proves the first equality in \eqref{eq:Dunford}.
\end{proof}
If there is an $r \in \mathbb{N}$ such that $R_{-r}(\sigma,A) \neq 0$ while $R_{-l}(\sigma,A) = 0$ for all $l > r$, then $\sigma$ is called a \emph{pole} of $R(\cdot,A)$ of \emph{order} $r \eqqcolon r(\sigma,A)$ and we set $r(\sigma,A) = \infty$ if $R(\cdot,A)$ has an essential singularity at $\sigma$.

\begin{lemma} \label{lemma:orderpole}
If $A$ satisfies \ref{hyp:SH1} and $\sigma$ is an isolated point of $\sigma(A)$, then $r(\sigma,A) = r(\sigma,\overline{A})$.    
\end{lemma}
\begin{proof}
Recall from \cref{lemma:spectraequalclosure} that $\sigma$ is also an isolated point of $\sigma(\overline{A})$. Since $R_l(\sigma,A) = R_l(\sigma,\overline{A})|_{\mathcal{R}(zI-A)}$ is bounded for all $l \in \mathbb{Z}$ and $\mathcal{R}(zI-A)$ is dense in $X$ for $z \in \rho(A)$, it follows that $R_l(\sigma,A) = 0$ if and only if $R_l(\sigma,\overline{A}) = 0$, which proves the result. 
\end{proof}

Consider the bounded linear operator $P_{\sigma}(A) \coloneqq R_{-1}(\sigma,A)$ at an isolated point $\sigma$ in $\sigma(A)$, which is precisely the \emph{residue} of $R(\cdot,A)$ at $\sigma$. Since $A$ is not necessarily closed, we cannot ensure in general that the Dunford integral representation of $P_\sigma(A)$ from \eqref{eq:Dunford} maps into $\mathcal{D}(A)$. Consequently, $P_\sigma(A)$ may not define a projection on $X$, as the composition $P_\sigma(A)^2$ may fail to be well-defined. However, as the closure $\overline{A}$ of $A$ is a closed linear operator, the standard results on closed linear operators (\cite[Section XV.2]{Gohberg1990}) imply that the \emph{Riesz spectral projection} $P_\sigma(\overline{A}) = R_{-1}(\sigma, \overline{A})$ is a projection on $X$ and decomposes this space into two closed $\overline{A}$-invariant linear subspaces as
\begin{equation*}
    X = \mathcal{R}(P_\sigma(\overline{A})) \oplus \mathcal{R}(I-P_\sigma(\overline{A})), \quad \mathcal{N}(P_\sigma(\overline{A})) = \mathcal{R}(I-P_\sigma(\overline{A})).
\end{equation*}
Moreover, the identity \eqref{eq:resolventclosure} implies that $P_\sigma(A) = P_\sigma(\overline{A})|_{\mathcal{R}(zI - A)}$ for some (and thus all) $z \in \rho(A)$. This in turn implies that $P_\sigma(A)$ is densely defined, and that its (unique) bounded linear extension $\overline{P_\sigma(A)}$ coincides with $P_\sigma(\overline{A})$. By combining these facts, we obtain a decomposition of $X$ into the topological closure of the ranges of $P_\sigma(A)$ and $I-P_\sigma(A)$, as we state in the following result.

\begin{lemma} \label{lemma:directsum}
If $A$ satisfies \ref{hyp:SH1} and $\sigma$ is an isolated point of $\sigma(A)$, then
\begin{equation} \label{eq:directsumXA}
    X = \overline{\mathcal{R}(P_\sigma(A))} \oplus \overline{\mathcal{R}(I-P_\sigma(A))}, \quad \overline{\mathcal{N}(P_\sigma(A))} \subseteq \overline{\mathcal{R}(I-P_\sigma(A))}.
\end{equation}
\end{lemma}
\begin{proof}
\cref{lemma:inclusions} together with the fact that $\overline{P_\sigma(A)} = P_\sigma(\overline{A})$ tells us that $\mathcal{R}(P_\sigma(A)) \subseteq \mathcal{R}(P_\sigma(\overline{A})) \subseteq \overline{\mathcal{R}(P_\sigma(A))}$. Since $P_\sigma(\overline{A})$ is a bounded linear projection, its range is closed. Thus, taking the closure of these inclusions implies $\mathcal{R}(P_\sigma(\overline{A})) = \overline{\mathcal{R}(P_\sigma(A))}$. A similar argument for the closable linear operator $I - P_\sigma(A)$, with closure the projection $I - P_\sigma(\overline{A})$, shows that $\mathcal{R}(I - P_\sigma(\overline{A})) = \overline{\mathcal{R}(I - P_\sigma(A))}$. 

To prove the remaining claim, let $\varphi \in \mathcal{N}(P_\sigma(A))$ be given. Then $(I - P_\sigma(A))\varphi = \varphi$ and thus $\varphi \in \mathcal{R}(I-P_\sigma(A))$. Hence, $\mathcal{N}(P_\sigma(A)) \subseteq \mathcal{R}(I-P_\sigma(A))$ and taking closures yields the claim.
\end{proof}

It is important to note that the reverse inclusion in the second part of \eqref{eq:directsumXA} does not hold in general because $P_\sigma(A)$ is not necessarily a projection. This observation is consistent with the general kernel properties from \cref{lemma:inclusions} and the earlier derived identity $\mathcal{N}(P_\sigma(\overline{A})) = \overline{\mathcal{R}(I-P_\sigma(A))}$ from \cref{lemma:directsum}.

To study the isolated points of $\sigma(A)$ in more detail, let us now assume that $\sigma \in \sigma(A)$ is isolated and satisfies $r(\sigma,A) < \infty$. Combining \cref{lemma:spectraequalclosure} and \cref{lemma:orderpole} yields $\sigma \in \sigma(\overline{A})$ and $r(\sigma, \overline{A}) < \infty$. Moreover, the spectral theory for closed linear operators tells us that $\sigma \in \sigma_p(\overline{A})$ and
\begin{equation*}
    E_\sigma(\overline{A}) = \mathcal{R}(P_\sigma(\overline{A}))= \mathcal{N}((\sigma I - \overline{A})^{k(\sigma,\overline{A})}), \quad Q_{\sigma}(\overline{A}) = \mathcal{N} (P_\sigma(\overline{A})) = \mathcal{R}((\sigma I - \overline{A})^{k(\sigma,\overline{A})}),
\end{equation*}
where it is important to note that $k(\sigma,\overline{A}) = r(\sigma,\overline{A})$, and that $\sigma$ is not necessarily an eigenvalue of $A$ (\cref{lemma:spectraequalclosure}). If in addition $\sigma \in \sigma_p(A)$, then \cref{lemma:mgakineq} and \cref{lemma:orderpole} yields the possible strict inequality $k(\sigma,A) \leq r(\sigma,A)$ and inclusion $E_\sigma(A) \subseteq E_\sigma(\overline{A})$.

\begin{remark} \label{remark:isolatedpoints}
Even when $\sigma$ is an isolated eigenvalue of finite type of $A$, implying that $\mathcal{R}(P_\sigma(A))$ is finite-dimensional and therefore closed, this range may differ from our more algebraic definition of the generalized eigenspace $E_\sigma(A)$ of $A$ at $\sigma$. This discrepancy arises because $P_\sigma(A)$ can map vectors out of $\mathcal{D}(A)$, leading to the possible strict inclusion $E_\sigma(A) \subsetneq \mathcal{R}(P_\sigma(A)) = E_\sigma(\overline{A})$. Consequently, a spectral analysis based purely on the isolated points of $A$ does not fully reflect the spectral structure of $A$ itself, but it rather captures most of the relevant spectral information of its closure $\overline{A}$. \hfill $\lozenge$
\end{remark}

In the sequel, we aim to analyse the spectral structure of two specific linear operators associated to a given closable linear operator $A$. First, we study spectral relations between $A$ and a restriction of $A$ to a subspace (\cref{def:partof}). Second, we investigate spectral relations between $A$ and a conjugated form of $A$ (\cref{def:similar}).

\begin{definition} \label{def:partof}
Let $X_{|}$ be a linear subspace of a complex Banach space $X$ and consider a linear operator $A : \mathcal{D}(A) \subseteq X \to X$. The linear operator $A_{|} : \mathcal{D}(A_{|}) \subseteq \overline{X_{|}} \to \overline{X_{|}}$ defined by
\begin{equation*}
    \mathcal{D}(A_{|}) \coloneqq \{ \varphi \in \mathcal{D}(A) \cap X_{|}  : A\varphi \in X_{|} \}, \quad A_{|}\varphi \coloneqq A\varphi.
\end{equation*}
is called the \emph{part of} $A$ in $X_{|}$. \hfill $\lozenge$
\end{definition}

The following result considers a closable linear operator $A$ and relates the spectral data of $A$ with that of $A_{|}$, which can be regarded as an extension of \cite[Lemma IV.1.15]{Engel2000} and \cite[Proposition 1.1]{Kaashoek1992} from closed to closable linear operators.

\begin{proposition} \label{prop:spectraequal}
Let $A : \mathcal{D}(A) \subseteq X \to X$ be a closable linear operator and let $A_{|}$ denote the part of $A$ in $X_{|} \subseteq X$. Then $A_{|}$ is closable and the following statements hold:
\begin{enumerate}
    \item If $\mathcal{D}(A) \subseteq X_{|}$, then there holds:
    \begin{enumerate}
        \item $\sigma_p(A_{|}) = \sigma_p(A)$.
        \item If $\sigma \in \sigma_p(A_{|})$, then $m_g(\sigma,A_{|}) = m_g(\sigma,A), m_a(\sigma,A_{|}) = m_a(\sigma,A)$ and $k(\sigma,A_{\vert}) = k(\sigma,A)$. If in addition $\sigma$ is of finite type, then the partial multiplicities of $\sigma$ considered as an eigenvalue of $A_{|}$ are equal to the partial multiplicities of $\sigma$ considered as an eigenvalue of $A$.
    \end{enumerate}

    \item If $\mathcal{D}(A) \subseteq X_{|} \subseteq \mathcal{R}(zI-A)$ for all $z \in \rho(A)$, then $A$ satisfies \ref{hyp:SH1}, and there holds:
    \begin{enumerate}
        \item $\sigma(A_{|}) \subseteq \sigma(A)$.
        \item If $\sigma \in \sigma(A_{|})$ is isolated, then $r(\sigma,A_{|}) = r(\sigma,A)$.
    \end{enumerate}
\end{enumerate}
\end{proposition}
\begin{proof}
To prove that $A_{|}$ is closable, we observe that $A$ is a linear extension of $A_{|}$, and $\overline{A}$ is a closed linear extension of $A$. Therefore $\overline{A}$ is a closed linear extension of $A_{|}$, which proves the claim. To prove the first statement, assume that $\mathcal{D}(A) \subseteq X_{|}$. Since $A$ is a linear extension of $A_{|}$, it follows that $\sigma_p(A_{|}) \subseteq \sigma_p(A)$. To prove the other inclusion, let $\sigma \in \sigma_p(A)$ be given. Then there exists a nonzero $\varphi \in \mathcal{D}(A)$ such that $A \varphi = \sigma \varphi \in \mathcal{D}(A) \subseteq X_{|}$. Hence, \cref{def:partof} tells us that $\varphi \in \mathcal{D}(A_{|})$ and therefore $\sigma \in \sigma_p(A_{|})$, which proves (a). To prove (b), let $\sigma \in \sigma_p(A_{|})$ be given. Clearly, $\mathcal{N}(\sigma I-A) = \mathcal{N}(\sigma I-A_{|})$ due to \cref{def:partof} and $\mathcal{D}(A) \subseteq X_{|}$, and thus $m_g(\sigma, A_{|}) = m_g(\sigma, A)$. An induction argument shows that $\mathcal{N}((\sigma I - A)^l) = \mathcal{N}((\sigma I - A_{|})^l) \subseteq \mathcal{D}(A_{|}) \subseteq \mathcal{D}(A)$ for any integer $l \geq 1$, and thus $m_a(\sigma,A_{|}) = m_a(\sigma,A)$. Clearly, $k(\sigma,A_{|}) = k(\sigma,A)$ if one of them is infinite. Assume now in addition that $\sigma$ is of finite type. Since the action of $A_{|}$ equals that of $A$ on $\mathcal{D}(A_{|})$, we have that any canonical basis of (generalized) eigenvectors of $A$ at $\sigma$ is also a canonical basis of (generalized) eigenvectors of $A_{|}$ at $\sigma$, and vice versa. This proves the claim regarding the partial multiplicities and the ascent.

To prove the second statement, assume that $\mathcal{D}(A) \subseteq X_{|} \subseteq \mathcal{R}(zI-A)$ for all $z \in \rho(A)$. This in particular implies that $A$ satisfies \ref{hyp:SH1}. To prove (a), fix $z \in \rho(A)$ and note that $zI - A_{|}$ is injective as $\mathcal{D}(A_{|}) \subseteq \mathcal{D}(A)$. To prove that $zI-A_{|}$ has dense range in $\overline{X_{|}}$, we show that $\mathcal{R}(zI-A_{|}) = X_{|}$. First, note from \cref{def:partof} that $\mathcal{R}(zI - A_{|}) \subseteq X_{|}$. Conversely, if $\varphi \in X_{|}$, then $\varphi \in \mathcal{R}(zI-A)$ by the additional assumption and thus $R(z,A)\varphi = \psi \in \mathcal{D}(A) \subseteq X_{|}$ is well-defined. Hence, $\varphi = (zI-A_{|})\psi \in \mathcal{R}(zI-A_{|})$, which proves the claim. Note that $R(z,A)_{|} : X_{|} \to \mathcal{D}(A_{|})$ is the algebraic inverse of $zI - A_{|}$ since $(zI-A_{|})R(z,A)_{|} = I_{X_{|}}$ and $R(z,A)_{|} (zI-A_{|}) = I_{\mathcal{D}(A_{|})}$. Hence, $R(z,A_{|}) = R(z,A)_{|}$, where the latter operator is bounded as $R(z,A)$ is bounded. We conclude that $z \in \rho(A_{|})$, which proves the claim. To prove (b), let $\sigma$ be an isolated point of $\sigma(A_{|})$. As $R(z,A_{|}) = R(z,A)_{|}$ for $z \in \rho(A) \subseteq \rho(A_{|})$, there holds $R_l(\sigma,A_{|}) = R_l(\sigma,A)_{|}$ for all $l \in \mathbb{Z}$, which proves the claim.
\end{proof}

\begin{definition} \label{def:similar}
Let $X$ and $Y$ be complex Banach spaces, and let $A: \mathcal{D}(A) \subseteq X \to X$ and  $B: \mathcal{D}(B) \subseteq Y \to Y$ be linear operators. If there exists an invertible bounded linear operator $V: X \to Y$ such that 
\begin{equation*}
    \mathcal{D}(B) = \{ \varphi \in Y :  V^{-1}\varphi \in \mathcal{D}(A) \}, \quad B = VAV^{-1},
\end{equation*}
then we say that $A$ and $B$ are \emph{similar} (by $V$) and denote this by $A \sim B$. \hfill $\lozenge$
\end{definition}

The following result considers a closable linear operator $A$ and relates the spectral data of $A$ with that of $B$, which can be regarded as an extension of the findings in \cite[Section II.2.1]{Engel2000} from closed to closable linear operators.

\begin{proposition} \label{prop:similarity}
Let $A: \mathcal{D}(A) \subseteq X \to X$ be a closable linear operator. If $A \sim B$ by $V : X \to Y$, then $B$ is closable, and the following statements hold:
\begin{enumerate}
    \item $\sigma_p(B) = \sigma_p(A)$ and $\sigma(B) = \sigma(A)$.
    \item If $\sigma \in \sigma_p(A)$, then $m_g(\sigma,B) = m_g(\sigma,A), m_a(\sigma,B) = m_a(\sigma,A)$ and $k(\sigma, B) = k(\sigma, A)$. If in addition $\sigma$ is of finite type, then the partial multiplicities of $\sigma$ considered as an eigenvalue of $A$ are equal to the partial multiplicities of $\sigma$ considered as an eigenvalue of $B$.
    \item If $A$ satisfies \ref{hyp:SH1}, then $B$ satisfies \ref{hyp:SH1}. If in addition $\sigma \in \sigma(A)$ is isolated, then $r(\sigma,B) = r(\sigma,A)$.
\end{enumerate}
\end{proposition}
\begin{proof}
We first prove that $B$ is closable by an application of \cref{lemma:closabledef}. Let $(\varphi_m)_m$ be a sequence in $\mathcal{D}(B)$ converging in norm to zero and suppose that $(B\varphi_m)_m$ converges in norm to some $\psi \in Y$. Since $V \in \mathcal{L}(X,Y)$ is invertible, the bounded inverse theorem guarantees $V^{-1} \in \mathcal{L}(Y,X)$. Consequently, the sequence $(V^{-1}\varphi_m)_m$ lies in $\mathcal{D}(A)$ and converges to zero. We observe that $A V^{-1} \varphi_m = V^{-1} B \varphi_m \to V^{-1}\psi$ as $m \to \infty$ by the continuity of $V^{-1}$. Since $A$ is closable, $V^{-1}\varphi_m \to 0$ and $AV^{-1}\varphi_m \to V^{-1}\psi$ as $m \to \infty$, it follows that $V^{-1}\psi = 0$. Since $V^{-1}$ is injective, $\psi = 0$ and thus $B$ is closable.

To prove the first statement, observe that $V$ (or $V^{-1}$) provides a one-to-one map from eigenvectors of $B$ (or $A$) onto eigenvectors of $A$ (or $B$), and so $\sigma_p(B) = \sigma_p(A)$. To prove $\sigma(B) = \sigma(A)$, let $z \in \rho(A)$ be given. Since $zI - A$ is injective, $zI - B = V(zI - A) V^{-1}$ is injective as well. Moreover, it holds that $Y = V(\overline{\mathcal{R}(zI-A)}) \subseteq \overline{V(\mathcal{R}(zI-A))} = \overline{\mathcal{R}(zI-B)}$ due to the continuity of $V$, so $zI - B$ has dense range. Moreover, as $R(z,B) = V R(z,A)V^{-1}$, we see that $R(z,B)$ is bounded as $V$ and $V^{-1}$ are bounded. Hence, $z \in \rho(B)$ and so $\rho(A) \subseteq \rho(B)$. The other inclusion can be proven by interchanging the roles of $A$ and $B$, and thus $\sigma(B) = \sigma(A)$. 

To prove the second statement, recall from the first part of the proof that $\mathcal{N}(\sigma I-B) = V(\mathcal{N}(\sigma I-A))$ and therefore $m_g(\sigma,B) = m_g(\sigma,A)$ as $V$ is bijective. An induction argument shows that $\mathcal{N}((\sigma I-B)^l) = V(\mathcal{N}((\sigma I-A)^l))$ for any integer $l \geq 1$, and thus $m_a(\sigma,B) = m_a(\sigma,A)$. Clearly, $k(\sigma,B)  = k(\sigma,A)$ if one of them is infinite. Assume now in addition that $\sigma$ is of finite type. If $\varphi_{1,0},\dots,\varphi_{1,k_{1}-1},\dots,\varphi_{p,0},\dots,\varphi_{p,k_{p}-1}$ is a canonical basis of (generalized) eigenvectors of $A$ at $\sigma$, then $V\varphi_{1,0},\dots,V\varphi_{1,k_{1}-1},\dots,V\varphi_{p,0},\dots,V\varphi_{p,k_{p}-1}$ yields a canonical basis of (generalized) eigenvectors of $B$ at $\sigma$ since $V$ is bounded and invertible. For the converse, one uses the fact that $V^{-1}$ is also bounded and invertible. This proves the claim regarding the partial multiplicities and the ascent.

To prove the third statement, let us assume that $A$ satisfies \ref{hyp:SH1} and let $\varphi \in \mathcal{D}(B)$ be given. Then $V^{-1}\varphi \in \mathcal{D}(A) \subseteq \mathcal{R}(zI-A)$ for all $z \in \rho(A)$. Note that $\varphi = V(V^{-1} \varphi) \in V(\mathcal{R}(zI-A)) = \mathcal{R}(zI-B)$, which shows that $\mathcal{D}(B) \subseteq \mathcal{R}(zI-B)$. Now, let $\sigma$ be an isolated point of $\sigma(A)$. As $R(z,B) = VR(z,A)V^{-1}$ for $z \in \rho(A)$, there holds $R_l(\sigma,B) = VR_l(\sigma,A)V^{-1}$ for all $l \in \mathbb{Z}$, which proves the claim.
\end{proof}

\subsection{Equivalence and Jordan chains} \label{subsec:equivalence}
Let $X,X',Y$ and $Y'$ be complex Banach spaces and $\Omega \subseteq \mathbb{C}$ a nonempty open set. In this section, we will consider an operator-valued function $L : \Omega \to L(X,Y)$. We call $L$ a \emph{bounded (closed or closable) operator-valued function} if the linear operator $L(z) : \mathcal{D}(L(z)) \to Y$ is bounded (closed or closable) for all $z \in \Omega$. We recall the definition of holomorphicity of type (A) from \cite[Section VII.2.1]{Kato1995} for operator-valued functions.

\begin{definition}
An operator-valued function $L : \Omega \to L(X,Y)$ is called \emph{holomorphic} on $\Omega$ if the domain $\mathcal{D}(L(z)) = \mathcal{X}$ is $z$-independent for all $z \in \Omega$ and the map $\Omega \ni z \to L(z)\varphi \in Y$ is holomorphic for all $\varphi \in \mathcal{X}$. \hfill $\lozenge$
\end{definition}

If $A$ is a closable linear operator, the operator-valued function $\mathbb{C} \ni z \mapsto zI-A \in L(X)$ is holomorphic and $\rho(A) \ni z \mapsto R(z,A) \in L(X)$ is holomorphic under the additional assumption of \ref{hyp:SH1}. This holds because i) $\mathbb{C}$ and $\rho(A)$ are open in the complex plane (\cref{prop:resolventopen}), ii) the domains $\mathcal{D}(zI-A) = \mathcal{D}(A)$ and $\mathcal{D}(R(z,A)) = \mathcal{R}(zI-A)$ are $z$-independent (\cref{prop:resolventproperties}), and iii) the following derivatives exist:
\begin{equation} \label{eq:examplederi}
\frac{d}{dz} (zI-A)\varphi = \varphi, \quad \frac{d}{dz} R(z,A)\psi = -R(z,A)^2 \psi,
\end{equation}
for all $\varphi \in \mathcal{D}(A)$ and $\psi \in \mathcal{R}(zI-A)$, see also \cite[Theorem 5.1-C]{Taylor1986}  for a proof of the second equality in \eqref{eq:examplederi}. The following definition is an extension of the notion of equivalence between closed and bounded linear operators presented in \cite[Section I.1.2]{Kaashoek1992}.

\begin{definition} \label{def:equi}
Two holomorphic closable operator-valued functions $L : \Omega \to L(X,Y)$ and $M : \Omega \to L(X',Y')$, with $z$-independent domains $\mathcal{D}(L(z)) =: \mathcal{X}$ and $\mathcal{D}(M(z)) =: \mathcal{X}'$, are called \emph{equivalent} on $\Omega$ if the following statements hold:
\begin{enumerate}
    \item There exist holomorphic closable operator-valued functions $E : \Omega \to L(X',X)$ and $F : \Omega \to L(Y,Y')$ whose values $E(z) : \mathcal{X}' \to \mathcal{X}$ and $F(z) : \mathcal{Y} \to \mathcal{Y}'$ are bijective with $\mathcal{D}(F(z)) \eqqcolon \mathcal{Y}$ and $\mathcal{R}(F(z)) \eqqcolon \mathcal{Y}'$ being $z$-independent for all $z \in \Omega$.
    \item $E(\cdot)^{-1} : \Omega \to L(X,X')$ and $F(\cdot)^{-1} : \Omega \to L(Y',Y)$ are holomorphic closable operator-valued functions.
    \item $\mathcal{R}(L(z)) \subseteq \mathcal{Y}$ so that we can consider the restriction $\tilde{F}(z) \coloneqq F(z)|_{\mathcal{R}(L(z))}$ for all $z \in \Omega$.
    \item $L$ and $M$ are conjugated by $E$ and $F$:
    \begin{equation} \label{eq:equivalenceM}
        M(z) = F(z)L(z)E(z), 
    \end{equation}
    for all $z \in \Omega$. \hfill $\lozenge$
\end{enumerate}
\end{definition}

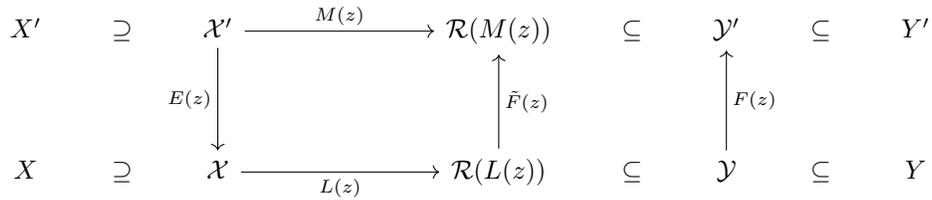
\begin{figure}[ht]
\centering
\[\begin{tikzcd}[column sep=scriptsize]
	{X'} & \supseteq & {\mathcal{X}'} &&&& {\mathcal{R}(M(z))} & \subseteq & {\mathcal{Y}'} & \subseteq & {Y'} \\
	\\
	X & \supseteq & {\mathcal{X}} &&&& {\mathcal{R}(L(z))} & \subseteq & {\mathcal{Y}} & \subseteq & Y
	\arrow["{M(z)}", from=1-3, to=1-7]
	\arrow["{E(z)}"', from=1-3, to=3-3]
	\arrow["{L(z)}"', from=3-3, to=3-7]
	\arrow["{\tilde{F}(z)}"', from=3-7, to=1-7]
	\arrow["{F(z)}"', from=3-9, to=1-9]
\end{tikzcd}\]
\caption{Relations between the spaces and operators involved in \cref{def:equi}.}
\label{fig:comm-diagram}
\end{figure}

Compared to the definition of equivalence for closed and bounded linear operators from \cite[Section I.1.2]{Kaashoek1992}, the above definition makes additional assumptions on the domains and ranges of the operators $E,F,L$ and $M$, see \cref{fig:comm-diagram} for a schematic overview. In particular, the first condition in \cref{def:equi} ensures that the domains and ranges of $E(z)$ and $F(z)$ are $z$-independent, which is essential for studying the holomorphicity of these operators and their inverses. In addition, this setup permits a well-defined analysis of the Jordan chains (\cref{def:Jordanchain2}) of $L$ and $M$ in \cref{prop:equivalent chains}. The second condition is necessary because $E$ and $F$ need not be closed operator-valued functions, for which this additional assumption would follow automatically, see \cite[Section I.2]{Gohberg1990} for further details. This condition also plays an important role in the proofs of \cref{prop:equivalent chains} and \cref{thm:spectralequal}. The third condition ensures that the conjugation in \eqref{eq:equivalenceM} is well-defined, and note for all $z \in \Omega$ that $\mathcal{R}(\tilde{F}(z)) = \mathcal{R}(M(z))$ since $E(z)$ maps onto $\mathcal{D}(L(z))$. As a consequence, $\tilde{F}(z) : \mathcal{R}(L(z)) \to \mathcal{R}(M(z))$ is a bijection for all $z \in \Omega$. It is also valuable to note that neither $E(z)$ and $F(z)$ nor their inverses $E(z)^{-1}$ and $F(z)^{-1}$ are assumed to be bounded for $z \in \Omega$. This additional assumption would limit our ability to construct the important upcoming equivalence from \cref{thm:rhoAcharac} in \cref{sec:construction}. For the sake of simplicity, we still call the closable linear operators $E(z)$ and $F(z)$ \emph{(algebraically) invertible}.

\begin{remark} \label{remark:equivalencetypes}
Equivalence between closed (and bounded) linear operators was first introduced by Gohberg et al. \cite{Gohberg1978}, see also \cite{Gohberg1990,Kaashoek2022} for more information and applications. More recently, Engstr\"om and Torshage introduced in \cite{Engstroem2017} a notion of equivalence of the form \eqref{eq:equivalenceM} between (holomorphic) closable linear operator-valued functions $L$ and $M$, but the values of $E$ and $F$ were assumed to be bounded, which is too restrictive for our upcoming construction. \hfill $\lozenge$
\end{remark}

Next, we introduce the notions of a characteristic value and a Jordan chain of a closable operator-valued function $L$. The following definition may be viewed as a natural extension of the classical concepts of eigenvalues and Jordan chains for the linear case $L(z)=zI-A$ with $A\in L(X)$, as presented in \cref{def:resolvent} and \cref{def:Jordanchain}, respectively. We also refer to \cite[Section I.1.2]{Kaashoek1992} for more information in the setting of closed linear operators.

\begin{definition} \label{def:Jordanchain2}
Let $L : \Omega \to L(X,Y)$ be a holomorphic closable operator-valued function with $\mathcal{X} = \mathcal{D}(L(z))$ for all $z \in \Omega$. A point $\sigma \in \Omega$ is called a \emph{characteristic value} of $L$ at $\sigma$ if there exists a nonzero \emph{eigenvector} $q_0 \in \mathcal{X}$ such that $L(\sigma)q_0 = 0$. An ordered set $\{q_0,\dots,q_{k-1}\}$ of vectors in $\mathcal{X}$ is called a \emph{Jordan chain} of $L$ at $\sigma$ if $q_0 \neq 0$ and
\begin{equation*}
    L(z)[q_0 + (z-\sigma)q_1 + \cdots + (z-\sigma)^{k-1}q_{k-1}] = \mathcal{O}((z-\sigma)^k).
\end{equation*}
The vectors $q_1,\dots,q_{k-1}$ are called \emph{generalized eigenvectors} of $L$ at $\sigma$. The number $k \in \mathbb{N}$ is called the \emph{length} of the chain. If there exists a maximal length of the chain starting with $q_0$, then we call this length the \emph{rank} of $q_0$ and the holomorphic function $\mathbb{C} \ni z \mapsto \sum_{l=0}^{k-1}(z-\sigma)^l q_l \in \mathcal{X}$ the \emph{root function} of $L$ corresponding to $\sigma$. If there is no maximal length, then we say that $q_0$ has \emph{infinite rank}. \hfill $\lozenge$
\end{definition}

Let us now introduce the notions of geometric and algebraic multiplicity of $L$ at a characteristic value $\sigma \in \Omega$. The linear subspace $\mathcal{N}(L(\sigma))$ of $X$ is called the \emph{eigenspace} of $L$ corresponding to $\sigma$, and its dimension $m_g(L(\sigma)) \in \mathbb{N} \cup \{ \infty \}$ is referred to as the \emph{geometric multiplicity} of $L$ at $\sigma$. To introduce the algebraic multiplicity $m_a(L(\sigma)) \in \mathbb{N} \cup \{ \infty \}$ of $L$ at $\sigma$, we proceed as follows: If $m_g(L(\sigma)) = \infty$, then we set $m_a(L(\sigma)) \coloneqq \infty$. If $p = m_g(L(\sigma)) < \infty$, we proceed similarly as in \cref{subsec:spectralprop}. Thus, we pick a basis $q_{1,0},\dots,q_{p,0}$ of $\mathcal{N}(L(\sigma))$ such that the ranks $k_j \in \mathbb{N} \cup \{\infty\}$ of the eigenvector $q_{j,0}$ are maximized sequentially. The (possibly infinite) integers $k_1 \leq \dots \leq k_p$, which are independent of the choice of basis, are called the \emph{zero multiplicities} of $L$ at $\sigma$. We define the largest zero multiplicity $k_p \eqqcolon k(L(\sigma)) \in \mathbb{N} \cup \{\infty\}$ to be the \emph{ascent} of $L$ at $\sigma$, and define $m_a(L(\sigma)) \coloneqq \sum_{j=1}^p k_j$ accordingly. If in addition $m_a(L(\sigma)) < \infty$, then $\sigma$ is called a \emph{characteristic value of finite type} of $L$. Following this procedure, we obtain a \emph{canonical system} of Jordan chains 
\begin{equation} \label{eq:Jordanchainq}
    q_{1,0},\dots,q_{1,k_1 - 1},\dots,q_{p,0},\dots,q_{p,k_p - 1}
\end{equation}
of $L$ at $\sigma$. Again, one can see \eqref{eq:Jordanchainq} as an extension of the classical notion of a canonical system of (generalized) eigenvectors for the linear case $L(z) = zI-A$ with $A \in L(X)$ presented in \eqref{eq:canonicalbasis}.

\begin{remark}
The attentive reader may have observed that our definition of $m_a(L(\sigma))$ differs in nature from that of $m_a(\sigma,A)$. In the classical linear case $L(z) = zI-A$ with $A \in L(X)$, recall from \eqref{eq:EsigmaQsigma} that the generalized eigenspace $E_\sigma(A)$ coincides with $\mathcal{N}((\sigma I - A)^{k(\sigma,A)})$ if $\sigma$ is of finite type. However, in the general nonlinear setting of $L$, the linear span of the Jordan chains of $L$ at $\sigma$ cannot in general be characterized as the kernel of powers of a single operator. We refer to \cite[Section 5.1.1]{Kaashoek2022} for a simple counterexample on $L \colon \mathbb{C} \to \mathcal{L}(\mathbb{C}^2)$ showing that $\dim \mathcal{N}(L(\sigma)^l) \neq m_a(L(\sigma))$ for all $l \geq 1$. \hfill $\lozenge$
\end{remark}

The next result is an extension of \cite[Proposition 1.2]{Kaashoek1992} from closed to closable linear operators.

\begin{proposition} \label{prop:equivalent chains}
If two holomorphic closable operator-valued functions $L$ and $M$ are equivalent on $\Omega$, then there is a one-to-one correspondence between their Jordan chains. In particular, $m_g(L(\sigma)) = m_g(M(\sigma)), m_a(L(\sigma)) = m_a(M(\sigma))$ and $k(L(\sigma)) = k(M(\sigma))$ for a characteristic value $\sigma \in \Omega$.
\end{proposition}
\begin{proof}
Let $\{q_0,\dots,q_{k-1}\} \subseteq \mathcal{X}$ be a Jordan chain of $L$ at $\sigma$. Since $E(\cdot)^{-1}$ is holomorphic by assumption (\cref{def:equi}), and any root function is holomorphic as well, it is clear that the map $\Omega \ni z \mapsto E(z)^{-1} \sum_{l=0}^{k-1}(z-\sigma)^lq_l \in X'$ is holomorphic, and thus analytic on $\Omega$. Hence there exist vectors $(q_l')_l$ in $\mathcal{X}'$ such that $E(z)^{-1} \sum_{l=0}^{k-1}(z-\sigma)^lq_l = \sum_{l=0}^\infty (z-\sigma)^l q_l'$ for all $z \in \Omega$. Then $q_0' \neq 0$ and a quick calculation using \eqref{eq:equivalenceM} shows that $M(z)\sum_{l=0}^{k-1}(z-\sigma)^lq_l' = \mathcal{O}((z-\sigma)^k)$, which proves that $\{q_0',\dots,q_{k-1}'\}$ is a Jordan chain of $M$ at $\sigma$. The other direction can be proven similarly by using the fact that $E$ is holomorphic and that the conjugation $L(z) = F(z)^{-1} M(z) E(z)^{-1}$ holds due to \cref{def:equi}. We conclude that there is a one-to-one correspondence between the Jordan chains of $L$ and $M$. 

To prove the remaining claim, let us observe from \eqref{eq:equivalenceM} that $E(\sigma)$ maps $\mathcal{N}(M(\sigma))$ in a one-to-one way onto $\mathcal{N}(L(\sigma))$, which proves that $m_g(L(\sigma)) = m_g(M(\sigma))$. Since there is a one-to-one correspondence between the Jordan chains of $L$ and $M$, the zero multiplicities of $L$ and $M$ coincide and thus $m_a(L(\sigma)) = m_a(M(\sigma))$ and $k(L(\sigma)) = k(M(\sigma))$, which proves the claim.
\end{proof}

We emphasize that the concepts introduced earlier in this subsection do not require the underlying complex Banach spaces $X$ and $Y$ to coincide. We now turn to the introduction of the resolvent set $\rho(L)$ of $L$ as an extension of \cref{def:resolvent}, where we assume that $X$ and $Y$ do coincide.

\begin{definition} \label{def:spectrum_characop} 
Let $L: \Omega \to L(Y)$ be a holomorphic closable operator-valued function and fix $z \in \Omega$. We say that $L(z)$ is \emph{invertible} if $0 \in \rho(L(z))$, i.e. $L(z)$ is injective, has dense range $\mathcal{R}(L(z))$ in $Y$, and the algebraic inverse $L(z)^{-1}: \mathcal{R}(L(z)) \to \mathcal{D}(L(z))$ is a bounded linear operator. We subsequently write $\rho(L) \coloneqq \{ z \in \Omega : L(z) \mbox{ is invertible} \}$ for the \emph{resolvent set} of $L$. \hfill $\lozenge$
\end{definition}

For the remainder of this section, we fix a complex Banach space $Y$ and a holomorphic closable operator-valued function $L : \Omega \to L(Y)$. In the spirit of the result established in \cref{subsec:spectralprop}, we aim to introduce the notion of the order of a pole at an isolated point $\sigma \notin \rho(L)$. Recall that in \ref{hyp:SH1}, we imposed a condition on $\mathcal{D}(A)$ that allowed us to conclude in \cref{prop:resolventproperties} that the range of $zI-A$ is independent of $z \in \rho(A)$. We would now like to establish a similar property for the range of the operator $L(z)$. However, since $L$ is not necessarily of the form $L(z) = zI-A$ for some $A \in L(X)$, a condition on the domain $\mathcal{D}(L(z))$ does not directly imply that the range $\mathcal{R}(L(z))$ is $z$-independent for $z \in \rho(L)$. To bypass this issue, we introduce the following (spectral) hypothesis:

\begin{enumerate}[label=({SH}{{\arabic*}})] 
\setcounter{enumi}{1}
    \item \label{hyp:SH2} 
    The range $\mathcal{R}(L(z))$ is $z$-independent for all $z \in \rho(L)$.
\end{enumerate}
To study the order of a pole of $L$ at an isolated point $\sigma \notin \rho(L)$, recall from the linear case (\cref{lemma:Dunford}) that holomorphicity of $L(\cdot)^{-1} : \rho(L) \to L(Y)$ would be helpful when $\rho(L)$ is open. Under these additional assumptions, it is clear that we can expand for every $\varphi \in \mathcal{R}(L(z))$ the map $L(\cdot)^{-1} \varphi$ at $\sigma$ locally as a Laurent series
\begin{equation} \label{eq:Dunford2}
    L(z)^{-1} \varphi = \sum_{l = -\infty}^{\infty} (z - \sigma)^l L_l(\sigma) \varphi, \quad L_l(\sigma)\varphi \coloneqq \frac{1}{2 \pi i} \oint_{\Gamma_{\sigma}} (z-\sigma)^{-(l+1)} L(z)^{-1} \varphi dz,
\end{equation}
whenever $0 < |z-\sigma| < \delta$ for some sufficiently small $\delta$. Note that the Dunford integral from \eqref{eq:Dunford2} is well-defined due to \ref{hyp:SH2}. However, note that $L_l(\sigma)$ might map back into $\mathcal{D}(\overline{L(z)})$ instead of $\mathcal{D}(L(z))$. If there exists a $r \in \mathbb{N}$ such that $L_{-r}(\sigma) \neq 0$ while $L_{-l}(\sigma) = 0$ for all $l > r$, then $\sigma$ is called a \emph{pole} of $L$ of \emph{order} $r \eqqcolon r(L(\sigma))$, and we set $r(L(\sigma)) = \infty$ if $L(\cdot)^{-1}$ has an essential singularity at $\sigma$.

\subsection{Characteristic operators and spectral properties} \label{subsec:char operators}
Throughout this section, we let $X$ and $Y$ be complex Banach spaces and $\Omega \subseteq \mathbb{C}$ be a nonempty open set. The following definition is a cornerstone of this paper and extends the definition of a characteristic matrix for a closed linear operator from \cite[Section I.2]{Kaashoek1992} to a characteristic operator for a closable linear operator.

\begin{definition} \label{def:characoperator}
Consider a closable linear operator $A : \mathcal{D}(A) \subseteq X \to X$ and a holomorphic closable operator-valued function $\Delta : \Omega \to L(Y)$. The operator-valued function $\Delta$ is called a \emph{characteristic operator} for $A$ on $\Omega$ if there exists a complex Banach space $Z$ along with dense linear subspaces $\tilde{X} \subseteq X$ and $\tilde{Z} \subseteq Z$ such that the holomorphic closable operator-valued functions 
\begin{equation*}
\Omega \ni z \mapsto \begin{pmatrix}
        \Delta(z) & 0 \\
        0 & I_{\tilde{X}}
        \end{pmatrix} \in L(Y \oplus X) \qquad \mbox{and} \qquad \Omega \ni z \mapsto \begin{pmatrix}
            I_{\tilde{Z}}& 0 \\
            0 & zI - A
        \end{pmatrix} \in L(Z \oplus X_A, Z \oplus X)
\end{equation*}
are equivalent on $\Omega$. \hfill $\lozenge$
\end{definition}

Recalling the definition of equivalence (\cref{def:equi}), the above definition guarantees that $\mathcal{D}(\Delta(z)) = \mathcal{
Y}$ is $z$-independent for all $z \in \Omega$, and there exist holomorphic closable operator-valued functions $E: \Omega \to L(Y \oplus X,Z \oplus X_A)$ and $F: \Omega \to L(Z \oplus X,Y \oplus X)$ such that $E(z) : \mathcal{Y} \oplus \tilde{X} \to \tilde{Z} \oplus \mathcal{D}(A)$ and $\tilde{F}(z) = F(z)|_{\tilde{Z} \oplus \mathcal{R}(zI-A)} : \tilde{Z} \oplus \mathcal{R}(zI-A) \to \mathcal{R}(\Delta(z)) \oplus \tilde{X}$ are bijective, and there holds
    \begin{equation} \label{eq:equivalencematrix}
        \begin{pmatrix}
            \Delta(z) & 0 \\
            0 & I_{\tilde{X}}
        \end{pmatrix}
        =
        F(z)
        \begin{pmatrix}
            I_{\tilde{Z}}& 0 \\
            0 & zI - A
        \end{pmatrix}
        E(z), \quad \forall z \in \Omega.
\end{equation}
The operator on the left-hand side of \eqref{eq:equivalencematrix} is called an $(\tilde{X},X)$\emph{-extension} of $\Delta(z)$ and the operator in the middle on the right-hand side of \eqref{eq:equivalencematrix} is called an $(\tilde{Z},Z)$\emph{-extension} of $zI-A$.

\begin{remark}
The necessity of working with dense linear subspaces $\tilde{X} \subseteq X$ and $\tilde{Z} \subseteq Z$ in the sequel stems from the dense range requirement in the definition of the spectrum of $A$ given in \cref{def:resolvent}. When $A$ is closed, recall from \cref{remark:definitionresolvent} that $\mathcal{R}(zI-A) = X$ for every $z \in \rho(A)$. In this situation, the $(\tilde{X},X)$- and $(\tilde{Z},Z)$-extensions simply become the usual $X$- and $Z$-extensions, corresponding to the standard notion of \emph{equivalence after extension} for closed linear operators, see \cite[Section 4.4]{Harm2008}. \hfill $\lozenge$
\end{remark}

Our next goal is to analyse the spectral properties of $A$ through the characteristic operator $\Delta$. As shown in \cref{thm:spectralequal} below, the relation \eqref{eq:equivalencematrix} already implies a connection between the point spectrum of $A$ and the characteristic values of $\Delta$. However, to fully describe the spectrum of $A$ in terms of $\Delta$, it is necessary to relate $R(\cdot, A)$ with $\Delta(\cdot)^{-1}$. The relation \eqref{eq:equivalencematrix} alone does not suffice for this purpose, primarily because the operators $E$ and $F$ may not be bounded or defined on the entire spaces $Y \oplus X$ and $Z \oplus X$, respectively. To overcome this issue, we formulate the following (spectral) hypothesis:

\begin{enumerate}[label=({SH}{{\arabic*}})]
\setcounter{enumi}{2}
    \item \label{hyp:SH3} The values of the operator-valued functions $G_\Delta : \rho(\Delta) \to L(X)$ and $G_{A} : \rho(A) \cap \Omega \to L(Y)$ defined by
    \begin{equation*}
        G_\Delta(z) \coloneqq \bigg[E(z)
        \begin{pmatrix}
            \Delta(z)^{-1} & 0 \\
            0 & I_{\tilde{X}}
        \end{pmatrix}
        F(z) \bigg]_{22}, \quad G_A(z) \coloneqq \bigg[E(z)^{-1}
        \begin{pmatrix}
            I_{\tilde{Z}} & 0 \\
            0 & R(z,A)
        \end{pmatrix}
        F(z)^{-1} \bigg]_{11},
    \end{equation*}
are bounded, and their maximal domains lie dense in $X$ and $Y$, respectively.
\end{enumerate}

\begin{theorem} \label{thm:spectralequal}
Let $\Delta$ be a characteristic operator for $A$ on $\Omega$. Then the following statements hold:
\begin{enumerate}
    \item The point spectrum of $A$ in $\Omega$ satisfies
    \begin{equation} \label{eq:pointspectrumsigmaADelta}
        \sigma_p(A) \cap \Omega = \{ \sigma \in \Omega : \sigma \mbox{ is a characteristic value of } \Delta \}.
    \end{equation}
    Furthermore, if $\sigma \in \sigma_p(A) \cap \Omega$, then $m_g(\sigma,A) = m_g(\Delta(\sigma)), m_a(\sigma,A) = m_a(\Delta(\sigma))$ and $k(\sigma,A) = k(\Delta(\sigma))$. If in addition $\sigma$ is of finite type, then the partial multiplicities of $A$ at $\sigma$ are equal to the zero multiplicities of $\Delta$ at $\sigma$.

    \item If $\rho(\Delta) \neq \emptyset$ and \ref{hyp:SH3} holds, then the spectrum of $A$ in $\Omega$ satisfies
    \begin{equation} \label{eq:spectrumsigmaADelta}
        \sigma(A) \cap \Omega = \{ \sigma \in \Omega : \Delta(\sigma) \mbox{ is not invertible} \}.
    \end{equation}
    Furthermore, if $A$ satisfies \ref{hyp:SH1} and $\Delta$ satisfies \ref{hyp:SH2}, and $\sigma \in \sigma(A) \cap \Omega$ is isolated, then $r(\sigma,A) = r(\Delta(\sigma))$. If in addition $\sigma \in \sigma_p(A) \cap \Omega$ is of finite type and $k(\sigma,A) = r(\Delta(\sigma))$, then the multiplicity theorem holds:
    \begin{equation} \label{eq:multiplicitythm}
        \dim \mathcal{N}((\sigma I - A)^{r(\Delta(\sigma))}) = m_a(\Delta(\sigma)).
    \end{equation}
\end{enumerate}
\end{theorem}

\begin{proof}
We start by proving the first statement. To prove \eqref{eq:pointspectrumsigmaADelta}, we show that there is a one-to-one correspondence between kernel vectors of $z I-A$ and kernel vectors of $\Delta(z)$ for $ z \in \Omega$. Let us write
\begin{equation*}
    M(z) = \begin{pmatrix}
            \Delta(z) & 0 \\
            0 & I_{\tilde{X}}
        \end{pmatrix}, \quad
    L(z) = \begin{pmatrix}
            I_{\tilde{Z}}& 0 \\
            0 & zI - A
        \end{pmatrix}, \quad \forall z \in \Omega,
\end{equation*}
and note that $\mathcal{D}(M(z)) = \mathcal{Y} \oplus \tilde{X}$ and $\mathcal{D}(L(z)) = \tilde{Z} \oplus \mathcal{D}(A)$ are $z$-independent. According to the proof of \cref{prop:equivalent chains}, $E(z)$ maps $\mathcal{N}(M(z)) = \mathcal{N}(\Delta(z)) \oplus \{0\}$ in a one-to-one way onto $\mathcal{N}(L(z)) = \{0\} \oplus \mathcal{N}(z I-A)$. Hence, the closable holomorphic operator-valued function $N : \Omega \to L(Y,X)$ with values $N(z) : \mathcal{Y} \to \mathcal{D}(A)$ defined by
\begin{equation*}
    N(z) q \coloneqq (0,I_{\mathcal{D}(A)}) E(z) 
    \begin{pmatrix}
        q \\
        0
    \end{pmatrix},
\end{equation*}
maps $\mathcal{N} (\Delta(z))$ in a one-to-one way onto $\mathcal{N}(z I-A)$, which proves the claim. 

Let us now assume that $\sigma \in \sigma_p(A) \cap \Omega$. By the previous argument, there holds in particular $m_g(\Delta(\sigma)) = m_g(\sigma,A)$. To prove the claim regarding the algebraic multiplicity and ascent, we first observe from the definition of $L$ that $m_a(\sigma,A) = m_a(L(\sigma))$ and $k(\sigma,A) = k(L(\sigma))$. \cref{prop:equivalent chains} tells us that $m_a(L(\sigma)) = m_a(M(\sigma))$ and $k(L(\sigma)) = k(M(\sigma))$. From the definition of $M$ we observe that $m_a(M(\sigma)) = m_a(\Delta(\sigma))$ and $k(M(\sigma)) = k(\Delta(\sigma))$, which proves the claim.

Assume now in addition that $\sigma$ is of finite type. Let $q_{1,0},\dots,q_{1,k_1 - 1},\dots,q_{p,0},\dots,q_{p,k_p - 1}$, with $k_1 \leq \dots \leq k_p$, be a canonical system of Jordan chains for $\Delta$ at $\sigma$. Consider the function $Q_i : \mathbb{C} \to \mathcal{Y}$ defined by $Q_i(z) \coloneqq \sum_{l=0}^{k_i - 1} (z-\sigma)^l q_{i,l}$ and recall from \Cref{def:Jordanchain2} that this map is a root function of $\Delta$ corresponding to $\sigma$, that is
\begin{equation} \label{eq:Delta(z)1BigO}
    \Delta(z) Q_i(z) = \mathcal{O}((z-\sigma)^{k_i}), \quad \forall i =1,\dots,p.
\end{equation}
Since $N$ and $Q_i$ are holomorphic on $\Omega$, the map $\Phi_i : \Omega \to \mathcal{D}(A)$ defined by $\Phi_i(z) \coloneqq N(z) Q_i(z)$ is holomorphic as well. As a consequence, this map is analytic and thus there exist vectors $\varphi_{i,0},\dots,\varphi_{i,k_{i-1}} \in \mathcal{D}(A)$ such that $\Phi_i(z) = \sum_{l=0}^{k_i - 1} (z-\sigma)^l \varphi_{i,l} + \mathcal{O}((z-\sigma)^{k_i})$. The equivalence \eqref{eq:equivalencematrix} together with \eqref{eq:Delta(z)1BigO} proves that $(zI - A)\Phi_i(z) = \mathcal{O}((z-\sigma)^{k_i})$ and thus
\begin{equation} \label{eq:Amu}
    (A-\sigma I)\varphi_{i,0} = 0, \dots,(A-\sigma I)\varphi_{i,k_i - 1} = \varphi_{i,k_i - 2}.
\end{equation}
Since $N(\sigma)$ maps $\mathcal{N} (\Delta(\sigma))$ in a one-to-one way onto $\mathcal{N}(\sigma I-A)$, it follows that $\varphi_{1,0},\dots,\varphi_{p,0}$ are linearly independent. Then $\varphi_{1,0},\dots,\varphi_{1,k_1 - 1},\dots,\varphi_{p,0},\dots,\varphi_{p,k_p - 1}$ are all linearly independent due to \eqref{eq:Amu}, which proves that $\varphi_{1,0},\dots,\varphi_{1,k_1 - 1},\dots,\varphi_{p,0},\dots,\varphi_{p,k_p - 1}$ is a canonical basis of (generalized) eigenvectors of $A$ at $\sigma$, which proves the claim.

Let us now prove the second statement. Therefore, assume that $\rho(\Delta) \neq \emptyset$ and \ref{hyp:SH3} holds. To prove \eqref{eq:spectrumsigmaADelta}, we show that $\rho(A) \cap \Omega = \rho(\Delta)$. 
Fix $z \in \rho(\Delta) \neq \emptyset$. By the previous argument, $zI - A$ is also injective and so we can consider the algebraic inverse of $zI - A: \mathcal{D}(A) \to \mathcal{R}(zI - A)$. Although we have not proven at this point that $\mathcal{R}(zI - A)$ is dense in $X$, we will abuse notation and call $R(z, A)$ for a moment the algebraic inverse of $zI-A$. The equivalence \eqref{eq:equivalencematrix} then implies that 
\begin{equation} \label{eq:invequivalence}
     \begin{pmatrix}
         I_{\tilde{Z}}& 0 \\
         0 & R(z,A)
     \end{pmatrix}
     =
     E(z)
     \begin{pmatrix}
         \Delta(z)^{-1} & 0 \\
         0 & I_{\tilde{X}}
     \end{pmatrix}
     F(z).
\end{equation}
But from here we see that $R(z,A) = G_\Delta(z)$ and hence it follows from \ref{hyp:SH3} that indeed the domain $\mathcal{D}(R(z, A)) = \mathcal{R}(zI - A)$ is dense in $X$ and that $R(z, A)$ is bounded. This proves that $z \in \rho(A) \cap \Omega$ and thus $\rho(\Delta) \subseteq \rho(A) \cap \Omega$. 
To prove the other inclusion, let $z \in \rho(A) \cap \Omega$ be given. Equation \eqref{eq:invequivalence} tells us that $\Delta(z)^{-1} = G_A(z)$ and due to \ref{hyp:SH3}, $\Delta(z)^{-1}$ is bounded and $\mathcal{R}(\Delta(z))$ is dense in $Y$. We conclude that $z \in \rho(\Delta)$, which proves the claim.

Assume now in addition that $\sigma \in \sigma(A) \cap \Omega$ is isolated and that \ref{hyp:SH1} and \ref{hyp:SH2} hold. Since $\rho(A)$ and $\Omega$ are open in $\mathbb{C}$ (\cref{prop:resolventopen}), it follows from the proof of the previous statement that $\rho(\Delta) = \rho(A) \cap \Omega$ is open in $\mathbb{C}$. As $\Delta(\cdot)^{-1}= G_A$ and the operator-valued functions $E(\cdot)^{-1}, F(\cdot)^{-1}$ and $R(\cdot,A)$ are holomorphic on $\rho(A) \cap \Omega$ due to \cref{def:equi} and \cref{prop:resolventopen}, it clear that $\Delta(\cdot)^{-1}$ is holomorphic on $\rho(\Delta)$. Since $E$ and $F$ are holomorphic on $\Omega$, we can use the Laurent series expansions \eqref{eq:Dunford} and \eqref{eq:Dunford2} to see from \eqref{eq:invequivalence} that $r(\Delta(\sigma)) = r(\sigma,A)$, which proves the claim.

Finally, if we assume in addition that $\sigma$ is an eigenvalue of $A$ in $\Omega$ of finite type and $k(\sigma,A) = r(\Delta(\sigma))$, then by the first part in the proof we obtain $\mathcal{N}((\sigma I - A)^{r(\Delta(\sigma))}) = \mathcal{N}((\sigma I - A)^{k(\sigma,A)})$. Since this last linear subspace equals $E_\sigma(A)$, the dimension of $\mathcal{N}((\sigma I - A)^{r(\Delta(\sigma))})$ equals that of $m_a(\sigma,A) = m_a(\Delta(\sigma))$, where this last equality follows from the first part of the proof.
\end{proof}
 
It turns out in \cref{sec:applications} that for some periodic evolutionary systems, the characteristic operator $\Delta$ has additional structure, such as closedness and compactness. In this case, the equivalence between $zI - A$ and $\Delta$ allows us to conclude that the spectrum of $A$ is discrete, as we show in the next lemma. Therefore, we call $\Delta$ a \emph{closed characteristic operator} for $A$ on $\Omega$ if $\Delta$ is a characteristic operator for $A$ on $\Omega$ and $\Delta(z)$ is a closed linear operator for all $z \in \Omega$. For clarity, we say that a closed linear operator $\Delta(z)$ with nonempty resolvent set has \emph{compact resolvent} if $R(\mu,\Delta(z))$ is compact for some (and hence for every) $\mu \in \rho(\Delta(z))$, recall \cite[Definition II.4.24]{Engel2000}.

\begin{lemma} \label{lemma:compactspectra}
Let $\Delta$ be a closed characteristic operator for $A$ on $\Omega$ satisfying $\rho(\Delta) \neq \emptyset$ and \ref{hyp:SH3}. If $\rho(\Delta(z)) \neq \emptyset$ and 
$\Delta(z)$ has compact resolvent for all $z \in \Omega$, 
then every spectral value of $A$ in $\Omega$ is an isolated eigenvalue of finite type and there holds
\begin{equation} \label{eq:spectracompact}
    \sigma(A) \cap \Omega = \{ \sigma \in \Omega :  \sigma \text{ is a characteristic value of } \Delta \}.
\end{equation}
\end{lemma}
\begin{proof} Let $\sigma \in \sigma(A) \cap \Omega$ be given and note that \cref{thm:spectralequal} applies. Hence, \eqref{eq:spectrumsigmaADelta} tells us that $\Delta(\sigma)$ is not invertible and so $0 \in \sigma(\Delta(\sigma))$. Due to the assumptions, it follows from \cite[Corollary IV.1.19]{Engel2000} that every spectral value of $\Delta(\sigma)$ is an isolated eigenvalue of finite type. In particular, $0$ is an isolated eigenvalue of finite type of $\Delta(\sigma)$. Since in addition $\Delta$ is holomorphic on $\Omega$ and $\rho(\Delta) \neq \emptyset$, the analytic Fredholm theorem \cite[Corollary XI.8.4]{Gohberg1990} ensures that $\sigma$ is an isolated characteristic value of finite type of $\Delta$. This proves \eqref{eq:spectracompact} as the other inclusion is clear. Since $\sigma$ is an isolated characteristic value of finite type of $\Delta$, it follows from \cref{thm:spectralequal} that $\sigma$ is an isolated eigenvalue of finite type of $A$.
\end{proof}

\begin{remark} \label{remark:Yfinitedim}
If the Banach space $Y$ is finite-dimensional and the linear operator $A$ is closed, then the results from this section simplify to the framework of characteristic matrices developed by Kaashoek and Verduyn Lunel in \cite{Kaashoek1992,Kaashoek2022}. In particular, the spaces $\tilde{X} = X$ and $\tilde{Z} = Z$, and the operators $E$ and $F$ become everywhere-defined holomorphic bounded operator-valued functions whose values are bijective so that automatically $E(\cdot)^{-1}$ and $F(\cdot)^{-1}$ are holomorphic. As a consequence, the (spectral) hypotheses \ref{hyp:SH1}, \ref{hyp:SH2} and \ref{hyp:SH3} are then automatically satisfied. Moreover, the identity $k(\sigma,A) = r(\Delta(\sigma))$ holds automatically for any isolated point $\sigma \in \sigma(A) \cap \Omega$, ensuring that the multiplicity theorem \eqref{eq:multiplicitythm} is always valid. In addition, due to the presence of a determinant in finite-dimensional vector spaces, the map $\det \Delta(\cdot) : \Omega \to \mathbb{C}$ is holomorphic and so all spectral values of $A$ are automatically isolated eigenvalues of finite type \cite[Theorem I.2.1]{Kaashoek1992}. Therefore, one can see \cref{lemma:compactspectra} as an extension of this result towards the infinite-dimensional setting of $Y$ by adding a compactness assumption. In the finite-dimensional setting of $Y$, the algebraic multiplicity $m_a(\sigma,A)$ also equals the order of $\sigma$ as a root of $\det \Delta(\cdot)$, see \cite[Theorem I.2.1]{Kaashoek1992}. To prove this result, one combines the theory of characteristic matrices with the theory of local Smith forms \cite{Kaashoek1992,Kaashoek2022,Diekmann1995}. \hfill $\lozenge$
\end{remark}

\section{General scheme for constructing characteristic operators} \label{sec:construction}
The main aim of this section is to develop a general scheme for constructing  characteristic operators $\Delta$ for a broad class of closable linear operators $\mathcal{A}$ that arise in many periodic evolutionary systems. This rather technical construction will be carried out in \cref{subsec:construction} and \cref{subsec:equivalencescheme}. In \cref{subsec:periodicspectral} we study periodic spectral properties of characteristic operators as this allows us, under certain conditions, to reduce the spectral problem for $\mathcal{A}$ on $\Omega$ in terms of $\Delta$ even further.

\subsection{Construction of $\Delta$ in terms of $\mathcal{A}$} \label{subsec:construction}

Let $X$ and $Y$ be complex Banach spaces. In the applications presented in \cref{sec:applications}, the space $X$ will serve as the \emph{state space} of the periodic evolutionary system of interest, whereas $Y$ will represent its associated \emph{reduction space}. For $T \geq 0$ and $Z \in \{X,Y\}$, we consider a complex Banach space 
$\mathcal{F}_T(\mathbb{R},Z)$ consisting of $Z$-valued $T$-periodic functions on $\mathbb{R}$ belonging to a specified function class $\mathcal{F}$. The case of $T=0$ will be important when we show that the class of autonomous evolutionary systems is included in our construction, see \cref{remark:autonomousDDE} for more information. The general idea of the construction of a characteristic operator $\Delta$ is to reduce a rather difficult spectral problem for a closable linear operator $\mathcal{A} \in L(\mathcal{F}_T(\mathbb{R},X))$ towards a much simpler problem formulated in terms of a closable operator-valued function $\Delta : \Omega \to L(\mathcal{F}_T(\mathbb{R},Y))$. For example, in the setting of periodic linear classical DDEs (\cref{subsec:classicalDDEs}), we choose the state space $X = C([-h,0], \mathbb{C}^n)$ and the reduction space $Y = \mathbb{C}^n$. This shows that the spectral problem reduces from a problem in 
$C_T(\mathbb{R}, C([-h,0], \mathbb{C}^n))$ to a simpler one in $C_T(\mathbb{R}, \mathbb{C}^n)$. To do so, we first construct a characteristic operator $\Delta$ associated with a broad class of closable linear operators $\mathcal{A}$ of the form \eqref{eq:D(A)rep}, culminating in \cref{thm:rhoAcharac}. This construction allows us to invoke \cref{thm:spectralequal} for the resulting equivalence, thereby establishing a connection between the spectral properties of $\mathcal{A}$ and $\Delta$, as detailed in \cref{cor:spectralrelations}. In \cref{sec:applications}, we illustrate how this reduction greatly facilitates the spectral analysis of a wide class of periodic evolutionary systems.

For the construction, consider a closable linear operator $D : \mathcal{D}(D) \subseteq \mathcal{F}_T(\mathbb{R},X) \to \mathcal{F}_T(\mathbb{R},X)$ satisfying the following conditions:
\begin{enumerate}[label=({H}{{\arabic*}})]
\setcounter{enumi}{0}
    \item \label{hyp:iota} There exists a linear subspace $\mathcal{Y} \subseteq \mathcal{F}_T(\mathbb{R}, Y)$ and a bounded bijective linear map $\iota: \mathcal{Y} \to \mathcal{N}(D) \neq \{0\}$ that has bounded (algebraic) inverse.
    \item \label{hyp:restriction} The operator $D$ has a restriction $D_0 : \mathcal{D}(D_0) \subseteq \mathcal{F}_T(\mathbb{R},X) \to \mathcal{F}_T(\mathbb{R},X)$ with bounded (algebraic) inverse such that $\Omega \coloneqq \rho(D_0) \neq \emptyset$ and 
    \begin{equation} \label{eq:D0_D}
    \mathcal{D}(D) = \mathcal{N} (D) \oplus \mathcal{D}(D_0).
    \end{equation}
    \item \label{hyp:domainrangeD} The domain $\mathcal{D}(D) \subseteq \mathcal{R}(zI-D_0)$ for all $z \in \Omega$.
\end{enumerate}
Note that $D_0$ is closable since it admits $\overline{D}$ as a closed linear extension. Moreover, \ref{hyp:restriction} and \ref{hyp:domainrangeD} guarantee that $\mathcal{D}(D_0) \subseteq \mathcal{R}(zI-D_0)$ for all $z \in \Omega$, which shows that $D_0$ satisfies \ref{hyp:SH1}. The following result will be helpful in the upcoming construction.
\begin{lemma}  \label{lemma:rangeinclusion}
    The following statements hold:
    \begin{enumerate}
        \item $\mathcal{R}(z_1 I-D_0) = \mathcal{R}(z_2 I - D_0)$ for all $z_1,z_2 \in \Omega$.
        \item $\mathcal{R}(zI-D_0) = \mathcal{R}(zI - D)$ for all $z \in \Omega \cup \{0\}$.
    \end{enumerate}
\end{lemma}
\begin{proof}
The first statement follows directly from \cref{prop:resolventproperties} since \ref{hyp:SH1} holds for $D_0$.

To prove the second statement, we first recall that $D_0$ is a restriction of $D$, from where it follows that $\mathcal{R}(zI - D_0) \subseteq \mathcal{R}(zI - D)$ for all $z \in \Omega \cup \{0 \}$.
To prove the converse inclusion, fix $z \in \Omega \cup \{0\}$ and $\psi \in \mathcal{R}(zI-D)$. Then there exists a $\varphi \in \mathcal{D}(D)$ such that $(zI-D) \varphi = \psi$. According to \ref{hyp:restriction}, we can write $\varphi = \tilde{\varphi} + \varphi_0$, where $(\tilde{\varphi},\varphi_0) \in \mathcal{N}(D) \oplus \mathcal{D}(D_0)$, and therefore $\psi = (zI - D)(\tilde{\varphi} + \varphi_0) = z\tilde{\varphi} + (zI-D_0)\varphi_0$. If $z \neq 0$, recall from \ref{hyp:domainrangeD} that $\tilde{\varphi} \in \mathcal{N}(D) \subseteq \mathcal{R}(zI-D_0)$, and so it is clear that $\psi \in \mathcal{R}(zI-D_0)$. If $z=0$, then $\psi = -D_0\varphi_0$ and thus $\psi \in \mathcal{R}(-D_0)$, which proves the claim.
\end{proof}

Before we proceed, let us remark that $0 \notin \Omega$ only occurs when $\mathcal{R}(D_0)$ is not dense in $\mathcal{F}_T(\mathbb{R},X)$. Nonetheless, this potential technicality does not interfere with the upcoming construction and, as we will see, $0 \in \Omega$ for all the examples provided in \cref{sec:applications}.

To develop the general scheme for characteristic operators, we introduce two linear operators
\begin{equation*}
    K : \mathcal{D}(K) \subseteq \mathcal{F}_T(\mathbb{R},X) \to \mathcal{F}_T(\mathbb{R},Y), \quad M : \mathcal{F}_T(\mathbb{R},X) \to \mathcal{F}_T(\mathbb{R},Y),
\end{equation*}
where $K$ is closable and $M$ is bounded. We additionally make the following assumptions on $K$:
\begin{enumerate}[label=({H}{{\arabic*}})]
\setcounter{enumi}{3}
    \item \label{hyp:K} The domain $\mathcal{D}(K) \supseteq \mathcal{D}(D)$ and the restriction $K|_{\mathcal{D}(D_0)} : \mathcal{D}(D_0) \subseteq \mathcal{F}_T(\mathbb{R},X) \to \mathcal{F}_T(\mathbb{R},Y)$ is bounded. 
\end{enumerate}
To $D,K$ and $M$, we associate two linear operators $\mathcal{A} : \mathcal{D}(\mathcal{A}) \to \mathcal{F}_T(\mathbb{R},X)$ and $\hat{\mathcal{A}} : \mathcal{D}(\hat{\mathcal{A}}) \to \hat{\mathcal{F}}_T(\mathbb{R},X)$ defined by
\begin{equation} \label{eq:D(A)rep}
    \mathcal{D}(\mathcal{A}) \coloneqq \{ \varphi \in \mathcal{F}_T(\mathbb{R},X) : \varphi \in \mathcal{D}(D), \ MD\varphi = K \varphi \}, \quad \mathcal{A}\varphi \coloneqq D\varphi,
\end{equation}
and
\begin{equation} \label{eq:D(Ahat)rep}
    \mathcal{D}(\hat{\mathcal{A}}) \coloneqq \bigg \{ \begin{pmatrix}
    q\\
    \varphi
    \end{pmatrix}
    \in \hat{\mathcal{F}}_T(\mathbb{R},X) : \varphi \in \mathcal{D}(D),\ q = M\varphi \bigg \}, \quad \hat{\mathcal{A}}
    \begin{pmatrix}
    q\\
    \varphi
    \end{pmatrix}
    \coloneqq
    \begin{pmatrix}
    K\varphi \\
    D\varphi
    \end{pmatrix},
\end{equation}
where $\hat{\mathcal{F}}_T(\mathbb{R},X) \coloneqq {\mathcal{F}}_T(\mathbb{R},Y) \oplus {\mathcal{F}}_T(\mathbb{R},X)$. 
We will refer to $\mathcal{A}$ and $\hat{\mathcal{A}}$ as the \emph{first} and \emph{second operator} associated with $D,K$ and $M$, respectively. Moreover, we point out that in the context of \eqref{eq:D(A)rep}, we can think of $K$ and $M$ as generalized boundary-value operators. 

\begin{lemma} \label{lemma:D0closable}
    The linear operators $\mathcal{A}$ and $\hat{\mathcal{A}}$ are closable.
\end{lemma}
\begin{proof}
Clearly, $\mathcal{A}$ is closable since it admits $\overline{D}$ as a closed linear extension. To prove that $\hat{\mathcal{A}}$ is closable, we use the characterization from \cref{lemma:closabledef}. Let $((q_m,\varphi_m))_m$ be a sequence in $\mathcal{D}(\hat{\mathcal{A}})$ converging in norm to zero and assume that $(\hat{\mathcal{A}}(q_m,\varphi_m))_m$ converges in norm to some $(p,\psi) \in \hat{\mathcal{F}}_T(\mathbb{R},X)$. Then \eqref{eq:D(Ahat)rep} tells us that $K\varphi_m \to p$ and $D\varphi_m \to \psi$ as $m \to \infty$. Since $K$ and $D$ are closable, $p$ and $\psi$ must be zero, which proves the claim.
\end{proof}

Next, we introduce a candidate for the characteristic operator $\Delta : \Omega \to L(\mathcal{F}_T(\mathbb{R},Y))$ of the closable linear operator $\mathcal{A}$ from \eqref{eq:D(A)rep}. Define the linear operator $\Delta(z) : \mathcal{D}(\Delta(z)) \subseteq \mathcal{F}_T(\mathbb{R},Y) \to \mathcal{F}_T(\mathbb{R},Y)$ with $z$-independent domain $\mathcal{Y} = \mathcal{D}(\Delta(z))$ from \ref{hyp:iota} by
\begin{equation} \label{eq:Delta(z)1}
    \Delta(z) \coloneqq -(zM - K)D_0R(z,D_0) \iota, \quad \forall z \in \Omega.
\end{equation}
To validate that $\Delta$ is a well-defined holomorphic closable operator-valued function (\cref{cor:Deltawelldefined}), we first need the following important result.

\begin{lemma} \label{lemma:Q(z)closable}
The closable linear operator $Q(z) : \mathcal{D}(Q(z)) \subseteq \mathcal{F}_T(\mathbb{R},X) \to \mathcal{F}_T(\mathbb{R},X)$ defined by
\begin{equation} \label{eq:actionQ}
    \mathcal{D}(Q(z)) \coloneqq \mathcal{D}(D), \quad Q(z) \coloneqq I - R(z,D_0)(zI - D), \quad \forall z \in \Omega \cup \{0\},
\end{equation}
is a projection. Moreover, $Q(z) \iota: \mathcal{Y} \to \mathcal{N}(zI-D)$ is a bounded linear operator with bounded (algebraic) inverse, which can be represented by
\begin{equation} \label{eq:Q(z)iota}
    Q(z)\iota = \iota - zR(z,D_0)\iota, \qquad [Q(z)\iota]^{-1} = \iota^{-1} Q(0) |_{\mathcal{N}(zI-D)}, \quad \forall z \in \Omega \cup \{0\}.
\end{equation}
\end{lemma}
\begin{proof}
First, note that the map $Q(z)$ is well-defined for all $z \in \Omega \cup \{0\}$ due to \ref{hyp:restriction} and \cref{lemma:rangeinclusion}. Second, since $I = I_{\mathcal{D}(D)}$ and $R(z,D_0)$ are bounded linear operators for all $z \in \Omega \cup \{0\}$ due to \ref{hyp:restriction}, and $zI-D$ is a closable linear operator, it follows from \cref{lemma:closablesum} and \cref{lemma:compoclosable} that $Q(z)$ is a closable linear operator. Third, recalling the resolvent identity \eqref{eq:resolventeq} on $D_0$ and using \ref{hyp:restriction}, a straightforward but rather lengthy calculation shows that
\begin{equation} \label{eq:Q(z)independent}
    Q(z_1) = Q(z_1) Q(z_2), \qquad \forall\, z_1, z_2 \in \Omega \cup \{0\}.
\end{equation}
Next, fix $z \in \Omega \cup \{0 \}$ and take $z_1 = z_2 = z$ in \eqref{eq:Q(z)independent} so that $Q(z)^2 = Q(z)$, i.e. $Q(z)$ is a projection. Consequently, we can directly compute that its kernel and range are given by
\begin{equation} \label{eq:propQ(z)}
    \mathcal{N}(Q(z)) = \mathcal{D}(D_0), \quad \mathcal{R}(Q(z)) = \mathcal{N}(zI - D).
\end{equation}
From here, it follows from \cite[Equation (4.8-2)]{Taylor1986} that
\begin{equation*}
    \mathcal{D}(D) = \mathcal{N}(zI- D) \oplus \mathcal{D}(D_0), \quad \forall z \in \Omega \cup \{0\}.
\end{equation*}
which we interpret as a $z$-translation along the kernel of $D$ of the direct sum decomposition \eqref{eq:D0_D}.

In the sequel, we will often need the action of $Q(z)$ on the range of $\iota$. Since $\mathcal{R}(\iota)= \mathcal{N}(D)$ due to \ref{hyp:iota}, we can directly compute from \eqref{eq:actionQ} that the action of $Q(z)\iota$ is given by \eqref{eq:Q(z)iota}. Moreover, since $\iota$ and $R(z, D_0)$ are bounded for $z \in \Omega \cup \{0\}$, we directly find that $\|Q(z)\iota\| \leq (1+|z| \|R(z,D_0)\|) \| \iota \| < \infty$, and hence $Q(z)\iota$ is bounded as well. This in particular also implies that $Q(z)\iota$ is closable (\cref{lemma:closablesum}).

It follows from \eqref{eq:propQ(z)} that $\mathcal{R}(Q(z)\iota) \subseteq \mathcal{N}(zI-D)$ and thus we can interpret $Q(z)\iota$ as mapping into $\mathcal{N}(zI-D)$. Our next aim is to show that this particular mapping is bijective and has bounded (algebraic) inverse. The first equality in \eqref{eq:Q(z)iota} can be rewritten as
\begin{equation} \label{eq:representationQ(z)iota}
    Q(z)\iota = (I - zR(z,D_0))\iota = [(zI-D_0)R(z,D_0) - zR(z,D_0)]\iota =- D_0 R(z,D_0) \iota,
\end{equation}
The right-hand side is a composition of injective maps, from where it follows that $Q(z)\iota$ is injective as well. To show that $Q(z)\iota$ is surjective, let $\psi \in \mathcal{N}(zI-D)$ be given and set $q = \iota^{-1}Q(0)\psi \in \mathcal{Y}$. Then it follows from \eqref{eq:Q(z)independent} that $Q(z)\iota q = Q(z)Q(0)\psi= Q(z)\psi$. Since $\psi \in \mathcal{N}(zI-D)$, it follows from \eqref{eq:actionQ} that $Q(z) \psi = \psi$. We conclude that $Q(z) \iota q = \psi$, and hence $Q(z) \iota$ is surjective, and thus bijective. The above computation also implies that the (algebraic) inverse of $Q(z) \iota$ is given by \eqref{eq:Q(z)iota}. To prove that this inverse map is bounded, let us first observe that
\begin{equation*}
    Q(0)\varphi = \varphi - D_0^{-1}D\varphi = (I - z D_0^{-1})\varphi, \quad \forall \varphi \in \mathcal{N}(zI-D).
\end{equation*}
Recalling that $D_0$ has a bounded (algebraic) inverse by \ref{hyp:restriction}, it follows that $Q(0)\rvert_{\mathcal{N}(zI-D)}$ is bounded. Since we also assumed in \ref{hyp:iota} that $\iota$ has bounded (algebraic) inverse, we conclude from \eqref{eq:Q(z)iota} that $Q(z) \iota$ has bounded (algebraic) inverse.
\end{proof}

\begin{corollary} \label{cor:Deltawelldefined}
The function $\Delta$ defined in \eqref{eq:Delta(z)1} admits the representation
\begin{equation} \label{eq:Delta(z)2}
    \Delta(z) = (zM - K)Q(z)\iota, \quad \forall z \in \Omega,
\end{equation}
and is therefore a holomorphic closable operator-valued function.
\end{corollary}
\begin{proof}
The expression \eqref{eq:Delta(z)2} for $\Delta$ is a direct consequence of the definition given in \eqref{eq:Delta(z)1} and the identity \eqref{eq:representationQ(z)iota}. Since $Q(z)\iota$ has range in $\mathcal{N}(zI-D) \subseteq \mathcal{D}(D) \subseteq \mathcal{D}(K)$, it is clear that $\Delta$ is well-defined. To prove that $\Delta$ is holomorphic, note that $z \mapsto zM-K$ and $Q (\cdot)\iota$ are holomorphic on $\Omega$ since the domains $\mathcal{D}(zM-K) = \mathcal{D}(K)$ and $\mathcal{D}(Q(z)\iota) = \mathcal{Y}$ are $z$-independent and the derivatives
\begin{equation*}
    \frac{d}{dz}(zM-K)\varphi = M\varphi, \quad \frac{d}{dz}Q(z)\iota q = - R(z,D_0) Q(z)\iota q,
\end{equation*}
exist for all $\varphi \in \mathcal{D}(K)$ and $q \in \mathcal{Y}$, which proves the claim. Here, the second equality follows from \eqref{eq:examplederi}. Since $M$ is bounded and $K$ is closable, it follows from \cref{lemma:closablesum} that $zM - K$ is closable. As $\Delta(z)$ is the composition of the closable linear operator $zM - K$ and the bounded linear operator $Q(z)\iota$, it follows from \cref{lemma:compoclosable} that $\Delta$ is a holomorphic closable operator-valued function.
\end{proof}

Before establishing that $\Delta$ 
is a characteristic operator for $\mathcal{A}$ on $\Omega$, we require some preliminary results. To this end, we first establish an equivalence after extension between $\Delta(z)$ and $zI - \hat{\mathcal{A}}$ on $\Omega$. This requires later on a precise description of the relation between $\mathcal{A}$ and $\hat{\mathcal{A}}$, which will be provided in \cref{lemma:graphJ} under the following hypothesis:

\begin{enumerate}[label=({H}{{\arabic*}})] 
\setcounter{enumi}{4}
    \item \label{hyp:domaincurlyA} The domain $\mathcal{D}(D) \subseteq \mathcal{R}(zI-\mathcal{A})$ for all $z \in \rho(\mathcal{A}) \cap \Omega$.
\end{enumerate}
As $\mathcal{D}(\mathcal{A}) \subseteq \mathcal{D}(D)$, \ref{hyp:domaincurlyA} tells us that $\mathcal{A}$ satisfies \ref{hyp:SH1}. Due to the next result, we also obtain that $\hat{\mathcal{A}}$ satisfies \ref{hyp:SH1}.

\begin{lemma} \label{lemma:graphJ}
Suppose that $\mathcal{A}$ and $\hat{\mathcal{A}}$ are the first and second operator associated with $D,K$ and $M$, respectively. Then $\mathcal{A}$ is similar to the part of $\hat{\mathcal{A}}$ in $\Gamma(M|_{\mathcal{R}(zI-\mathcal{A})})$ for all $z \in \rho(\mathcal{A}) \cap \Omega$. More precisely, let $J \in \mathcal{L}(\mathcal{F}_T(\mathbb{R}, X),\Gamma(M))$ be the linear isomorphism defined by 
\begin{equation} \label{eq:defJ}
    J \varphi  \coloneqq \begin{pmatrix}
    M \varphi \\
    \varphi
\end{pmatrix},
\end{equation}
then $J \mathcal{A} J^{-1}$ is the part of $\hat{\mathcal{A}}$ in $\Gamma(M|_{\mathcal{R}(zI-\mathcal{A})})$ for all $z \in \rho(\mathcal{A}) \cap \Omega$. Moreover, it holds that 
\begin{equation} \label{eq:JDAinclusion}
    J \mathcal{D}(\mathcal{A}) \subseteq \mathcal{D}(\hat{\mathcal{A}}) \subseteq \Gamma(M|_{\mathcal{R}(zI-\mathcal{A})}), \quad \forall z \in \rho(\mathcal{A}) \cap \Omega.
\end{equation}
and as a consequence of \cref{prop:spectraequal} and \cref{prop:similarity}, $\hat{\mathcal{A}}$ also satisfies \ref{hyp:SH1}.
\end{lemma}

\begin{remark}
One can also prove, as performed by Kaashoek and Verduyn Lunel in \cite[Lemma 3.3]{Kaashoek1992}, that $\mathcal{A}$ is similar to the part of $\hat{\mathcal{A}}$ in $\Gamma(M)$. However, using this result, one cannot apply \cref{prop:spectraequal} in the proof of the upcoming result \cref{cor:spectralrelations} that relates the spectral data of $\mathcal{A}$ with that of $\hat{\mathcal{A}}$ since $\Gamma(M)$ is in general not a subset of $\mathcal{R}(zI- \hat{\mathcal{A}})$ for $z \in \rho(\mathcal{A}) \cap \Omega$. \hfill $\lozenge$
\end{remark}

\begin{proof}[Proof of \cref{lemma:graphJ}]
Let us first prove \eqref{eq:JDAinclusion}. From the definition of $\mathcal{D}(\mathcal{A})$ and $\mathcal{D}(\hat{\mathcal{A}})$ given in \eqref{eq:D(A)rep} and \eqref{eq:D(Ahat)rep}, it follows directly that $J \mathcal{D}(\mathcal{A}) \subseteq \mathcal{D}(\hat{\mathcal{A}}) = \Gamma(M|_{\mathcal{D}(D)}) \subseteq \Gamma(M|_{\mathcal{R}(zI-\mathcal{A})})$ for all $z \in \rho(\mathcal{A}) \cap \Omega$, where the last inclusion follows from \ref{hyp:domaincurlyA}. To prove the second claim, let us first observe from \cref{def:similar} that $J\mathcal{A}J^{-1} : \mathcal{D}(J\mathcal{A}J^{-1}) \subseteq \Gamma(M) \to \Gamma(M)$ reads
\begin{equation} \label{eq:JaJinv}
    \mathcal{D}(J\mathcal{A}J^{-1}) = \bigg \{
    \begin{pmatrix}
        M \varphi \\
        \varphi
    \end{pmatrix}
    \in \hat{\mathcal{F}}_T(\mathbb{R},X) : \varphi \in \mathcal{D}(\mathcal{A}) \bigg \}, \quad J\mathcal{A}J^{-1}  = \hat{\mathcal{A}},
\end{equation}
where we recalled \eqref{eq:D(A)rep} and \eqref{eq:D(Ahat)rep} to obtain the last equality. Here, we also used that $\Gamma(M)$ is a Banach space, since it is a closed linear subspace of $\mathcal{F}_T(\mathbb{R},X)$ as $M$ is bounded. In the light of \cref{def:partof}, let us denote $\hat{\mathcal{A}}_{|}$ by the part of $\hat{\mathcal{A}}$ in $\Gamma(M|_{\mathcal{R}(zI-\mathcal{A})})$ for some fixed $z \in \rho(\mathcal{A}) \cap \Omega$. Since $J$ is continuous and $\mathcal{R}(zI-\mathcal{A})$ is dense in $\mathcal{F}_T(\mathbb{R},X)$, we know that $\Gamma(M|_{\mathcal{R}(zI-\mathcal{A})})$ is dense in $\Gamma(M)$. Therefore, $\hat{\mathcal{A}}_{|} : \mathcal{D}(\hat{\mathcal{A}}_{|}) \subseteq \Gamma(M) \to \Gamma(M)$ is given by $\mathcal{D}(\hat{\mathcal{A}}_{|}) = \mathcal{D}(J\mathcal{A}J^{-1})$ with action $\hat{\mathcal{A}}_{|} = \hat{\mathcal{A}},$ where we used \eqref{eq:D(A)rep}, \eqref{eq:D(Ahat)rep} and \ref{hyp:domaincurlyA} in the first equality, and \eqref{eq:JaJinv} in the second equality.

To prove the third claim, let us observe that $\mathcal{D}(\hat{\mathcal{A}}) \subseteq \Gamma(M|_{\mathcal{R}(zI-\mathcal{A})}) \subseteq \mathcal{R}(zI-\hat{\mathcal{A}})$ for all $z \in \rho(\mathcal{A}) \cap \Omega$, where the second inclusion follows from a direct computation using \eqref{eq:D(A)rep} and \eqref{eq:D(Ahat)rep}. Applying now \cref{prop:spectraequal} and \cref{prop:similarity} onto $\hat{\mathcal{A}}, J\mathcal{A}J^{-1}$ and $\mathcal{A}$, we obtain $\rho(\hat{\mathcal{A}}) \subseteq \rho(J\mathcal{A}J^{-1}) = \rho(\mathcal{A})$ and thus $\mathcal{D}(\hat{\mathcal{A}}) \subseteq \mathcal{R}(zI-\hat{\mathcal{A}})$  also holds for all $z \in \rho(\hat{\mathcal{A}}) \cap \Omega$, which completes the proof.
\end{proof}

\subsection{Construction of explicit equivalences and spectral properties} \label{subsec:equivalencescheme}

We are now in the position to establish an equivalence after extension between $\Delta(z)$ and $zI - \hat{\mathcal{A}}$ on $\Omega$. In particular, the next result shows that the conjugation operators $E$ and $F$ from \cref{def:equi}, together with the (algebraic) inverses $E(\cdot)^{-1}$ and $F(\cdot)^{-1}$, can be selected to be holomorphic bounded (and thus closable by \cref{lemma:closablesum}) operator-valued functions.

\begin{theorem} \label{thm:Delta equivalence}
Suppose that $\hat{\mathcal{A}}$ is the second operator associated with $D, K$ and $M$. Then there exist holomorphic bounded operator-valued functions $E : \Omega \to L(\hat{\mathcal{F}}_T(\mathbb{R},X),\hat{\mathcal{F}}_T(\mathbb{R},X)_{\hat{\mathcal{A}}})$ and $F : \Omega \to L(\hat{\mathcal{F}}_T(\mathbb{R},X))$, whose values are bijective mappings, satisfying
\begin{equation} \label{eq:Delta_equivalence}
    \begin{pmatrix}
        \Delta(z) & 0 \\
        0 & I 
    \end{pmatrix}
    =
    F(z)(z I - \hat{\mathcal{A}})E(z), \quad \forall z \in \Omega.
\end{equation}
Furthermore, the operators $E(z) \in \mathcal{L}(\mathcal{Y} \oplus \mathcal{R}(zI-D_0),\mathcal{D}(\hat{\mathcal{A}}))$ and $F(z) \in \mathcal{L}(\mathcal{F}_T(\mathbb{R},Y) \oplus \mathcal{R}(zI-D_0))$, and their bounded (algebraic) inverses, are given by
\begin{alignat*}{2}
    E(z)
    \begin{pmatrix}
    q\\
    \varphi
    \end{pmatrix}
    &=
    \begin{pmatrix}
    MQ(z)\iota q + MR(z,D_0)\varphi\\
    Q(z)\iota q + R(z,D_0)\varphi
    \end{pmatrix}, \quad E(z)^{-1}
    \begin{pmatrix}
    M\psi\\
    \psi
    \end{pmatrix}
    &&=
    \begin{pmatrix}
    \iota^{-1}Q(0)\psi\\
    (zI-D)\psi
    \end{pmatrix}, \\
    F(z)
    \begin{pmatrix}
    q\\
    \varphi
    \end{pmatrix}
    &=
    \begin{pmatrix}
    q - (zM-K)R(z,D_0)\varphi\\
    \varphi
    \end{pmatrix}, \quad
    F(z)^{-1}
    \begin{pmatrix}
    q\\
    \varphi
    \end{pmatrix}
    &&=
    \begin{pmatrix}
    q + (zM-K)R(z,D_0)\varphi \\
    \varphi
    \end{pmatrix}.
\end{alignat*}
\end{theorem}

\begin{remark} \label{remark:Ahatequiv}
The attentive reader may have already observed that $\Delta$ is not, strictly speaking in terms of \cref{def:characoperator}, a characteristic operator of $\hat{\mathcal{A}}$ on $\Omega$, since the equivalence \eqref{eq:Delta_equivalence} is not of the form \eqref{eq:equivalencematrix}, since
\begin{equation*}
zI-\hat{\mathcal{A}} \cong 
    \begin{pmatrix}
        zI & -K \\
        0  & zI - D
    \end{pmatrix}, \quad \forall z \in \Omega. 
\end{equation*}
However, for the sake of simplicity, we will nevertheless call $\Delta$ a characteristic operator for $zI - \hat{\mathcal{A}}$ on $\Omega$ anyway. We point out that the equivalence established here mainly serves as a preparatory step for showing that $\Delta$ is in fact a characteristic operator of $\mathcal{A}$ on $\Omega$ in \cref{thm:rhoAcharac}. In addition, the above rather simple and explicit equivalence proves useful for deriving the Jordan chains of $\mathcal{A}$ (and $\hat{\mathcal{A}}$) due to \cref{lemma:graphJ}, see in particular \cref{cor:Jordanchains}. \hfill $\lozenge$
\end{remark}

\begin{proof}[Proof of \cref{thm:Delta equivalence}]
We verify the criteria from \cref{def:characoperator} in the light of \cref{remark:Ahatequiv}. First, recall that the closable linear operator $\hat{\mathcal{A}}$ defined in \eqref{eq:D(A)rep} satisfies \ref{hyp:SH1} by \cref{lemma:graphJ}. Second, observe that the domain of the operator on the left-hand side of \eqref{eq:Delta_equivalence} is $z$-independent since $\mathcal{D}(\Delta(z)) = \mathcal{Y}$ for all $z \in \Omega$, and since the domain of $I = I_{\mathcal{R}(zI-D_0)}$ is $z$-independent due to \cref{lemma:rangeinclusion}. Similarly, the domain $\mathcal{D}(zI - \hat{\mathcal{A}}) = \mathcal{D}(\hat{\mathcal{A}})$ is $z$-independent as well. Third, we establish the equivalence \eqref{eq:Delta_equivalence} and the stated properties of $E, E^{-1}, F$ and $ F^{-1}$ in four steps.

\textbf{Step 1:} Before we construct the operator-valued function $E$, we first consider the operator-valued function $H : \Omega \to L(\hat{\mathcal{F}}_T(\mathbb{R},X), \mathcal{F}_T(\mathbb{R},X)_D)$ with values $H(z) : \mathcal{Y} \oplus \mathcal{R}(zI-D_0) \to \mathcal{D}(D)$ defined by
\begin{equation} \label{eq:H(z)rep}
    H(z) \begin{pmatrix}
        q \\
        \varphi
    \end{pmatrix}
    \coloneqq Q(z)\iota q + R(z,D_0) \varphi.
\end{equation}
We will prove that $H$ is a holomorphic closable operator-valued function whose values are bijective mappings. In particular, we show that for every $z \in \Omega$ 
the (algebraic) inverse $H(z)^{-1} : \mathcal{D}(D) \to \mathcal{Y} \oplus \mathcal{R}(zI-D_0)$ is given by
\begin{equation*}
    H(z)^{-1} \psi = 
    \begin{pmatrix}
        \iota^{-1} Q(0)\psi \\
        (zI-D)\psi
    \end{pmatrix}.
\end{equation*}
Note that $H$ and $H(\cdot)^{-1}$ are holomorphic since $Q(\cdot)\iota, R(\cdot,D_0)$ and $z \mapsto zI-D$ are holomorphic on $\Omega$. As $Q(z)\iota$ and $R(z,D_0)$ are bounded for all $z \in \Omega$, $H$ is a holomorphic bounded operator-valued function. To prove that $H(z)$ is injective, suppose that $Q(z)\iota q + R(z,D_0)\varphi = 0$. Since $Q(z)\iota q \in \mathcal{N}(zI - D)$ and $R(z,D_0)\varphi \in \mathcal{D}(D_0)$, it follows that $Q(z) \iota q = - R(z, D_0) \varphi \in \mathcal{N}(zI - D) \cap \mathcal{D}(D_0) = \{0 \}$ due to \eqref{eq:D0_D}, and thus $Q(z) \iota q = R(z, D_0) \varphi = 0$. As $Q(z)\iota$ and $R(z,D_0)$ are injective, $q$ and $\varphi$ must be zero, and thus $H(z)$ is injective. Our next aim is to show that $\mathcal{R}(H(z)) = \mathcal{D}(D)$. Since $Q(z)\iota$ has range in $\mathcal{N}(zI-D) \subseteq \mathcal{D}(D)$ and $R(z,D_0)$ has range in $\mathcal{D}(D_0) \subseteq \mathcal{D}(D)$, we have that $\mathcal{R}(H(z)) \subseteq \mathcal{D}(D)$. Conversely, let $\psi \in \mathcal{D}(D)$ be given. If we set $(q,\varphi) = (\iota^{-1} Q(0)\psi,(zI-D)\psi)$, then $(q,\varphi) \in \mathcal{Y} \oplus \mathcal{R}(zI - D_0) = \mathcal{D}(H(z))$ due to \cref{lemma:rangeinclusion} and \cref{lemma:Q(z)closable}. Moreover, we find that 
\begin{equation*}
    H(z) \begin{pmatrix}
    q \\ \varphi
\end{pmatrix} 
= Q(z) Q(0) \psi + R(z, D_0)(zI-D)\psi = Q(z) \psi + (\psi - Q(z)\psi) = \psi,
\end{equation*}
where we used \eqref{eq:actionQ} and \eqref{eq:Q(z)independent}. This proves that $H(z)$ is bijective. To prove that $H(z)^{-1}$ is bounded, recall from \ref{hyp:iota} that $\iota^{-1}$ is bounded and from \eqref{eq:actionQ} that $Q(0)$ and $zI-D$ are bounded in the graph norm $\|\cdot\|_{\overline{D}}$. Hence, we conclude that $H(z)^{-1}$ is bounded. 

\textbf{Step 2:} We next recall the linear isomorphism $J \in \mathcal{L}(\mathcal{F}_T(\mathbb{R},X), \Gamma(M))$ from \eqref{eq:defJ}, where we also recall that $\Gamma(M)$ denotes the graph of the operator $M$. We subsequently define the operator-valued function $E \coloneqq JH(\cdot)$. Then $E$ is a holomorphic bounded operator-valued function since $H$ is a holomorphic bounded operator-valued function and $J$ is bounded. Note that $\mathcal{D}(E(z)) = \mathcal{Y} \oplus \mathcal{R}(zI-D_0)$ is $z$-independent due to \cref{lemma:rangeinclusion}. Moreover, $E(z)$ is injective since $J$ and $H(z)$ are both injective. Since $\mathcal{R}(H(z)) = \mathcal{D}(D)$, it follows from \eqref{eq:D(Ahat)rep} that $\mathcal{R}(E(z)) = \mathcal{D}(\hat{\mathcal{A}})$. Hence, the values of $E$ are (algebraically) invertible, and for $(M \psi, \psi) \in \mathcal{D}(\hat{\mathcal{A}})$, we find that
\begin{equation*}
    E(z)^{-1}
        \begin{pmatrix}
        M\psi \\
        \psi
    \end{pmatrix}
    = H(z)^{-1} \psi =
    \begin{pmatrix}
        \iota^{-1} Q(0)\psi \\
        (zI-D)\psi
    \end{pmatrix}.
\end{equation*}
Clearly, $E(\cdot)^{-1}$ is holomorphic as $H(\cdot)^{-1}$ is holomorphic. To prove that $E(z)^{-1} = H(z)^{-1}J^{-1}$ is bounded, it remains to show that $J^{-1} : \Gamma(M) \to \hat{\mathcal{F}}_T(\mathbb{R},X)$ is bounded, but this follows from the bounded inverse theorem. 

\textbf{Step 3:} We subsequently define the operator-valued function $F$ as mentioned in the statement of \cref{thm:Delta equivalence} with values $F(z) : \mathcal{F}_T(\mathbb{R},Y) \oplus \mathcal{R}(zI-D_0) \to \mathcal{F}_T(\mathbb{R},Y) \oplus \mathcal{R}(zI-D_0)$. Note that $F(z)$ has $z$-independent domain and range due to \cref{lemma:rangeinclusion}. Moreover, $F$ and $F(\cdot)^{-1}$ are holomorphic since $z \mapsto zM-K$ and $R(\cdot,D_0)$ are holomorphic on $\Omega$. To prove that $F(z)$ is bounded, let us recall from  \ref{hyp:K} that $K$ acts as a bounded linear operator on $\mathcal{D}(D_0)$ and that $R(z, D_0)$ is bounded, which proves that $KR(z,D_0)$ is bounded. As $M$ is bounded, the claim follows. Moreover, we directly see from the expression of $F(z)$ that it has trivial kernel and is thus injective. Let us now prove that $F(z)$ is surjective. Therefore, let $(p,\varphi) \in \mathcal{F}_T(\mathbb{R},Y) \oplus \mathcal{R}(zI-D_0)$ be given. If we set $q = p + (zM-K)R(z,D_0)\varphi$, then $(q,\varphi) \in \mathcal{F}_T(\mathbb{R},Y) \oplus \mathcal{R}(zI-D_0)$ and clearly $F(z)(q,\varphi) = (p,\varphi),$ which proves the claim. This proves that $F(z)$ is bijective. A straightforward computation shows that $F(z)^{-1}$ is the (algebraic) inverse of $F(z)$. By the same argument as used for $F(z)$, we see that $F(z)^{-1}$ is bounded as well. Next, note that
\begin{equation*}
    \mathcal{R}(zI-\hat{\mathcal{A}}) = \bigg \{
    \begin{pmatrix}
        (zM-K)\varphi \\
        (zI-D)\varphi 
    \end{pmatrix}: \varphi \in \mathcal{D}(D) \bigg \} \subseteq  \mathcal{F}_T(\mathbb{R},Y) \oplus \mathcal{R}(zI-D_0) = \mathcal{D}(F(z)),
\end{equation*}
where the inclusion follows from \cref{lemma:rangeinclusion}. This allows us to consider the restriction $\tilde{F}(z) = F(z)|_{\mathcal{R}(zI-\hat{\mathcal{A}})}$ and recall from the text below \cref{def:equi} that $\mathcal{R}(\tilde{F}(z)) = \mathcal{R}(\Delta(z)) \oplus \mathcal{R}(zI-D_0)$.

\textbf{Step 4:} It remains to verify the equivalence \eqref{eq:Delta_equivalence}. Recalling the action of $\hat{\mathcal{A}}$ from \eqref{eq:D(Ahat)rep}, we find for $(q,\varphi) \in \mathcal{Y} \oplus \mathcal{R}(zI-D_0)$ that
\begin{equation*}
    (zI - \hat{\mathcal{A}})E(z) \begin{pmatrix}
    q \\
    \varphi
    \end{pmatrix} = 
    \begin{pmatrix}
    (zM - K)(Q(z)\iota q + R(z, D_0) \varphi) \\
    (zI - D) (Q(z) \iota q + R(z, D_0)\varphi)
\end{pmatrix} = \begin{pmatrix}
    \Delta(z)q + (zM-K)R(z,D_0)\varphi \\
    \varphi
\end{pmatrix},
\end{equation*}
where we used \eqref{eq:Delta(z)2} to simplify the expression in the first row, and recalled from \cref{lemma:Q(z)closable} that $Q(z)\,\iota q \in \mathcal{N}(zI - D)$ to simplify the expression in the second row. Recalling the definition of $F(z)$, the claim follows.
\end{proof}

As $\mathcal{A}$ is our main operator of interest, we would like to prove that $\Delta$ is also a characteristic operator for $\mathcal{A}$ on $\Omega$. To achieve this, we decompose $\hat{\mathcal{F}}_T(\mathbb{R},X) = \Gamma(M) \oplus \Gamma(M)^c$ as
\begin{equation*}
    \begin{pmatrix}
        q \\
        \varphi
    \end{pmatrix}
    =
    \begin{pmatrix}
        M\varphi \\
        \varphi
    \end{pmatrix}
    +
    \begin{pmatrix}
        q - M\varphi \\
        0
    \end{pmatrix},
\end{equation*}
and define two linear operators $T_1 : \mathcal{D}(T_1)  \subseteq \hat{\mathcal{F}}_T(\mathbb{R},X) \to \hat{\mathcal{F}}_T(\mathbb{R},X)$ and $T_2 : \mathcal{D}(T_2) \subseteq \hat{\mathcal{F}}_T(\mathbb{R},X) \to \hat{\mathcal{F}}_T(\mathbb{R},X)$ with domains 
\begin{equation*}
    \mathcal{D}(T_1) \coloneqq \mathcal{D}(T_2) \coloneqq \bigg \{ \begin{pmatrix}
        q \\
        \varphi
    \end{pmatrix} \in \hat{\mathcal{F}}_T(\mathbb{R},X) : \varphi \in \mathcal{D}(D) \bigg \}
\end{equation*}
and actions given by 
\begin{equation} \label{eq:T12action}
 T_1 
    \begin{pmatrix}
        q \\
        \varphi
    \end{pmatrix}
    \coloneqq JD\varphi = \begin{pmatrix}
        MD \varphi \\
        D \varphi
    \end{pmatrix},    
    \quad 
    T_2 
        \begin{pmatrix}
        q \\
        \varphi
    \end{pmatrix}
    \coloneqq
        \begin{pmatrix}
        (K - MD)\varphi \\
        0
    \end{pmatrix}.
\end{equation}
The operator $T_1$ is well-defined since $M$ acts on the entire space $\mathcal{F}_T(\mathbb{R},X)$, and the operator $T_2$ is well-defined due to the assumption $\mathcal{D}(K) \supseteq \mathcal{D}(D)$ from \ref{hyp:K}. Moreover, it holds that $\mathcal{R}(T_1) \subseteq \Gamma(M)$, $\mathcal{R}(T_2) \subseteq \Gamma(M)^c$, 
$\hat{\mathcal{A}} = T_1 + T_2$ and $J \mathcal{D}(\mathcal{A}) = \mathcal{D}(\hat{\mathcal{A}})\cap \mathcal{N} (T_2) $.

\begin{lemma} \label{lemma:T2onto}
The following statements hold:
\begin{enumerate}
    \item $\mathcal{R}(\Delta(z)) \subseteq \mathcal{R}(MD - K)$ and $\mathcal{R}(\Delta(z)) \oplus \{0\} \subseteq \mathcal{R}(T_2)$ for all $z \in \Omega$.
    \item If $\rho(\Delta) \neq \emptyset$, then $\mathcal{R}(T_2)$ is dense in $\Gamma(M)^c$.
\end{enumerate}
\end{lemma}
\begin{proof}
To prove the first statement, let $z \in \Omega$ be given and recall from \eqref{eq:Delta(z)1} that $\Delta(z) = (zM-K)Q(z)\iota$. Since $(zI - D) Q(z) \iota = 0$ due to \eqref{eq:propQ(z)}, it follows that
\begin{equation} \label{eq:zM_on_Q}
    z M Q(z) \iota = MD Q(z) \iota,
\end{equation}
meaning that $\Delta(z) = (MD-K)Q(z)\iota$, which implies that $\mathcal{R}(\Delta(z)) \subseteq \mathcal{R}(MD - K)$. We next recall from \cref{lemma:Q(z)closable} that $Q(z)\iota : \mathcal{Y} \to \mathcal{N}(z I-D)$ is bijective and thus the range of $\Delta(z)$ is equal to the range of $(MD - K)|_{\mathcal{N}(zI-D)}$. Moreover, it holds that $\mathcal{D}(MD - K) = \mathcal{D}(D)$, since $M$ is bounded and $\mathcal{D}(K) \supseteq \mathcal{D}(D)$ due to \ref{hyp:K}. Since $\mathcal{N}(zI-D)$ is a subspace of $\mathcal{D}(D) =\mathcal{D}(MD-K)$, it also follows that $\mathcal{R}(\Delta(z)) \oplus \{0\} \subseteq \mathcal{R}(T_2)$.

To prove the second statement, let $z \in \rho(\Delta)$ be given. Then $\mathcal{R}(\Delta(z)) \oplus \{0\}$ is dense in $\Gamma(M)^c$ and by the first part of the proof, we obtain that $\mathcal{R}(T_2)$ is dense in $\Gamma(M)^c$.
\end{proof}

As a preparatory step for the proof of \cref{thm:rhoAcharac}, we next prove that the linear operator $MD - K : \mathcal{D}(D) \subseteq \mathcal{F}_T(\mathbb{R},X) \to \mathcal{F}_T(\mathbb{R},Y)$ is closable. While $MD$ and $K$ are closable, we cannot directly guarantee that their difference is closable as well,
but this turns out to be true due to the underlying structure of the linear operators $M,D$ and $K$.

\begin{lemma} \label{lemma:MDKclosable}
The linear operator $MD-K$ is closable.    
\end{lemma}
\begin{proof}
Let $\varphi \in \mathcal{D}(D)$ be given and recall from the proof of \cref{thm:Delta equivalence} that $H(z)$ is a bijective mapping for every $z \in \Omega$. Hence, $\varphi = H(z)(q,\varphi_0) = Q(z)\iota q + R(z,D_0)\varphi_0$ for some $q \in \mathcal{Y}$ and $\varphi_0 \in \mathcal{R}(zI-D_0)$. Recalling \eqref{eq:zM_on_Q}, we find that
\begin{equation*} 
    (MD-K)\varphi = zMQ(z)\iota q - KQ(z)\iota q + (MD-K)R(z,D_0)\varphi_0.
\end{equation*}
Moreover, since $R(z,D_0)$ maps into $\mathcal{D}(D_0) \subseteq \mathcal{D}(D)$, we obtain the resolvent identity $D R(z, D_0) = zR(z, D_0)-I$, so that
\begin{equation} \label{eq:actionMDK}
     (MD-K)\varphi = zMQ(z)\iota q - KQ(z)\iota q + M(zR(z, D_0) - I) \varphi_0 - KR(z, D_0) \varphi_0. 
\end{equation}
We next consider the four terms in \eqref{eq:actionMDK} one by one:
because $M$ and $Q(z) \iota$ are bounded, the linear operator $zMQ(z)\iota$ is bounded.  Since $K$ is closable and $Q(z)$ is bounded, it follows from \cref{lemma:compoclosable} that $KQ(z)$ is closable. Since $M$ and $R(z, D_0)$ are bounded, $M(zR(z, D_0) - I)$ is bounded. Finally, since $K|_{\mathcal{D}(D_0)}$ is bounded by \ref{hyp:K} and since $R(z, D_0)$ is bounded, $KR(z, D_0)$ is bounded as well. Hence, \eqref{eq:actionMDK} tells us that $MD-K$ is the sum of a closable linear operator and three bounded linear operators, and we conclude from \cref{lemma:closablesum} that $MD-K$ is closable.    
\end{proof}

To continue, we actually need an equality between the linear (sub)spaces $\mathcal{R}(\Delta(z))$ and $ \mathcal{R}(MD - K)$ of $\mathcal{F}_T(\mathbb{R},Y)$ for at least one $z \in \rho(\Delta)$. Therefore, we impose the following hypothesis:
\begin{enumerate}[label=({H}{{\arabic*}})]
\setcounter{enumi}{5}
\item \label{hyp:rangeequal} 
$\rho(\Delta) \neq \emptyset$ and there exists a $z_0 \in \rho(\Delta)$ such that $\mathcal{R}(\Delta(z_0)) = \mathcal{R}(MD - K)$.
\end{enumerate}

\begin{lemma} \label{lemma:DAhatcompo}
We have the direct sum decomposition
\begin{equation} \label{eq:D(Ahat)directsum}
    \mathcal{D}(\hat{\mathcal{A}}) = \mathcal{D}(J\mathcal{A}J^{-1}) \oplus \mathcal{R}(T_2^+),
\end{equation}
where $T_2^+$ is a right inverse of $T_2$.
\end{lemma}
\begin{proof}
According to \ref{hyp:rangeequal}, let $z_0\in \rho(\Delta)$ be such that $\mathcal{R}(\Delta(z_0)) = \mathcal{R}(MD - K)$. Since $\mathcal{R}(\Delta(z_0)) \oplus \{0 \} \subseteq \mathcal{R}(T_2)$ due to \cref{lemma:T2onto}, and
$\mathcal{R}(T_2) \subseteq \mathcal{R}(MD - K) \oplus \{0 \}$ by \eqref{eq:T12action}, it follows that $\mathcal{R}(T_2) = \mathcal{R}(\Delta(z_0)) \oplus \{0\}$. Consider the linear map $T_2^+ : \mathcal{R}(T_2) \to \mathcal{D}(T_2)$ defined by
\begin{equation*}
    T_2^{+}
    \begin{pmatrix}
        q \\
        0
    \end{pmatrix}
    \coloneqq
    \begin{pmatrix}
        -MQ(z_0)\iota\Delta(z_0)^{-1}q \\
        -Q(z_0)\iota\Delta(z_0)^{-1}q
    \end{pmatrix},
\end{equation*}
which is well-defined since $Q(z_0)$ maps into $\mathcal{D}(D)$, recall \cref{lemma:Q(z)closable}. 
A brief calculation using \eqref{eq:zM_on_Q} and \cref{lemma:Q(z)closable} shows that $T_2 T_2^{+} = I_{\mathcal{R}(T_2)}$, and thus $T_2^+$ is a right inverse of $T_2$. We next introduce the projection operator $P \coloneqq T_2^{+}T_2 |_{\mathcal{D}(\hat{\mathcal{A}})} : \mathcal{D}(\hat{\mathcal{A}}) \to \mathcal{D}(\hat{\mathcal{A}})$ defined by
\begin{equation} \label{eq:projectionP}
    P
    \begin{pmatrix}
        M\varphi \\
        \varphi
    \end{pmatrix}
    \coloneqq
    \begin{pmatrix}
        MQ(z_0)\iota\Delta(z_0)^{-1}(MD-K)\varphi \\
        Q(z_0)\iota\Delta(z_0)^{-1}(MD-K)\varphi
    \end{pmatrix}. 
\end{equation}
To determine the range and kernel of this projection, we first observe that $\mathcal{R}(P) = \mathcal{R}(T_2^+)$ since $\mathcal{R}(T_{2})$ equals the range of $T_2$ restricted to $\mathcal{D}(\hat{\mathcal{A}})$. Consider now the linear operator $I-P : \mathcal{D}(\hat{\mathcal{A}}) \to \mathcal{D}(\hat{\mathcal{A}})$ and note directly that $\mathcal{D}(J\mathcal{A}J^{-1}) \subseteq \mathcal{R}(I-P)$ by a simple calculation. To prove the other inclusion, let $(M\psi,\psi) \in \mathcal{R}(I-P)$ be given. Then $\psi = [I - Q(z_0)\iota \Delta(z_0)^{-1} (MD-K)] \varphi$ for some $\varphi \in \mathcal{D}(D)$. Recalling \eqref{eq:zM_on_Q} shows that $(MD-K)\psi = 0$ and thus $(M \psi,\psi) \in \mathcal{D}(J\mathcal{A}J^{-1})$. We conclude that $P$ is a projection with range $\mathcal{R}(T_2^+)$ and kernel $\mathcal{D}(JAJ^{-1})$ so that direct sum decomposition \eqref{eq:D(Ahat)directsum} follows.
\end{proof}

We are now finally in the position to prove that the operator-valued function $\Delta$ from \eqref{eq:Delta(z)1} is a characteristic operator for $\mathcal{A}$ on $\Omega$.

\begin{theorem} \label{thm:rhoAcharac}
Suppose that $\mathcal{A}$ is the first operator associated with $D,K$ and $M$. Then the function $\Delta$ defined in \eqref{eq:Delta(z)1} is a characteristic operator for $\mathcal{A}$ on $\Omega$.
\end{theorem}

\begin{proof}
We verify the conditions of \cref{def:characoperator}. First, recall that the closable linear operator $\mathcal{A}$ defined in \eqref{eq:D(A)rep} satisfies \ref{hyp:SH1}. Second, we choose the spaces in that definition as $\tilde{X} = \mathcal{R}(zI - D_0) \subseteq \mathcal{F}_T(\mathbb{R}, X)$ and $\tilde{Z} = \mathcal{R}(MD -K) \subseteq Z$, where we take $Z = \mathcal{F}_T(\mathbb{R}, Y)$. By definition of $\Omega = \mathcal{\rho}(D_0)$, the space $\tilde{X}$ lies dense in $\mathcal{F}_T(\mathbb{R}, X)$ for all $z \in \Omega$. Moreover, recall from \ref{hyp:rangeequal} that there exists a $z_0 \in \rho(\Delta)$ such that $\mathcal{R}(\Delta(z_0)) = \mathcal{R}(MD-K) = \tilde{Z}$, meaning that the latter space is also dense in $Z$. Next, observe that $\mathcal{D}(\Delta(z)) = \mathcal{Y}$ and $\mathcal{D}(I_{\mathcal{R}(zI - D_0)}) = \mathcal{R}(zI - D_0)$ are $z$-independent due to \cref{lemma:rangeinclusion}. Moreover, $\mathcal{D}(zI - \mathcal{A}) = \mathcal{D}(\mathcal{A})$ and $\mathcal{R}(\Delta(z_0)) = \mathcal{R}(MD - K)$ are likewise $z$-independent due to \ref{hyp:rangeequal}. Third, we show that the holomorphic closable operator-valued functions
\begin{equation*}
    z \mapsto 
    \begin{pmatrix}
        \Delta(z) & 0 \\
        0 & I_{\mathcal{R}(zI - D_0)}
    \end{pmatrix}
    \quad \text{and} \quad
    z \mapsto
    \begin{pmatrix}
        I_{\mathcal{R}(MD - K)} & 0 \\
        0 & zI - \mathcal{A}
    \end{pmatrix}
\end{equation*}
are equivalent on $\Omega$ by verifying the conditions of \cref{def:equi} in three steps.

\textbf{Step 1:} We construct the conjugation operator-valued function $E : \Omega \to L(\hat{\mathcal{F}}_T(\mathbb{R},X),\mathcal{F}_T(\mathbb{R},Y) \oplus \mathcal{F}_T(\mathbb{R},X)_{\mathcal{A}})$ with values $E(z) : \mathcal{Y} \oplus \mathcal{R}(zI-D_0) \to \mathcal{R}(MD-K) \oplus \mathcal{D}(\mathcal{A})$ using the operator-valued function $H$ from \eqref{eq:H(z)rep} and the projection $P$ from \eqref{eq:projectionP} as
\begin{equation} 
\begin{aligned}
\label{eq:tildeE}
    E(z) 
    \begin{pmatrix}
        q \\
        \varphi
    \end{pmatrix}
    &\coloneqq
 \begin{pmatrix}
        (MD-K) H(z)(q, \varphi) \\
        J^{-1}(I-P)J H(z)(q, \varphi)
    \end{pmatrix} \\
    &=
    \begin{pmatrix}
        (MD - K)[Q(z)\iota q + R(z,D_0) \varphi] \\
        (I - Q(z_0)\iota\Delta(z_0)^{-1}(MD-K))[Q(z)\iota q + R(z,D_0) \varphi]
    \end{pmatrix}.
    \end{aligned}
\end{equation}
Since the operator-valued function $H$ from \eqref{eq:H(z)rep} is holomorphic on $\Omega$, we conclude that $E$ is holomorphic. To prove that $E(z)$ is closable, 
let $((q_m,\varphi_m))_m$ be a sequence in $\mathcal{Y} \oplus \mathcal{R}(zI-D_0)$ converging in norm to zero and suppose that $(E(z)(q_m,\varphi_m))_m$ converges in norm to some $(p,\psi) \in \hat{\mathcal{F}}_T(\mathbb{R},X)$. Since $MD-K$ is closable (\cref{lemma:MDKclosable}), $H(z)$ is bounded and $\mathcal{R}(H(z)) = \mathcal{D}(D)$ by the results from the proof of \cref{thm:Delta equivalence}, it follows from \cref{lemma:compoclosable} that $(MD-K)H(z)$ is closable and thus $\Psi_m = (MD-K)H(z) (q_m,\varphi_m) \to 0$ as $m \to \infty$, proving that $p = 0$. As a consequence, $-Q(z_0)\iota \Delta(z_0)^{-1} \Psi_m \to \psi$ as $m \to \infty$. Since $Q(z_0) \iota$ and $\Delta(z_0)^{-1}$ are bounded (\cref{lemma:Q(z)closable}), it follows from \cref{lemma:closablesum} that $Q(z_0) \iota \Delta(z_0)^{-1}$ is closable and thus $\psi = 0$, which proves the claim. 

To prove that $E(z)$ is injective, suppose that $E(z)(q,\varphi) = 0$. Then \eqref{eq:tildeE} tells us that $(MD-K)H(z)(q,\varphi) = 0$. Substituting this equality into the second component of $E(z)(q,\varphi) = 0$ yields $H(z)(q,\varphi) = 0$ and thus $(q,\varphi) = (0,0)$ since $H(z)$ is injective, recall the proof of \cref{thm:Delta equivalence}. 

To prove that $E(z)$ is surjective, note from \eqref{eq:tildeE} that
\begin{equation*}
    \mathcal{R}(E(z)) = 
    \bigg\{
    \begin{pmatrix}
        (MD-K)\varphi \\
        J^{-1}(I-P)J\varphi
    \end{pmatrix}
    :
    \varphi \in \mathcal{D}(D) \bigg\}.
\end{equation*}
Here, the first component in the right-hand side follows from the fact that the linear operator $H(z)$ from \eqref{eq:H(z)rep} maps its domain $\mathcal{Y} \oplus \mathcal{R}(zI-D_0)$ onto $\mathcal{D}(D)$. Since we proved in \cref{lemma:DAhatcompo} that $\mathcal{R}(I-P) = \mathcal{D}(J\mathcal{A}J^{-1})$, it is clear that $\mathcal{R}(E(z)) = \mathcal{R}(MD-K) \oplus \mathcal{D}(\mathcal{A})$, which proves the claim. We conclude that $E(z)$ is bijective and a straightforward computation shows that its (algebraic) inverse $E(z)^{-1} : \mathcal{R}(\Delta(z_0)) \oplus \mathcal{D}(\mathcal{A}) \to \mathcal{Y} \oplus \mathcal{R}(zI-D_0)$ has action
\begin{equation} \label{eq:E(z)inv}
    E(z)^{-1} 
    \begin{pmatrix}
        q \\
        \varphi
    \end{pmatrix}
    =
    \begin{pmatrix}
        \Delta(z_0)^{-1}q + \iota^{-1} Q(0)\varphi \\
        (z-z_0)Q(z_0)\iota \Delta(z_0)^{-1}q + (zI-D)\varphi
    \end{pmatrix},
\end{equation}
where $\mathcal{R}(\Delta(z_0)) = \mathcal{R}(MD-K)$ and $\mathcal{R}(zI-D) = \mathcal{R}(zI-D_0)$ due to \ref{hyp:rangeequal} and \cref{lemma:rangeinclusion}, respectively. Moreover, the map $E(\cdot)^{-1}$ is holomorphic since $z \mapsto zI - D$ is holomorphic. An argument analogous to the one given above shows that its values are also closable linear operators.

\textbf{Step 2:} We construct the conjugation operator-valued function $F : \Omega \to L(\hat{\mathcal{F}}_T(\mathbb{R},X))$ with values $F(z) : \mathcal{R}(\Delta(z_0)) \oplus \mathcal{R}(zI-D_0) \to \mathcal{R}(MD-K) \oplus \mathcal{R}(zI-D_0)$ defined by
\begin{equation} \label{eq:tildeF}
    F(z) 
    \begin{pmatrix}
        q \\
        \varphi
    \end{pmatrix}
    \coloneqq
    \begin{pmatrix}
        \Delta(z)\Delta(z_0)^{-1}q - (MD - K)R(z,D_0) \varphi \\
        (z-z_0)Q(z_0)\iota\Delta(z_0)^{-1}q + \varphi
    \end{pmatrix}.
\end{equation}
Note that the domain and range of $F(z)$ are $z$-independent for all $z \in \Omega$ due to \cref{lemma:rangeinclusion}. Moreover, since $\Delta$ and $R(\cdot,D_0)$ are holomorphic, it follows that $F$ is a holomorphic operator-valued function. To prove that $F(z)$ is closable, let $((q_m,\varphi_m))_m$ be a sequence in $\mathcal{R}(\Delta(z_0)) \oplus \mathcal{R}(zI - D_0)$ converging in norm to zero and suppose that $(F(z)(q_m,\varphi_m))_m$ converges in norm to some $(p,\psi) \in \hat{\mathcal{F}}_T(\mathbb{R},X)$. Since $\Delta(z_0)^{-1}$ is bounded and $\Delta(z)$ is closable, it follows from \cref{lemma:compoclosable} that $\Delta(z)\Delta(z_0)^{-1}$ is closable and so $\Delta(z)\Delta(z_0)^{-1} q_m \to 0$ as $m \to \infty$. Hence, $-(MD-K)R(z,D_0)\varphi_m \to p$ as $m \to \infty$. Since $MD - K$ is closable due to \cref{lemma:MDKclosable} and $R(z,D_0)$ is bounded, it follows from \cref{lemma:compoclosable} that $p = 0$. From the second component of \eqref{eq:tildeF},
we obtain $(z-z_0)Q(z_0)\iota \Delta(z_0)^{-1} q_m \to \psi$ as $m \to \infty$. As argued earlier, the linear operator $(z-z_0)Q(z_0)\iota \Delta(z_0)^{-1}$ is closable, proving that $\psi = 0$. 

To prove that $F(z)$ is injective, suppose that $F(z)(q,\varphi) = 0$. Then $\varphi = (z_0-z)Q(z_0) \iota \Delta(z_0)^{-1} q$ and substituting this into the first component of \eqref{eq:tildeF} by using \eqref{eq:Delta(z)2} yields
\begin{align*}
    0 &= \left[ (zM - K) Q(z) \iota - (z - z_0) (MD-K) R(z, D_0) Q(z_0) \iota\right] \Delta(z_0)^{-1} q \\
    &= (MD-K)\left[ Q(z) \iota - (z - z_0) R(z, D_0) Q(z_0) \iota\right] \Delta(z_0)^{-1} q,
\end{align*}
where in the last step we recalled \eqref{eq:zM_on_Q}. Next, we take a look at the linear operator $Q(z)\iota - (z_0-z)R(z,D_0)Q(z_0)\iota$ and observe from \eqref{eq:Q(z)iota} that
\begin{align*}
    Q(z)\iota - (z_0-z)R(z,D_0)Q(z_0)\iota &= \iota -zR(z,D_0)\iota - (z_0-z)R(z,D_0)\iota + z_0(z_0-z)R(z,D_0)R(z_0,D_0)\iota \\
    &= \iota -zR(z,D_0)\iota - (z_0-z)R(z,D_0)\iota +z_0(R(z,D_0)-R(z_0,D_0))\iota \\
    &= Q(z_0)\iota,
\end{align*}
where we used the resolvent equation \eqref{eq:resolventeq} in the second equality. Hence,
\begin{equation*}
    0 = (MD-K)[Q(z)\iota - (z_0-z)R(z,D_0)Q(z_0)\iota] \Delta(z_0)^{-1} q = (MD-K)Q(z_0)\iota \Delta(z_0)^{-1}q = q,
\end{equation*}
where the third equality follows from \eqref{eq:Delta(z)2} in combination with \eqref{eq:zM_on_Q}. 
Since $q=0$, we observe from the second component of $F(z)$ that $\varphi = 0$, meaning that $F(z)$ is injective. 

To prove that $F(z)$ is surjective, let $(p,\psi) \in \mathcal{R}(MD-K) \oplus \mathcal{R}(zI-D_0)$ be given. Consider the vectors $q = p + (MD-K)R(z,D_0)\psi$ and $\varphi = \psi + (z_0-z)Q(z_0)\iota\Delta(z_0)^{-1}[p + (MD-K)R(z,D_0)\psi]$. As $p \in \mathcal{R}(MD-K)$, there exist a $p_0 \in \mathcal{D}(D)$ such that $p = (MD-K)p_0$ and thus $q = (MD-K)[p_0 + R(z,D_0)\psi] \in \mathcal{R}(MD-K)$. Since $Q(z_0)\iota$ maps into $\mathcal{D}(D) \subseteq \mathcal{R}(zI-D_0)$ for $z \in \Omega$ by \ref{hyp:domainrangeD}, we conclude that $\varphi \in \mathcal{R}(zI-D_0)$. A straightforward, but rather lengthy calculation shows that $F(z)(q,\varphi) = (p,\psi),$ which proves the claim. We conclude that $F(z)$ is bijective and its (algebraic) inverse $F(z)^{-1} : \mathcal{R}(\Delta(z_0)) \oplus \mathcal{R}(zI-D_0) \to \mathcal{R}(\Delta(z_0)) \oplus \mathcal{R}(zI-D_0)$ has action
\begin{equation} \label{eq:F(z)inv}
    F(z)^{-1} 
    \begin{pmatrix}
        q \\
        \varphi
    \end{pmatrix}
    =
    \begin{pmatrix}
        q + (MD-K)R(z,D_0)\varphi \\
        \varphi + (z_0-z)Q(z_0)\iota\Delta(z_0)^{-1}[q + (MD-K)R(z,D_0)\varphi]
    \end{pmatrix},
\end{equation}
where we recalled again that $\mathcal{R}(\Delta(z_0)) = \mathcal{R}(MD-K)$ due to \ref{hyp:rangeequal}. Clearly, $F(\cdot)^{-1}$ is holomorphic since $R(\cdot,D_0)$ is holomorphic on $\Omega$. An argument analogous to the one given above shows that its values are closable linear operators. Next, note that
\begin{equation*}
    \mathcal{R}(\Delta(z_0)) \oplus \mathcal{R}(zI-\mathcal{A})\subseteq \mathcal{R}(\Delta(z_0)) \oplus \mathcal{R}(zI-D_0),
\end{equation*}
since $\mathcal{A}$ is a restriction of $D$ and $\mathcal{R}(zI-D) = \mathcal{R}(zI-D_0)$ due to \cref{lemma:rangeinclusion}. This allows us to consider the restriction $\tilde{F}(z) = F(z)|_{\mathcal{R}(\Delta(z_0)) \oplus \mathcal{R}(zI-\mathcal{A})}$, and recall from the text below \cref{def:equi} that $\mathcal{R}(\tilde{F}(z)) = \mathcal{R}(\Delta(z)) \oplus \mathcal{R}(zI-D_0)$. 

\textbf{Step 3:} A straightforward calculation shows that
\begin{equation*}
    \begin{pmatrix}
        \Delta(z) & 0 \\
        0 & I_{\mathcal{R}(zI - D_0)}
    \end{pmatrix}
    =
    F(z)
    \begin{pmatrix}
        I_{\mathcal{R}(MD-K)}& 0 \\
        0 & zI - \mathcal{A}
    \end{pmatrix}
    E(z), \quad \forall z \in \Omega,
\end{equation*}
which proves the theorem.
\end{proof}

Since we have established that $\Delta$ is a characteristic operator for both $\mathcal{A}$ and $\hat{\mathcal{A}}$ on $\Omega$, we are now in the position to analyse their spectral properties by applying \cref{thm:spectralequal}. In particular, we show that $\Delta$ satisfies the (spectral) hypotheses \ref{hyp:SH2} and \ref{hyp:SH3}. Thereby, we can relate the spectral information of $\mathcal{A}$ and $\hat{\mathcal{A}}$ to that of $\Delta$. As a consequence, we can also relate the spectral information of $\mathcal{A}$ to that of $\hat{\mathcal{A}}$.

\begin{corollary} \label{cor:spectralrelations}
Suppose that $\mathcal{A}$ and $\hat{\mathcal{A}}$ are the first and second operator associated with $D,K$ and $M$ respectively, and let $\Delta$ defined in  \eqref{eq:Delta(z)1} be the associated characteristic operator on $\Omega$. Then $\Delta$ satisfies \ref{hyp:SH2} and \ref{hyp:SH3}, and there holds
\begin{gather}
    \label{eq:pointspecAAhatDelta}
    \sigma_p(\mathcal{A}) \cap \Omega = \{ \sigma \in \Omega : \sigma \mbox{ is a characteristic value of } \Delta \} = \sigma_p(\hat{\mathcal{A}}) \cap \Omega \\
    \label{eq:specAAhatDelta}
    \sigma(\mathcal{A}) \cap \Omega = \{ \sigma \in \Omega : \Delta(\sigma) \mbox{ is not invertible} \} = \sigma(\hat{\mathcal{A}}) \cap \Omega.
\end{gather}
Furthermore, the resolvent of $\mathcal{A}$ and $\hat{\mathcal{A}}$ at $z \in \rho(\Delta)$ can be represented by
\begin{align} \label{eq:resolventADelta}
\begin{split}
    R(z,\mathcal{A})\varphi &= Q(z)\iota \Delta(z)^{-1}[M\varphi - (zM-K)R(z,D_0)\varphi] + R(z,D_0)\varphi, \quad \forall \varphi \in \mathcal{R}(zI-\mathcal{A}), \\
        R(z,\hat{\mathcal{A}})
    \begin{pmatrix}
        q \\
        \varphi
    \end{pmatrix}
    &= J \big(Q(z)\iota \Delta(z)^{-1}[q - (zM-K)R(z,D_0)\varphi] + R(z,D_0)\varphi \big), \quad \forall \begin{pmatrix}
        q \\
        \varphi
    \end{pmatrix} \in \mathcal{R}(zI-\hat{\mathcal{A}}),
\end{split}
\end{align}
and there holds $\mathcal{R}(zI-\mathcal{A}) = \mathcal{R}(zI-D_0)$, which is $z$-independent for all $z \in \rho(\mathcal{A}) \cap \Omega$. Moreover, if $\sigma \in \sigma_p(\mathcal{A}) \cap \Omega$, then $m_g(\sigma,\mathcal{A}) = m_g(\sigma,\hat{\mathcal{A}}) = m_g(\Delta(\sigma)), m_a(\sigma,\mathcal{A}) = m_a(\sigma,\hat{\mathcal{A}}) = m_a(\Delta(\sigma))$ and $k(\sigma,\mathcal{A}) = k(\sigma,\hat{\mathcal{A}}) = k(\Delta(\sigma))$. If in addition $\sigma$ is of finite type, then the partial multiplicities of $\sigma$ considered as an eigenvalue of $\mathcal{A}$ and $\hat{\mathcal{A}}$ are equal to the zero multiplicities of $\Delta$ at $\sigma$. Moreover, if $\sigma$ is isolated in $\sigma(\mathcal{A}) \cap \Omega$, then $r(\sigma,\mathcal{A}) = r(\sigma,\hat{\mathcal{A}}) = r(\Delta(\sigma))$.
\end{corollary}
\begin{proof}
To prove \eqref{eq:pointspecAAhatDelta}, recall that $\Delta$ is a characteristic operator for $\mathcal{A}$ and $\hat{\mathcal{A}}$ on $\Omega$. The assertion then follows immediately from \eqref{eq:pointspectrumsigmaADelta}. Moreover, the claims regarding the relationships among the various notions of multiplicity and the ascent follow directly from \cref{thm:spectralequal}.

To establish the first equality in \eqref{eq:specAAhatDelta}, we shall invoke the results of \cref{thm:spectralequal} applied to the equivalence between $\mathcal{A}$ and $\Delta$ given in \cref{thm:rhoAcharac}. We begin by recalling from \ref{hyp:rangeequal} that $\rho(\Delta)\neq\emptyset$. Consequently, it suffices to verify that \ref{hyp:SH3} holds on this equivalence. To compute the function $G_\Delta(z) : \mathcal{R}(zI-\mathcal{A}) \to \mathcal{D}(\mathcal{A})$ defined in \ref{hyp:SH3}, we combine \eqref{eq:tildeE} and \eqref{eq:tildeF} and find, after a lengthy but routine calculation, that
\begin{align*}
    G_\Delta(z) &= J^{-1}(I - P)J[-Q(z)\iota \Delta(z)^{-1}[(MD-K)R(z,D_0)] + R(z,D_0)] \\
            &= Q(z)\iota \Delta(z)^{-1}[M-(zM-K)R(z,D_0)] + R(z,D_0), \quad \forall z \in \rho(\Delta),
\end{align*}
where $I-P$ is the projection from the proof of \cref{lemma:DAhatcompo} with range $\mathcal{D}(J\mathcal{A}J^{-1})$. Since \ref{hyp:rangeequal} holds, the linear operator $G_\Delta(z)$ can be extended to its maximal domain $\mathcal{R}(zI - D_0)$, which is dense in $\mathcal{F}_T(\mathbb{R}, X)$. Hence, $G_\Delta(z) : \mathcal{R}(zI-D_0) \to \mathcal{D}(\mathcal{A})$ is bounded as it is a composition of bounded linear operators, where we explicitly recall that $K|_{\mathcal{D}(D_0)}$ is bounded by \ref{hyp:K}. Since $z \in \rho(\Delta) \subseteq \Omega$, we observe that $\mathcal{D}(G_\Delta(z)) = \mathcal{R}(zI - D_0)$ is dense in $\mathcal{F}_T(\mathbb{R},X)$. As $R(z,\mathcal{A}) = G_\Delta(z)$, the representation of the resolvent of $\mathcal{A}$ at $z \in \rho(\Delta)$ provided in \eqref{eq:resolventADelta} follows and $\mathcal{R}(zI-\mathcal{A}) = \mathcal{R}(zI-D_0)$ is $z$-independent due to \cref{lemma:rangeinclusion}. To compute $G_\mathcal{A}(z) : \mathcal{R}(\Delta(z)) \to \mathcal{Y}$, we combine \eqref{eq:E(z)inv} and \eqref{eq:F(z)inv} to obtain
\begin{equation*}
    G_\mathcal{A}(z) = \Delta(z_0)^{-1} + (z_0 - z) \iota^{-1}Q(0) R(z,\mathcal{A})Q(z_0)\iota \Delta(z_0)^{-1}, \quad \forall z \in \rho(\mathcal{A}) \cap \Omega.
\end{equation*}
Due to \cref{lemma:T2onto} and \ref{hyp:rangeequal}, the linear operator $G_\mathcal{A}(z)$ is still well-defined on its maximal domain $\mathcal{R}(\Delta(z_0))$. Clearly, $G_{\mathcal{A}}(z)$ is bounded since it is the composition of bounded linear operators. In particular, the graph norm $\|\cdot\|_{\overline{\mathcal{A}}}$ ensures the boundedness of the $Q(0)$ component, see a similar argument establishing the boundedness of $H(z)^{-1}$ in the proof of \cref{thm:Delta equivalence}. Moreover, $\mathcal{D}(G_\mathcal{A}(z)) = \mathcal{R}(\Delta(z_0))$, which is dense in $\mathcal{F}_T(\mathbb{R},Y)$ since $z_0 \in \rho(\Delta) \neq \emptyset$ by \ref{hyp:rangeequal}. This proves that \ref{hyp:SH3} holds for $\Delta$ as a characteristic operator of $\mathcal{A}$ on $\Omega$, meaning that the first equality in \eqref{eq:specAAhatDelta} holds. As $\Delta(z)^{-1} = G_\mathcal{A}(z),$ we obtain $\mathcal{R}(\Delta(z)) = \mathcal{R}(\Delta(z_0))$ for all $z \in \rho(\Delta)$ and thus $\Delta$ satisfies \ref{hyp:SH2}.

To establish the second equality in \eqref{eq:specAAhatDelta}, let us first recall from \cref{lemma:graphJ} that $\mathcal{A}$ is similar to the part of $\hat{\mathcal{A}}$ in $\Gamma(M|_{\mathcal{R}(zI-\mathcal{A})})$ for all $z \in \rho(\mathcal{A}) \cap \Omega$. Hence, the spectral results from \cref{prop:spectraequal} and \cref{prop:similarity} tell us that $\sigma(\mathcal{A}) \subseteq \sigma(\hat{\mathcal{A}})$, and thus $\sigma(\mathcal{A}) \cap \Omega \subseteq \sigma(\hat{\mathcal{A}}) \cap \Omega$. To prove the other inclusion, let $z \in \rho(\mathcal{A}) \cap \Omega$ be given, then $z \in \rho(\Delta)$ by the first equality in \eqref{eq:specAAhatDelta}. Recalling the representation of $E(z)$ and $F(z)$ from \cref{thm:Delta equivalence}, we obtain for all $(q,\varphi) \in \mathcal{R}(zI-\hat{\mathcal{A}})$ that
\begin{equation*}
    R(z,\hat{\mathcal{A}}) \begin{pmatrix}
        q \\
        \varphi
    \end{pmatrix} = J \big(Q(z)\iota \Delta(z)^{-1}[q - (zM-K)R(z,D_0)\varphi] + R(z,D_0)\varphi \big),
\end{equation*}
and thus $z \in \rho(\hat{\mathcal{A}}) \cap \Omega$, which proves the claim. This also proves the second equality in \eqref{eq:resolventADelta}.

The remaining claim regarding the order of a pole is now a consequence of \cref{prop:spectraequal}, \cref{prop:similarity}, \cref{thm:spectralequal} and \cref{lemma:graphJ}.
\end{proof}

The following result states the multiplicity theorem between $\mathcal{A}$ and $\Delta$. The proof is a direct consequence of the abstract multiplicity theorem from \cref{thm:spectralequal}, and is therefore omitted.

\begin{corollary} \label{cor:multiplicity}
Let $\sigma$ be an isolated characteristic value of finite type of $\Delta$ satisfying $k(\sigma,\mathcal{A}) = r(\Delta(\sigma))$. Then the multiplicity theorem $\dim \mathcal{N}((\sigma I - \mathcal{A})^{r(\Delta(\sigma))}) =  m_a(\Delta(\sigma))$ holds.
\end{corollary}

The following result shows that we can compute a canonical basis of (generalized) eigenvectors of $\mathcal{A}$ and $\hat{\mathcal{A}}$ in terms of a canonical system of Jordan chains of $\Delta$ by using the simpler equivalence derived in \cref{thm:Delta equivalence}.

\begin{corollary} \label{cor:Jordanchains}
Let $\sigma$ be a characteristic value of $\Delta$ of finite type, and let 
$\{ \{ q_{i,0}, \dots, q_{i,k_i-1} \}$ : $i=1,\dots,p\}$, be a canonical system of Jordan chains of $\Delta$ at $\sigma$. Define for each $i\in \{1,\dots,p\}$ the vectors
\begin{equation*}
    \varphi_{i,k} = \sum_{l=0}^k (-1)^{l+1} D_0 R(\sigma,D_0)^{l+1} \iota\, q_{i,k-l},
    \qquad k = 0,\dots,k_i-1 .
\end{equation*}
Then $\{ \{ {\varphi}_{i,0},\dots,{\varphi}_{i,k_i-1} \} : i=1,\dots,p\}$ is a canonical basis of (generalized) eigenvectors of $\mathcal{A}$ at $\sigma$, and $\{ \{J \varphi_{i,0},\dots,J{\varphi}_{i,k_i-1} \} : i=1,\dots,p\}$ is a canonical basis of (generalized) eigenvectors of $\hat{\mathcal{A}}$ at $\sigma$.
\end{corollary}
\begin{proof}
Let $\{q_0,\dots,q_{k-1} \}$ be a Jordan chain of $\Delta$ at $\sigma$ of length $k$. Due to the equivalence \eqref{eq:Delta_equivalence}, we see for all $z \in \Omega$ that
\begin{equation*}
    E(z)
    \begin{pmatrix}
        \sum_{i=0}^{k-1}(z-\sigma)^{i}q_i \\
        0
    \end{pmatrix}
    =JW_\sigma(z), \quad W_\sigma(z) \coloneqq - \sum_{i=0}^{k-1} (z-\sigma)^{i} D_0 R(z,{D_0}) \iota q_i,
\end{equation*}
and so $z \mapsto W_\sigma(z)$ is a root function for $zI - \hat{\mathcal{A}}$ at $\sigma$. From the proof of \cref{thm:spectralequal}, it is clear that we have to expand $W_\sigma$ up to order $k$ in a neighbourhood of $\sigma$. Since $D_0$ satisfies \ref{hyp:SH1}, it follows from \cref{prop:resolventproperties} that $R(z,D_0) = \sum_{l=0}^{\infty} (-1)^l(z-\sigma)^l R(\sigma,D_0)^{l+1}$ for $z$ sufficiently close to $\sigma$. Hence,
\begin{equation*}
    W_\sigma(z) = \sum_{i=0}^{k-1} \bigg( \sum_{l=0}^i (-1)^{l+1} D_0R(\sigma,D_0)^{l+1} \iota q_{i-l} \bigg)(z-\sigma)^i + \mathcal{O}((z-\sigma)^k), 
\end{equation*}
for $z$ sufficiently close to $\sigma$, which shows that $\{J \varphi_{0},\dots,J{\varphi}_{k-1} \}$ is a Jordan chain of $\hat{\mathcal{A}}$ at $\sigma$. Since this Jordan chain is a subset of $\mathcal{D}(J\mathcal{A}J^{-1})$, \cref{lemma:graphJ} tells us that $\{ \varphi_{0},\dots,{\varphi}_{k-1} \}$ is a Jordan chain of $\mathcal{A}$ at $\sigma$. Because $\rho(\Delta) \neq \emptyset$ due to \ref{hyp:rangeequal}, \cref{thm:spectralequal} holds and so the Jordan chains for $\mathcal{A}$ and $\Delta$ are in one-to-one correspondence due to \cref{prop:equivalent chains}.
\end{proof}

\subsection{Periodic spectral properties and similarity} \label{subsec:periodicspectral}
In the concrete applications in \cref{sec:applications}, it turns out that the eigenvalues of $\mathcal{A}$, and thus the characteristic values of $\Delta$ (\cref{cor:spectralrelations}), will play the role of the Floquet exponents, which are generically only uniquely determined up to an additive multiple of $\frac{2 \pi i}{T}$. Consequently, it would be interesting to investigate whether such periodic spectral patterns are also visible in $\mathcal{A}$ and $\Delta$. If this is the case, the spectral problem could be significantly simplified on $\Omega$, as we would only need to examine the spectral behaviour of $\mathcal{A}$ and $\Delta$ on the subset $\{z \in \Omega : \Im(z) \in (-\frac{\pi}{T}, \frac{\pi}{T}] \}$ instead of the entire domain $\Omega$.

To establish such a construction, we show that $zI - \mathcal{A} \sim (z + \frac{2 \pi i}{T}  k)I - \mathcal{A}$ and $\Delta(z) \sim \Delta(z+\frac{2 \pi i}{T} k)$ for all $z \in \Omega$ and $k \in \mathbb{Z}$. From now on, we refer to the maps 
$z \mapsto \Delta(z)$ and $z \mapsto zI - \mathcal{A}$ as 
$\frac{2\pi i}{T}\mathbb{Z}$\emph{-similar} on $\Omega$.
To achieve this, let $Z \in \{X,Y\}$ be given and consider a bounded linear operator $\varepsilon_k : \mathcal{F}_T(\mathbb{R},Z) \to \mathcal{F}_T(\mathbb{R},Z)$ with action $(\varepsilon_k u)(t) \coloneqq e^{\frac{2 \pi k i}{T}t}u(t)$ for all $k \in \mathbb{Z}$. Although $\varepsilon_k$ depends on the underlying space $Z$, this dependence is suppressed in the notation, as it will be clear from the context. To proceed, we impose the following (periodic) hypothesis, which we will verify in concrete examples in \cref{sec:applications}:
\begin{enumerate}[label=({PH}{{\arabic*}})]
\setcounter{enumi}{0}
\item \label{hyp:ek} For all $z \in \Omega$ and $k \in \mathbb{Z}$, the map $\varepsilon_k$ leaves the spaces $\mathcal{D}(D_0), \mathcal{D}(D),\mathcal{D}(K)$ and $\mathcal{R}(zI-D_0)$ invariant in $\mathcal{F}_T(\mathbb{R},X)$, and the spaces $\mathcal{Y}$ and $\mathcal{R}(MD-K)$ invariant in $\mathcal{F}_T(\mathbb{R},Y)$.
\end{enumerate}
Note that $\varepsilon_k$ is invertible with bounded inverse $\varepsilon_{-k}$. To construct the similarities, we have to assume the following relations:
\begin{equation} \tag{PH2} \label{eq:periodicconditions}
    \varepsilon_k D \varepsilon_k^{-1} = \frac{2 \pi k i}{T} I + D, \ \varepsilon_k K \varepsilon_k^{-1} = \frac{2 \pi k i}{T} M + K, \ \varepsilon_k M \varepsilon_k^{-1} = M, \ \varepsilon_k Q\bigg(z + \frac{2 \pi k i}{T}\bigg) \iota \varepsilon_k^{-1} = Q(z)\iota,
\end{equation}
for all $z \in \Omega$ and $k \in \mathbb{Z}$. The hypothesis \ref{hyp:ek} guarantees that the relations from \eqref{eq:periodicconditions} are well-defined. Before we can prove the main result of this section (\cref{prop:Deltasimilar}), we first have to show that $\Omega$ itself admits such a periodic pattern.
\begin{lemma}
If $z \in \Omega$, then $z + \frac{2\pi i}{T} \mathbb{Z} \subseteq \Omega$.
\end{lemma}
\begin{proof}
Let $z \in \Omega = \rho(D_0)$ be given. Then the equation $(zI-D_0)\varphi = \psi$ has a unique bounded solution $\varphi = R(z,D_0)\psi \in \mathcal{D}(D_0)$ for $\psi$ in the dense subspace $\mathcal{R}(zI-D_0)$ of $\mathcal{F}_T(\mathbb{R},X)$. Restricting the first equality in \eqref{eq:periodicconditions} to $\mathcal{D}(D_0)$, and recalling \ref{hyp:ek}, yields eventually
\begin{equation*}
    \bigg[\bigg( z + \frac{2 \pi k i}{T}\bigg) I - \varepsilon_k D_0 \varepsilon_{k}^{-1} \bigg] \varphi = (zI-D_0)\varphi = \psi, \quad \forall k \in \mathbb{Z},
\end{equation*}
meaning that $z + \frac{2 \pi k i}{T} \in \rho(\varepsilon_k D_0 \varepsilon_{k}^{-1}) = \rho(D_0)$, where this equality follows from \cref{prop:similarity} as $\varepsilon_k$ is bounded and the closable linear operator $D_0$ satisfies \ref{hyp:SH1}. This proves the result.
\end{proof}
\begin{proposition} \label{prop:Deltasimilar}
The maps $z \mapsto zI - \mathcal{A}$ and $z \mapsto \Delta(z)$ are $\frac{2 \pi i}{T}\mathbb{Z}$-similar on $\Omega$.
\end{proposition}
\begin{proof}
Let $z \in \Omega$ and $k \in \mathbb{Z}$ be given. It follows from \eqref{eq:Delta(z)2} and \eqref{eq:periodicconditions} that
\begin{equation*}
    \varepsilon_k \Delta \bigg(z + \frac{2 \pi k i}{T} \bigg) \varepsilon_k^{-1} = \bigg[ \bigg( z + \frac{2 \pi k i}{T}\bigg)\varepsilon_k M \varepsilon_k^{-1} - \varepsilon_k K \varepsilon_k^{-1} \bigg] \varepsilon_k Q\bigg(z + \frac{2 \pi k i}{T}\bigg) \iota \varepsilon_k^{-1} = \Delta(z),
\end{equation*}
which is well-defined as $\varepsilon_{k}^{-1} = \varepsilon_{-k}$ leaves $\mathcal{D}(D)$ and $\mathcal{Y}$ invariant due to \ref{hyp:ek}.

To prove the second claim, let us first show that $\varepsilon_k$ leaves $\mathcal{D}(\mathcal{A})$ invariant. Since it leaves $\mathcal{D}(D)$ invariant due to \ref{hyp:ek}, it remains to show that $MD \varepsilon_k^{-1}\varphi = K \varepsilon_k^{-1} \varphi $ for all $k \in \mathbb{Z}$ and $\varphi \in \mathcal{D}(\mathcal{A})$. Let $\varphi \in \mathcal{D}(\mathcal{A})$ be given and note that $MD\varepsilon_k^{-1} \varphi =  \varepsilon_k^{-1} M(\frac{2 \pi k i}{T} + D)\varphi$. Since $MD \varphi = K \varphi$ by assumption and $K \varepsilon_{k}^{-1} = \varepsilon_k^{-1}(\frac{2 \pi k i}{T}M + K)$, the claim follows. Since the action of $\mathcal{A} $ is given by $D$, the first equality in \eqref{eq:periodicconditions} tells us that $\varepsilon_k \mathcal{A} \varepsilon_k^{-1} = \frac{2 \pi k i}{T} I + \mathcal{A}$, which proves the result.
\end{proof}
In view of the preceding result, we refer from now on to $\Delta$ as a $\frac{2\pi i}{T}\mathbb{Z}$\emph{-similar characteristic operator} for $\mathcal{A}$ on $\Omega$ if \ref{hyp:ek} and \eqref{eq:periodicconditions} are satisfied. By arguments analogous to those used in the proof of \cref{prop:Deltasimilar}, it follows as well that $\Delta$ is a $\frac{2\pi i}{T}\mathbb{Z}$-similar characteristic operator for $\hat{\mathcal{A}}$ on $\Omega$, since the mapping $z \mapsto zI-\hat{\mathcal{A}}$ is likewise $\frac{2\pi i}{T}\mathbb{Z}$-similar on $\Omega$. To establish this fact, one uses the map $\hat{\varepsilon}_k : \hat{\mathcal{F}}_T(\mathbb{R},X) \to \hat{\mathcal{F}}_T(\mathbb{R},X)$ defined by $ \hat{\varepsilon}_k (q,\varphi)\coloneqq(\varepsilon_k q,\varepsilon_k\varphi)$ for all $k \in \mathbb{Z}$.

\section{Applications to classes of periodic evolutionary systems} \label{sec:applications}
This section discusses three classes of periodic linear functional differential equations, for which the characteristic operator constructed in \cref{sec:characoperators} will provide us information on the Floquet exponents. All classes of periodic linear functional differential equations considered in this section can be written in the form
\begin{equation} \label{eq:FDE} \tag{FDE}
    \dot{x}(t) = L(t)x_t, \quad t \geq s,
\end{equation}
where $x(t) \in \mathbb{C}^n$, $s \in \mathbb{R}$ denotes a starting time and the \emph{history} $x_t$ is defined by $x_t(\theta) := x(t+\theta)$ for all $\theta \in I$, where $I \subseteq \mathbb{R}$ is a (given) closed interval. This setting indicates that the state space $X$ should be chosen as a function space of the form $\mathcal{F}(I,\mathbb{C}^n)$, while the reduction space $Y = \mathbb{C}^n$ is given by the codomain of $X$. Furthermore, the map $\mathbb{R} \ni t \mapsto L(t) \in \mathcal{L}(X,\mathbb{C}^n)$ is assumed to be $T$-periodic for some (minimal) $T \geq 0$, whereas the setting of $T=0$ will be important when we show that the class of autonomous FDEs is included in our construction (\cref{remark:autonomousDDE}). Since we are dealing with derivatives in \eqref{eq:FDE}, we have to assume in addition that $X$ admits a linear subspace $\mathcal{F}^1(I,\mathbb{C}^n)$ consisting of functions that admit a derivative in $X$. The type of derivative depends on the context, choice and underlying topology of the function space $X$. For example, when one works with the space of continuous functions $X = C(I,\mathbb{C}^n)$, then $\mathcal{F}^1(I,\mathbb{C}^n) = C^1(I,\mathbb{C}^n)$ denotes the space of continuously differentiable functions (strong derivative). Likewise, when one works with $X = L^p(I,\mathbb{C}^n)$ for some $1 \leq p \leq \infty$, then $\mathcal{F}^1(I,\mathbb{C}^n) = W^{1,p}(I,\mathbb{C}^n)$ denotes a Sobolev space (weak derivative). The regularity of the map $t \mapsto L(t)$ is closely related to the choice of $\mathcal{F}$ and will be made clear in each upcoming example. Regarding the choice of $\mathcal{F}$, we only impose the regularity condition that for each fixed $z \in \mathbb{C}$ the exponential map $(\theta \mapsto e^{z\theta}) \in \mathcal{F}^{1}(I,\mathbb{C}^n)$.

We call $\sigma \in \mathbb{C}$ a \emph{Floquet exponent} (of \eqref{eq:FDE}) if there exists a nonzero function $v$ of the form $v(t) = e^{\sigma t} q(t)$ with $q \in \mathcal{F}_T^1(\mathbb{R},\mathbb{C}^n)$ that solves \eqref{eq:FDE}. Hence, our problem of interest is equivalent to finding a \emph{Floquet pair} $(\sigma,q) \in \mathbb{C} \times \mathcal{F}_T^1(\mathbb{R},\mathbb{C}^n)$ satisfying the periodic linear FDE
\begin{equation} \label{eq:Floqexpgeneral}
    \dot{q}(t) + \sigma q(t) - L(t)[\theta \mapsto e^{\sigma \theta} q(t + \theta)] = 0, \quad \forall t \in \mathbb{R}.
\end{equation}
The function $v$ is called a \emph{Floquet (eigen)solution} associated to the Floquet exponent $\sigma$. In the presence of a (strongly continuous) forward evolutionary system $U \coloneqq \{U(t,s)\}_{t \geq s} \subseteq \mathcal{L}(X)$ for \eqref{eq:FDE}, one can prove that this notion of a Floquet exponent is equivalent to the notion of a Floquet exponent in terms of the $s$-independent spectral values $\lambda = e^{\sigma T}$ of the \emph{monodromy operator} $U(s+T,s)$, see \cite[Lemma 8.1.2]{Hale1993} for a proof in the setting of classical DDEs (\cref{subsec:classicalDDEs}). For systems that do not generate a forward evolutionary system (on $X$), such as mixed functional differential equations (MFDEs) (\cref{subsec:MFDEs}), one can still look for Floquet exponents and solutions by finding Floquet pairs $(\sigma,q)$ satisfying \eqref{eq:Floqexpgeneral}. Therefore, the definition of Floquet exponents in terms of \eqref{eq:Floqexpgeneral} is applicable and usable in a wider class of equations than the approach via $U$. To resemble a part of the abstract results from \cref{subsec:periodicspectral} on $\frac{2 \pi i }{T} \mathbb{Z}$-similarity, the following results shows that the Floquet exponents of \eqref{eq:FDE} are indeed uniquely determined up to additive multiples of $\frac{2 \pi i}{T}$.
\begin{proposition} \label{prop:Floquetexpunique}
If $\sigma \in \mathbb{C}$ is a Floquet exponent, then $\sigma + \frac{2 \pi i}{T}k$ is a Floquet exponent for all $k \in \mathbb{Z}$.
\end{proposition}
\begin{proof}
Let $\sigma$ be a Floquet exponent, then $t \mapsto e^{\sigma t}q(t)$ is a solution to \eqref{eq:FDE} where $q \in \mathcal{F}_T^1(\mathbb{R},\mathbb{C}^n)$. Let $k \in \mathbb{Z}$ be given and consider the map $q_k \in \mathcal{F}_T^1(\mathbb{R},\mathbb{C}^n)$ given by $q_k(t) = e^{-\frac{2 \pi i}{T}kt}q(t)$ for all $t \in \mathbb{R}$. Then the function $t \mapsto e^{(\sigma + \frac{2 \pi i}{T}k)t} q_k(t) = e^{\sigma t}q(t)$ is a solution to \eqref{eq:FDE}, meaning that $\sigma + \frac{2 \pi i}{T}k$ is a Floquet exponent. 
\end{proof}
The central result of this section is to establish the characterization
\begin{equation} \label{eq:sigmaAFloquet}
    \sigma(\mathcal{A}) \cap \Omega 
    = \{ \sigma \in \Omega : \sigma \text{ is a characteristic value of } \Delta \} 
    = \{ \sigma \in \Omega : \sigma \text{ is a Floquet exponent} \},
\end{equation}
for three representative subclasses of \eqref{eq:FDE}, namely classical DDEs (\cref{subsec:classicalDDEs}), iDDEs (\cref{subsec:infinitedelay}) and MFDEs (\cref{subsec:MFDEs}). Here, $\Omega = \mathbb{C}$ for classical DDEs and MFDEs, while for iDDEs it is a right half-plane in $\mathbb{C}$ containing the origin. In each case, the set $\Omega$ is sufficiently large to allow the application of the principle of linearized stability and to study local bifurcations. Moreover, a consequence of \cref{prop:Floquetexpunique} and \eqref{eq:sigmaAFloquet} is that it is sufficient to search for spectral values of $\mathcal{A}$, or equivalently for characteristic values of $\Delta$ within the horizontal strip $\{ z \in \Omega : \Im(z) \in (-\frac{\pi}{T},\frac{\pi}{T}] \}$. To prove \eqref{eq:sigmaAFloquet}, we proceed in four steps. First, we show that $\Delta$ is a closed characteristic operator for $\hat{\mathcal{A}}$ on $\Omega$ (\cref{thm:Delta equivalence}). Second, we verify that $\Delta$ satisfies \ref{hyp:rangeequal}. Third, combining these two key results, we conclude that $\Delta$ is a characteristic operator for $\mathcal{A}$ on $\Omega$ (\cref{thm:rhoAcharac}). Finally, we apply \cref{cor:spectralrelations} and \cref{lemma:compactspectra} to establish \eqref{eq:sigmaAFloquet}.

\subsection{Classical delay differential equations} \label{subsec:classicalDDEs}
In the setting of classical delay differential equations (DDEs),
the state space $X$ is the Banach space $C([-h,0],\mathbb{C}^n)$ equipped with the supremum norm $\| \cdot \|_{\infty}$. The (maximal) delay $h > 0$ is assumed to be finite, and the map $L$ is assumed to be in $C_T(\mathbb{R},\mathcal{L}(X,\mathbb{C}^n))$. It follows from a vector-valued version of the Riesz representation theorem that, for every $t \in \mathbb{R}$, the operator $L(t)$ can be represented uniquely as a Riemann-Stieltjes integral
\begin{equation} \label{eq:L(t)RieszDDE}
    L(t)\varphi = \int_0^h d_2 \zeta(t,\theta) \varphi(-\theta).
\end{equation}
Here, the map $\zeta : \mathbb{R} \times [0,h] \to \mathbb{C}^{n \times n}$ is a matrix-valued function, and $\zeta(\cdot,\theta)$ is continuous and $T$-periodic for all $\theta \in [0,h]$. Moreover, $\zeta(t,\cdot)$ is of bounded variation, right-continuous on the open interval $(0,h)$ for all $t \in \mathbb{R}$ and normalized by the requirement $\zeta(\cdot,0) = 0$. The notation $d_2 \zeta$ reflects that in the computation of the Riemann-Stieltjes integral, we integrate with respect to the second variable of $\zeta$. We also write $d_2 \zeta$ in front of the integrand since $\zeta$ is matrix-valued. For any starting time $s \in \mathbb{R}$, the initial value problem
\begin{equation} \label{eq:LinearDDE}
    \begin{dcases}
    \dot{x}(t) = L(t)x_t, \quad &t \geq s, \\
    x_s = \varphi, \quad &\varphi \in X,    
    \end{dcases}
\end{equation}
admits a unique solution $x$ on $[s,\infty)$ that is generated by a strongly continuous forward evolutionary system $U$ by $x_t = U(t,s)\varphi$ for all $t \geq s$, see \cite[Theorem XII.2.7 and XII.3.1]{Diekmann1995}. The associated (generalized) generator of $U$ (at time $t \in \mathbb{R}$) is given by the closed linear operator $A(t) : \mathcal{D}(A(t)) \subseteq X \to X$ with domain and action
\begin{equation*}
    \mathcal{D}(A(t)) = \{ \varphi \in C^{1}([-h,0],\mathbb{C}^n) : \varphi '(0) = L(t) \varphi \}, \quad A(t)\varphi = \varphi',
\end{equation*}
see also \cite{Article2,Clement1988} for more information. When we take the derivative with respect to the state variable of a function $\varphi \in C^1([-h,0],\mathbb{C}^n)$, we denote this by an accent, i.e. $\varphi' \in X$. However, when we take the derivative with respect to the time variable of a function $\varphi \in C_T^1(\mathbb{R},X)$, we denote this by a dot, i.e. $\dot{\varphi} \in C_T(\mathbb{R},X)$. To this operator, let us recall from \eqref{eq:introcurlyA} that we associate the linear operator $\mathcal{A} : \mathcal{D}(\mathcal{A}) \subseteq C_T(\mathbb{R},X) \to C_T(\mathbb{R},X)$ defined by 
\begin{equation} \label{eq:curlya_dde}
     \mathcal{D}(\mathcal{A}) = \{ \varphi \in C_T^1(\mathbb{R},X) : \varphi(t) \in \mathcal{D}(A(t)) \mbox{ for all $t \in \mathbb{R}$} \}, \quad (\mathcal{A}\varphi)(t) = A(t)\varphi(t) - \dot{\varphi}(t). 
\end{equation}
As we will show in \cref{lemma:operatorlemmaDDE}, the operator $\mathcal{A}$ from \eqref{eq:curlya_dde} coincides with the closable linear operator $\mathcal{A}$ introduced in \eqref{eq:D(A)rep} from \cref{sec:construction}. The purpose of this section is to construct a characteristic operator $\Delta$ for $\mathcal{A}$ and, on this basis, to establish several properties of its spectrum. In particular, we will prove that the spectrum of $\mathcal{A}$ consists solely of isolated eigenvalues of finite type and satisfies \eqref{eq:sigmaAFloquet}. To do so, we first construct operators $D, K$ and $M$ so that $\mathcal{A}$ is the first operator associated to $D, K$ and $M$.

\begin{lemma} \label{lemma:operatorlemmaDDE}
Let $D : \mathcal{D}(D) \to C_T(\mathbb{R},X)$ be the operator defined by
\begin{equation} \label{eq:domainDDDE}
    \mathcal{D}(D) = \{ \varphi \in C_T^1(\mathbb{R},X) : \varphi(t)' \in X \mbox{ for all $t \in \mathbb{R}$} \}, \quad (D\varphi)(t) = \varphi(t)' - \dot{\varphi}(t).
\end{equation}
Then $D$ is a closable linear operator satisfying \ref{hyp:iota}, \ref{hyp:restriction}, \ref{hyp:domainrangeD} and \ref{hyp:domaincurlyA} with $\Omega = \mathbb{C}$. Moreover, define the linear operators $K : \mathcal{D}(K) \subseteq C_T(\mathbb{R},X) \to C_T(\mathbb{R},\mathbb{C}^n)$ and $M : C_T(\mathbb{R},X) \to C_T(\mathbb{R},\mathbb{C}^n)$ by
\begin{equation} \label{eq:KMdde}
     \mathcal{D}(K) = C_T^1(\mathbb{R},X), \qquad (K\varphi)(t) = L(t) \varphi(t) - \dot{\varphi}(t)(0), \qquad (M\varphi)(t) = \varphi(t)(0), 
\end{equation}
then $K$ satisfies \ref{hyp:K}, and $\mathcal{A}$ defined in \eqref{eq:curlya_dde} is the first operator associated to $D, K$ and $M$.  
\end{lemma}
\begin{proof}
Put $I = [-h,0], Y = \mathbb{C}^n$ and $X =  C([-h,0],\mathbb{C}^n)$ in the construction of \cref{subsec:construction}. We first show that the operator $D$ is closable and satisfies the hypotheses \ref{hyp:iota}, \ref{hyp:restriction}, \ref{hyp:domainrangeD} and \ref{hyp:domaincurlyA}. To show that the linear operator $D$ is closable, let $(\varphi_m)_m$ be a sequence in $\mathcal{D}(D)$ converging in norm to zero and assume that the sequence $(D\varphi_m)_m$ converges in norm to some $\psi \in C_T(\mathbb{R},X)$. We have to show that $\psi = 0$. Set $\psi_m = D\varphi_m$ for every $m \in \mathbb{N}$ and notice that this equation is equivalent to
\begin{equation*}
    \frac{\partial}{\partial \theta} \varphi_m(t)(\theta) - \frac{\partial}{\partial t} \varphi_m(t)(\theta) = \psi_m(t)(\theta),  \quad \forall t \in \mathbb{R}, \ \theta \in [-h,0].
\end{equation*}
This is an inhomogeneous transport equation that has the solution
\begin{equation*}
    \varphi_m(t)(\theta) = \varphi_m(t+\theta)(0) + \int_0^\theta \psi_m(t+\theta-s)(s) ds.
\end{equation*}
Since the convergence of $\varphi_m \to 0$ and $\psi_m \to \psi$ is uniform in the first and second variable, we obtain
\begin{align}
    \begin{split} \label{eq:integralpsi}
    \int_0^\theta  \psi(t+\theta-s)(s) ds &= \lim_{m \to \infty} \int_0^\theta  \psi_m(t+\theta-s)(s) ds \\
    &= \lim_{m \to \infty} [\varphi_m(t)(\theta) - \varphi_m(t+\theta)(0)] = 0,
    \end{split}
\end{align}
where in the last step we used that $\varphi_m \to 0$ uniformly as $m \to \infty$. Let $t_1 \in \mathbb{R}$ and $\theta_1,\theta_2 \in [-h,0]$ with $\theta_1 \neq \theta_2$ be given. If $t_2 = t_1 + \theta_1 - \theta_2$, then according to \eqref{eq:integralpsi} we obtain
\begin{align*}
    \int_{\theta_1}^{\theta_2} \psi(t_1+\theta_1-s)(s) ds &= \int_{0}^{\theta_1} \psi(t_1+\theta_1-s)(s) ds + \int_{\theta_1}^{\theta_2} \psi(t_1+\theta_1-s)(s) ds \\
    &= \int_{0}^{\theta_2} \psi(t_1+\theta_1-s)(s) ds = \int_{0}^{\theta_2} \psi(t_2+\theta_2-s)(s) ds = 0
\end{align*}
since the second and last integral vanish due to \eqref{eq:integralpsi}. Differentiating the left-hand side with respect to $\theta_2$ yields $\psi(t_2)(\theta_2) = 0$ for all $\theta_2 \in [-h,0]$. Since $t_1$ and $\theta_1$, and thus $t_2$, were arbitrary, we conclude that $\psi = 0$, i.e. $D$ is closable.

To verify that the operator $D$ satisfies \ref{hyp:iota}, we first prove that $\mathcal{N} (D) \neq  \{0\}$. The kernel of $D$ consists of functions $\varphi \in \mathcal{D}(D)$ satisfying the transport equation
\begin{equation*}
    \frac{\partial}{\partial t} \varphi(t)(\theta) = \frac{\partial}{\partial \theta} \varphi(t)(\theta), \quad \forall t \in \mathbb{R}, \ \theta \in [-h,0],
\end{equation*}
which is equivalent to saying that $\varphi$ has the translation property $\varphi(t)(\theta) = \varphi(t+\theta)(0)$ for all $t \in \mathbb{R}$ and $\theta \in [-h,0]$. Hence 
\begin{equation*}
\mathcal{N} (D) = \{ \varphi \in \mathcal{D}(D) : \varphi(t)(\theta) = \varphi(t+\theta)(0) \mbox{ for all }  t \in \mathbb{R} \mbox{ and } \theta \in [-h, 0]\}  \neq \{0\}.
\end{equation*}
Moreover, if we set $\mathcal{Y} = C_T^1(\mathbb{R},\mathbb{C}^n)$
and define the linear operator $\iota : \mathcal{Y} \to \mathcal{N}(D)$ by
\begin{equation} \label{eq:iotaDDE}
    (\iota q)(t)(\theta) = q(t+\theta),
\end{equation}
then $\iota$ is clearly a bijection with (algebraic) inverse $\iota^{-1} : \mathcal{N}(D) \to \mathcal{Y}$ given by $(\iota^{-1}\varphi)(t) = \varphi(t)(0)$. Clearly, $\iota$ and $\iota^{-1}$ are bounded and thus \ref{hyp:iota} is satisfied.

Our next aim is to verify that \ref{hyp:restriction} holds with $\Omega = \mathbb{C} \ni 0$. Therefore, let us first introduce the linear operator $D_0 : \mathcal{D}(D_0) \to C_T(\mathbb{R},X)$ by
\begin{equation*}
    \mathcal{D}(D_0) = \{ \varphi \in \mathcal{D}(D) : \varphi(t)(0) = 0 \mbox{ for all } t \in \mathbb{R} \}, \quad  D_0\varphi = D\varphi,
\end{equation*}
and recall from the argument detailed above \cref{lemma:rangeinclusion} that $D_0$ is closable. We claim that $\mathcal{D}(D) = \mathcal{N} (D) \oplus \mathcal{D}(D_0)$. Let $\varphi \in \mathcal{N} (D) \cap \mathcal{D}(D_0)$, then $\varphi(t)(\theta) = \varphi(t+\theta)(0) = 0$ for all $t \in \mathbb{R}$ and $\theta \in [-h,0]$ which proves that $\varphi = 0$ and thus $\mathcal{N}(D) \cap \mathcal{D}(D_0) = \{0 \}$. Let $\varphi \in \mathcal{D}(D)$ be given and define $\psi \in \mathcal{D}(D)$ by $\psi(t)(\theta) = \varphi(t)(\theta) - \varphi(t+\theta)(0)$. Then $\psi(\cdot)(0) = 0$ and so $\psi \in \mathcal{D}(D_0)$. Define now $\phi \in \mathcal{D}(D)$ by $\phi(t)(\theta) = \varphi(t+\theta)(0)$, then $\phi \in \mathcal{N} (D)$ and since $\varphi = \psi + \phi$, the claim is proven. To verify that $\rho(D_0) = \mathbb{C} \neq \emptyset$, we start by studying solutions of
\begin{equation} \label{eq:resolventD01}
    (zI -D_0)\varphi = \psi
\end{equation}
with $z \in \mathbb{C}, \varphi \in \mathcal{D}(D_0)$ and $\psi \in C_T^1(\mathbb{R},X)$. Notice that \eqref{eq:resolventD01} is equivalent to solving the inhomogeneous transport equation with boundary condition
\begin{equation*}
    \bigg( z + \frac{\partial}{\partial t} - \frac{\partial}{\partial \theta}  \bigg) \varphi(t)(\theta) = \psi(t)(\theta), \quad \varphi(t)(0) = 0, \quad \forall t \in \mathbb{R}, \ \theta \in [-h,0],
\end{equation*}
which has the unique solution 
\begin{equation} \label{eq:resolventD02}
    \varphi(t)(\theta) = \int_\theta^0 e^{z(\theta - s)} \psi(t+\theta -s)(s) ds,
\end{equation}
if we can prove that $\varphi \in \mathcal{D}(D_0)$. To show this, let us first observe that
\begin{align}
\begin{split} \label{eq:derivativesresolventDDE}
    \dot{\varphi}(t)(\theta) &= \int_\theta^0 e^{z(\theta - s)} \dot{\psi}(t+\theta -s)(s) ds, \\
    \varphi(t)'(\theta) &= - \psi(t)(\theta) + \int_\theta^0 e^{z(\theta - s)} [z \psi(t+\theta-s)(s) + \dot{\psi}(t+\theta-s)(s)] ds,
\end{split}
\end{align}
where the last expression follows from the Leibniz integral rule. Since $\psi \in C_T^1(\mathbb{R},X)$, the first equality in \eqref{eq:derivativesresolventDDE} implies that $\varphi \in C_T^1(\mathbb{R},X)$ and the second equality in \eqref{eq:derivativesresolventDDE} implies that $\varphi(t)' \in X$ for all $t \in \mathbb{R}$, which shows that $\varphi \in \mathcal{D}(D_0)$. As a consequence, $C_T^1(\mathbb{R},X) \subseteq \mathcal{R}(zI - D_0) \subseteq C_T(\mathbb{R},X)$ and so $zI - D_0$ has dense range. In \cref{sec:nonclosed}, we actually prove that these inclusions are (in general) strict, see \cref{prop:rangenotequal} for further details. Moreover, for every $z \in \mathbb{C}$ the linear operator $zI - D_0$ has a densely defined linear inverse given by
\begin{equation} \label{eq:resolventD0}
    [R(z,D_0) \varphi](t)(\theta) = \int_\theta^0 e^{z(\theta - s)} \varphi(t+\theta-s)(s) ds, \quad \forall \varphi \in \mathcal{R}(zI - D_0),
\end{equation}
and this operator is bounded, as it holds that
\begin{equation} \label{eq:resolventD0Mz}
    \|R(z,D_0) \varphi \|_{\infty} \leq M_z \| \varphi \|_{\infty}, \quad M_z \coloneqq \frac{1 - e^{-h \Re(z)}}{\Re(z)},
\end{equation}
where $M_z$ is understood as its limiting value $h$ whenever $\Re(z) = 0$. As $0 \in \Omega = \mathbb{C}$, we have that $D_0^{-1}$ is not only an (algebraic) inverse of $D_0$, but in fact an inverse of $D_0$ in terms of \cref{def:resolvent}. We conclude that \ref{hyp:restriction} holds.

It follows directly from \eqref{eq:domainDDDE} and the above argumentation that $\mathcal{D}(D) \subseteq C_T^1(\mathbb{R},X) \subseteq \mathcal{R}(zI-D_0)$ for all $z \in \mathbb{C}$, which shows that \ref{hyp:domainrangeD} is satisfied. 

While \cref{cor:spectralrelations} is proven after the introduction of \ref{hyp:domaincurlyA}, the formula \eqref{eq:resolventADelta} itself is independent of this hypothesis. Hence, $\mathcal{R}(zI-\mathcal{A}) = \mathcal{R}(zI-D_0)$ for all $z \in \rho(\mathcal{A}) \cap \Omega$, which proves that \ref{hyp:domaincurlyA} is satisfied as a consequence of \ref{hyp:domainrangeD}.

Our next aim is to verify \ref{hyp:K}. To do this, consider the multiplication operator $M_L : C_T(\mathbb{R},X) \to C_T(\mathbb{R},\mathbb{C}^n)$ with action $(M_L \varphi)(t) \coloneqq L(t) \varphi(t)$ and the standard differential operator $G : \mathcal{D}(G) \subseteq C_T(\mathbb{R},\mathbb{C}^n) \to C_T(\mathbb{R},\mathbb{C}^n)$ with domain $\mathcal{D}(G) \coloneqq C_T^1(\mathbb{R},\mathbb{C}^n)$ and action $Gq \coloneqq \dot{q}$. Then we can write $K = M_L - GM$. Since $G$ is a closed and $M$ is bounded linear operator, it follows from \cref{lemma:compositionclosed} that $GM$, with domain $\mathcal{D}(GM) = \{ \varphi \in C_T(\mathbb{R},X) : M\varphi \in \mathcal{D}(G) \} \supseteq C_T^1(\mathbb{R},X)$, is closed. Since $M_L$ is bounded, it follows from \cref{lemma:closedsum} that the operator $K$, with (possibly) smaller domain $\mathcal{D}(K) = C_T^1(\mathbb{R},X) \subseteq \mathcal{D}(GM)$ satisfying $\mathcal{D}(K) \supseteq \mathcal{D}(D)$ is closable. Moreover, the restriction $K|_{\mathcal{D}(D_0)} = M_L |_{\mathcal{D}(D_0)}$, since $\mathcal{D}(D_0) \subseteq \mathcal{N}(M)$, is bounded as $M_L$ is bounded. This shows that \ref{hyp:K} is satisfied.

We conclude from \cref{lemma:D0closable} that the linear operator $\mathcal{A}$ is closable and its representation follows from \eqref{eq:D(A)rep}. Therefore, $\mathcal{A}$ is the first operator associated with $D,K$ and $M$.
\end{proof}
The second operator $\hat{\mathcal{A}}$ associated with $D, K$ and $M$ now takes the form 
\begin{equation*}
    \mathcal{D}(\hat{\mathcal{A}}) = \bigg \{ \begin{pmatrix}
    q\\
    \varphi
    \end{pmatrix}
    \in \hat{C}_T^1(\mathbb{R},X) : \varphi \in \mathcal{D}(D), \ q = M\varphi \bigg \}, \quad 
    \hat{\mathcal{A}}
    \begin{pmatrix}
    q\\
    \varphi
    \end{pmatrix}
    =
    \begin{pmatrix}
    K\varphi \\
    D\varphi
    \end{pmatrix}.
\end{equation*}
To confirm our need on the construction of the framework in \cref{sec:characoperators} regarding closable, not necessarily closed linear operators, it is proven in \cref{sec:nonclosed} that the linear operators $D, D_0, \mathcal{A}$ and $\hat{\mathcal{A}}$ are not closed, see \cref{prop:DAnotclosed} for the main result. Moreover, we illustrate in \cref{sec:nonclosed} that in general $C_T^1(\mathbb{R},X) \neq \mathcal{R}(zI-D_0) \neq C_T(\mathbb{R},X)$ for all $z \in \mathbb{C}$, which highlights the necessity of the dense range definition of the resolvent, recall \cref{remark:definitionresolvent}. Similar results also hold for the linear operators $D, D_0, \mathcal{A}$ and $\hat{\mathcal{A}}$ in the upcoming two subsections.

\begin{theorem} \label{thm:charmatrixDDE}
The holomorphic operator-valued function $\Delta : \mathbb{C} \to L(C_T(\mathbb{R},\mathbb{C}^n))$ with values $\Delta(z) : C_T^1(\mathbb{R},\mathbb{C}^n) \to C_T(\mathbb{R},\mathbb{C}^n)$ given by
\begin{equation} \label{eq:DeltaDDE}
     (\Delta(z)q)(t) = \dot{q}(t) + z q(t) - \int_0^h d_2 \zeta(t,\theta) e^{-z \theta} q(t-\theta),
\end{equation}
is a $\frac{2\pi i}{T}\mathbb{Z}$-similar closed characteristic operator for $\hat{\mathcal{A}}$ on $\mathbb{C}$ and the equivalence is given by
\begin{equation} \label{eq:charoperatorAhatDDE}
     \begin{pmatrix}
         \Delta(z) & 0 \\
         0 & I 
     \end{pmatrix}
     =
     F(z)(z I - \hat{\mathcal{A}})E(z), \quad \forall z \in \mathbb{C},
 \end{equation}
where $E : \mathbb{C} \to L(\hat{C}_T(\mathbb{R},X),\hat{C}_T(\mathbb{R},X)_{\hat{\mathcal{A}}})$ and $F : \mathbb{C} \to L(\hat{C}_T(\mathbb{R},X))$ are holomorphic bounded operator-valued functions, such that $E(z) \in \mathcal{L}(C_T^1(\mathbb{R},\mathbb{C}^n) \oplus \mathcal{R}(zI-D_0), \mathcal{D}(\hat{\mathcal{A}}))$ and $F(z) \in \mathcal{L}(C_T(\mathbb{R},\mathbb{C}^n) \oplus \mathcal{R}(zI-D_0))$ are bijective mappings, given by
\begin{alignat*}{2}
     E(z)
     \begin{pmatrix}
     q\\
     \varphi
     \end{pmatrix}
     &=
     \begin{pmatrix}
     q\\
     \psi
     \end{pmatrix}, \quad
         \psi(t)(\theta) &&= e^{z \theta}q(t+\theta)+ \int_{\theta}^0 e^{z(\theta-s)}\varphi(t+\theta-s)(s) ds, \\
     F(z)
     \begin{pmatrix}
     q\\
     \varphi
     \end{pmatrix}
     &=
     \begin{pmatrix}
     q + \phi\\
     \varphi
     \end{pmatrix}, \quad
         \phi(t) &&= \int_0^h d_2\zeta(t,\theta) \int_{-\theta}^0 e^{-z(\theta+s)}\varphi(t-\theta-s)(s) ds.
 \end{alignat*}
\end{theorem}
\begin{proof}
We start by proving that the operator-valued function defined in \eqref{eq:DeltaDDE} is a characteristic operator for $\hat{\mathcal{A}}$. Let $q \in \mathcal{Y}$ be given and note for all $t \in \mathbb{R}$ and $\theta \in [-h,0]$ that
\begin{equation} \label{eq:Q(z)DDE}
    (Q(z)\iota q)(t)(\theta) = (\iota q - zR(z,D_0) \iota q)(t)(\theta) 
    = q(t+\theta) - z \int_{\theta}^0 e^{z (\theta-s)} q(t+\theta) ds = e^{z \theta}q(t+\theta),
\end{equation}
where we used \eqref{eq:iotaDDE} and \eqref{eq:resolventD0}. Recalling \eqref{eq:KMdde} yields
\begin{align*}
    (\Delta(z)q)(t) &= ((zM-K)Q(z)\iota q)(t) \\
    &= zq(t) -  \bigg( L(t)[\theta \mapsto e^{z \theta} q(t+\theta)] - \dot{q}(t) \bigg) = \dot{q}(t) + zq(t) - \int_0^h d_2 \zeta(t,\theta) e^{- z \theta} q(t-\theta),
\end{align*}
where we used \eqref{eq:Delta(z)2} in the first equality and \eqref{eq:L(t)RieszDDE} in the last equality. 

We next prove that $\Delta(z)$ is closed. Fix $z \in \mathbb{C}$ and note that we can write $\Delta(z) = G + N(z)$, where $N(z) : C_T(\mathbb{R},\mathbb{C}^n) \to C_T(\mathbb{R},\mathbb{C}^n) $ is the linear operator defined by $(N(z)q)(t) \coloneqq zq(t) - L(t)[\theta \mapsto e^{z \theta}q(t+\theta)]$. Note that $N(z)$ is bounded since
\begin{equation*}
    \|N(z)q\|_{\infty} \leq \big(|z| + \sup_{\theta \in [-h,0]} |e^{z\theta}| \|L\|_\infty \big)  \| q \|_{\infty} \leq \big( |z| + \max\{1,e^{-h \Re (z)} \} \|L\|_\infty  \big) \| q \|_{\infty}.
\end{equation*}
As $\Delta(z)$ is the sum of the closed operator $G$, with domain $\mathcal{Y} = C_T^1(\mathbb{R},\mathbb{C}^n)$, and the bounded operator $N(z)$, it follows from \cref{lemma:closedsum} that $\Delta(z)$ is closed.

Since $\hat{\mathcal{A}}$ from \eqref{eq:D(Ahat)rep} is the second operator associated with $D,K$ and $M$, we can apply \cref{thm:Delta equivalence}. We conclude that $\Delta$ is a characteristic operator for $\hat{\mathcal{A}}$ on $\mathbb{C}$ and the concrete representations of $E$ and $F$ can be verified directly using the explicit representations of $Q(z), R(z,D_0), M$ and $K$ mentioned earlier in this subsection.

Let us now prove that $\Delta$ is a $\frac{2 \pi i}{T}\mathbb{Z}$-similar characteristic operator for $\hat{\mathcal{A}}$ on $\mathbb{C}$. Since $t \mapsto e^{\frac{2 \pi k i}{T}t}$ is analytic on $\mathbb{R}$, it is clear that \ref{hyp:ek} holds. Moreover, a straightforward calculation using \eqref{eq:KMdde} and \eqref{eq:Q(z)DDE} shows that the relations from \eqref{eq:periodicconditions} are satisfied, which proves the claim.
\end{proof}

In order to establish that $\Delta$ is a characteristic operator for $\mathcal{A}$ on $\mathbb{C}$, we need the following result.

\begin{lemma} \label{lemma:characinvertible}
The characteristic operator $\Delta$ from \eqref{eq:DeltaDDE} satisfies \ref{hyp:rangeequal} and $\rho(\Delta(z)) \neq \emptyset$ for all $z \in \mathbb{C}$.
\end{lemma}
\begin{proof}
Let us first prove the second claim, i.e. for every $z \in \mathbb{C}$ there exists a $\mu \in \mathbb{C}$ such that $\mu I - \Delta(z)$ is invertible. To do so, consider the linear operators $S(z) : \mathcal{D}(S(z)) \subseteq C_T(\mathbb{R},\mathbb{C}^n) \to C_T(\mathbb{R},\mathbb{C}^n)$, with domain $\mathcal{D}(S(z)) \coloneqq C_T^1(\mathbb{R},\mathbb{C}^n) $, and $C(z) : C_T(\mathbb{R},\mathbb{C}^n) \to C_T(\mathbb{R},\mathbb{C}^n)$ defined by 
\begin{equation} \label{eq:defn_S_C}
    (S(z)q)(t) \coloneqq z q(t) - \dot{q}(t), \qquad (C(z)q)(t) \coloneqq - L(t)[\theta \mapsto e^{z \theta}q(t+\theta)],
\end{equation}
and note that we can write
\begin{equation} \label{eq:mIDelta}
    \mu I - \Delta(z) = S(\mu - z) - C(z).
\end{equation}
The operator $C(z)$ is bounded since $\|C(z)\| \leq \max\{1,e^{-h \Re (z)} \} \|L\|_\infty$ as $L \in C_T(\mathbb{R},\mathcal{L}(X,\mathbb{C}^n))$. Additionally, one can verify that the inverse of the closed linear operator $S(z)$ is given by 
\begin{equation} \label{eq:S(z)inverse}
    (S(z)^{-1}q)(t) = \frac{1}{e^{z T}-1} \bigg( \int_0^t e^{z(t-s)}q(s) ds + \int_t^T e^{z(t+T-s)}q(s) ds\bigg), \quad \forall q \in C_T(\mathbb{R},\mathbb{C}^n),
\end{equation}
whenever $z \in \mathbb{C} \setminus \hspace{-2pt}\frac{2 \pi i }{T} \mathbb{Z}$, and that for $z \in \mathbb{R} \setminus \{0\}$ one has $\|S(z)^{-1}\| \leq 1/|z| \to 0$ as $z \to \pm \infty$. Hence, for every fixed $z \in \mathbb{C}$, there exists a $\mu \in \mathbb{C} \setminus (z + \frac{2 \pi i }{T} \mathbb{Z} )$ such that $\|S(\mu - z)^{-1}\| < \|C(z)\|^{-1}$. Thus, $I - C(z)S(\mu - z)^{-1}$ is an invertible operator in $\mathcal{L}(C_T(\mathbb{R},\mathbb{C}^n))$ by the Neumann series. Hence, $[I - C(z)S(\mu - z)^{-1}]S(\mu - z) = \mu I - \Delta(z)$ is invertible and this shows that $\mu \in \rho(\Delta(z))$, i.e. $\rho(\Delta(z)) \neq \emptyset$.

To verify \ref{hyp:rangeequal}, we have to prove that $\rho(\Delta) \neq \emptyset$ and that there exists a $z_0 \in \rho(\Delta)$ with $\mathcal{R}(\Delta(z_0)) = \mathcal{R}(MD-K)$. To prove the first claim, we can set $\mu = 0$ in the construction above and take a $z_0 \in \mathbb{R}$ with $z_0 > 0$ sufficiently large to obtain $\| S(-z_0)^{-1} \| <  \|L\|_\infty^{-1} \leq \|C(z_0)\|^{-1}$, where the last inequality follows from the fact that $z_0 > 0$. Hence, $-[I - C(z_0)S(-z_0)^{-1}]S(-z_0) = \Delta(z_0)$ is invertible and thus $z_0 \in \rho(\Delta)$, proving that $\rho(\Delta) \neq \emptyset$. To prove the second claim, recall that $\Delta(z_0)$ is closed, which proves that $\mathcal{R}(\Delta(z_0)) = C_T(\mathbb{R},\mathbb{C}^n) = \mathcal{R}(MD-K)$ by \cref{lemma:T2onto}.
\end{proof}

By recalling \cref{thm:rhoAcharac} and combining the two preceding results, we arrive at the following

\begin{corollary}
The holomorphic operator-valued function $\Delta$ from \eqref{eq:DeltaDDE} is a $\frac{2\pi i}{T}\mathbb{Z}$-similar closed characteristic operator for $\mathcal{A}$ on $\mathbb{C}$.
\end{corollary}

Having obtained a characteristic operator $\Delta$ for $\mathcal{A}$, we can next apply \cref{cor:spectralrelations} and characterise the spectrum of $\mathcal{A}$ in terms of the characteristic values of $\Delta$.

\begin{theorem} \label{thm:eigenfunctionsDDE}
The spectrum of $\mathcal{A}$ consists solely of isolated eigenvalues of finite type:
\begin{equation} \label{eq:spectrumDDE}
    \sigma(\mathcal{A}) =  \{ \sigma \in \mathbb{C} : \sigma \textit{ is a characteristic value of } \Delta \} = \{\sigma \in \mathbb{C} : \sigma \mbox{ is a Floquet exponent} \}.
\end{equation}
For $\sigma \in \sigma(\mathcal{A})$, there holds $m_g(\sigma,\mathcal{A}) = m_g(\Delta(\sigma)), m_a(\sigma,\mathcal{A}) = m_a(\Delta(\sigma)), k(\sigma,\mathcal{A}) = k(\Delta(\sigma))$ and $r(\sigma,\mathcal{A}) = r(\Delta(\sigma))$, and the partial multiplicities of $\sigma$ considered as an eigenvalue of $\mathcal{A}$ are equal to the zero multiplicities of $\sigma$ considered as a characteristic value of $\Delta$. Furthermore, if $\{ \{{q}_{i,0},\dots,{q}_{i,k_i-1} \} : i=1,\dots,p\}$ is a canonical system of Jordan chains of $\Delta$ at $\sigma$, then $\{ \{\varphi_{i,0},\dots,\varphi_{i,k_i-1} \} : i=1,\dots,p \}$ is a canonical system of (generalized) eigenvectors of $\mathcal{A}$ at $\sigma$, where
\begin{equation*}
    \varphi_{i,k}(t)(\theta) = e^{\sigma \theta} \sum_{l=0}^k \frac{\theta^l}{l!}q_{i,k-l}(t+\theta), \quad \forall k = 0,\dots,k_{i}-1, \ t \in \mathbb{R}, \ \theta \in [-h,0].
\end{equation*}
\end{theorem}
\begin{proof}
With the aim of applying \cref{lemma:compactspectra}, we next prove that the operator $\Delta(z)$ has compact resolvent for all $z \in \mathbb{C}$. To do so, consider the operators $S$ and $C$ from \eqref{eq:defn_S_C}. Let us first observe that $S(z)^{-1}$ from \eqref{eq:S(z)inverse} is compact for all $z \in \mathbb{C} \setminus \frac{2 \pi i }{T} \mathbb{Z}$ since it maps continuous functions towards $C^1$-smooth functions. Recalling from the proof of \cref{lemma:characinvertible} that $C(z)$ is bounded and $\rho(\Delta(z)) \neq \emptyset$, the equality \eqref{eq:mIDelta} in combination with \cite[Exercise II.4.30]{Engel2000} shows that $R(\mu,\Delta(z))$ is compact for all $\mu \in \mathbb{C} \setminus (z+\frac{2 \pi i}{T} \mathbb{Z})$, which proves the claim. \cref{lemma:compactspectra} tells us that $\sigma(\mathcal{A})$ consists solely of isolated eigenvalues of finite type that are precisely the characteristic values of $\Delta$. But it follows from \eqref{eq:Floqexpgeneral} that a characteristic value is a Floquet exponent, which proves the second equality in \eqref{eq:spectrumDDE}. The claim regarding the multiplicities follows from \cref{cor:spectralrelations}.

To prove the remaining claim regarding the Jordan chains, one can apply directly \cref{cor:Jordanchains} and use \eqref{eq:resolventD0} to obtain an explicit representation of $R(\sigma,D_0)^{l+1}$ for sufficiently large $l \in \mathbb{N}$. However, the proof we present next is somewhat shorter and more insightful. Let $\{{q}_{0},\dots,{q}_{k-1} \}$ be a Jordan chain of $\Delta$ at $\sigma$. From the expression for the operator-valued function $E$ defined in \cref{thm:charmatrixDDE}, we derive that
\begin{equation*}
E(z)
    \begin{pmatrix}
        \sum_{i=0}^{k-1} (z-\sigma)^{i} q_i \\
        0
    \end{pmatrix}
    =
    \begin{pmatrix}
        \sum_{i=0}^{k-1} (z-\sigma)^{i} q_i \\
        t \mapsto e^{z \cdot} \sum_{i=0}^{k-1} (z-\sigma)^i q_i(t+\cdot)
    \end{pmatrix}, 
\end{equation*}
and thus it follows from the equivalence \eqref{eq:charoperatorAhatDDE} that $z \mapsto (t \mapsto e^{z \cdot} \sum_{i=0}^{k-1} (z-\sigma)^i q_i(t+\cdot))$ is a root function for $zI - \hat{\mathcal{A}}$ at $\sigma$. From the proof of \cref{thm:spectralequal}, it follows that we have to expand $e_z \coloneqq (\theta \mapsto e^{z \theta})$ up to order $k$ in a neighbourhood of $\sigma$. Since
\begin{equation*}
    e_z(\theta) = e^{\sigma \theta} e^{(z-\sigma)\theta} = e^{\sigma \theta} \bigg( \sum_{l=0}^{k-1} \frac{\theta^{l}}{l!}(z-\sigma)^l + \mathcal{O}((z-\sigma)^k) \bigg),
\end{equation*}
we have for all $t \in \mathbb{R}$ and $\theta \in [-h,0]$ that 
\begin{equation*}
    e^{z \theta} \sum_{i=0}^{k-1} (z-\sigma)^i q_i(t+\theta) = \sum_{i=0}^{k-1} \bigg( e^{\sigma \theta} \sum_{l=0}^i \frac{\theta^l}{l!}q_{i-l}(t+\theta) \bigg) (z-\sigma)^i + \mathcal{O}((z-\sigma)^k),
\end{equation*}
which shows that the ordered set of vectors $\{\varphi_0,\dots,\varphi_{k-1} \}$, where $\varphi_i(t)(\theta) = e^{\sigma \theta} \sum_{l=0}^i \frac{\theta^l}{l!}q_{i-l}(t+\theta)$, forms a Jordan chain of $\mathcal{A}$ at $\sigma$.
\end{proof}

\begin{corollary} \label{cor:resolventDDE}
If $z \in \mathbb{C}$ is not a Floquet exponent, then the resolvent of $\mathcal{A}$ at $z$ has the following representation
\begin{equation*}
    (R(z,\mathcal{A}) \varphi)(t)(\theta) = e^{z \theta} \varphi_z(t+\theta) + \int_\theta^0 e^{z(\theta-s)} \varphi(t+\theta-s)(s) ds, \quad \forall \varphi \in \mathcal{R}(zI-\mathcal{A}) \supseteq C_T^1(\mathbb{R},X),
\end{equation*}
where $\varphi_z$ is given by
\begin{equation} \label{eq:varphizDelta}
    \varphi_z = \Delta(z)^{-1} \bigg( t \mapsto \varphi(t)(0) + \int_0^h d_2 \zeta(t,\theta)\int_{-\theta}^{0} e^{-z(\theta + s)} \varphi(t -\theta - s)(s) ds \bigg).
\end{equation}
\end{corollary}
\begin{proof}
Recall from \eqref{eq:spectrumDDE} that $z \in \rho(\mathcal{A})$ if and only if $z \in \rho(\Delta)$ if and only if $z$ is not a Floquet exponent. The representation of $R(z,\mathcal{A})$ follows now from \eqref{eq:resolventADelta}, where we recall the representation of $Q(z)\iota$ and $R(z,D_0)$ from \eqref{eq:Q(z)DDE} and \eqref{eq:resolventD0}, respectively.
\end{proof}

\begin{remark}
One could also use \cite[Proposition 12]{Bosschaert2025} to prove \cref{cor:resolventDDE}. However, using this approach, one must prove in addition that $R(z,\mathcal{A})$ is a densely defined bounded linear operator and that $\sigma(\mathcal{A})$ actually coincides with the set of Floquet exponents. To prove the first claim, one uses the fact that $C_T^1(\mathbb{R},X)$ is dense in $C_T(\mathbb{R},X)$ and that $\Delta(z)^{-1}$ is a bounded linear operator on $C_T(\mathbb{R},\mathbb{C}^n)$ due the closedness of $\Delta(z)$, recall \cref{thm:charmatrixDDE}. To prove the second claim, one needs a result like \cref{lemma:compactspectra}, see \cite{Article2} for more information. \hfill $\lozenge$
\end{remark}

To determine the function $\varphi_z \in C_T^1(\mathbb{R},\mathbb{C}^n)$ in \cref{cor:resolventDDE} in practice, we note that it is the unique solution $v$ of the (in)homogeneous periodic linear classical DDE $\Delta(z)v = w$, where $w$ denotes the function in brackets on right-hand side of \eqref{eq:varphizDelta}. The numerical methods and corresponding software used to compute such solutions are described in \cref{remark:jordanchainpractice}. The next result is a direct consequence of \cref{cor:multiplicity}.

\begin{corollary} \label{cor:multiplicityDDE}
If $\sigma \in \mathbb{C}$ is a Floquet exponent such that $k(\sigma,\mathcal{A}) = r(\Delta(\sigma))$, then the multiplicity theorem $\dim \mathcal{N}((\sigma I - \mathcal{A})^{r(\Delta(\sigma))}) =  m_a(\Delta(\sigma))$ holds.
\end{corollary}

The next corollary concerns so-called \emph{elementary solutions} of the periodic linear DDE \eqref{eq:LinearDDE} (at time $s$), i.e. solutions of the form $t \mapsto e^{\sigma (t-s)}p(t)q(t)$, where $p : \mathbb{R} \to \mathbb{R}$ is a polynomial and $q \in C_T^1(\mathbb{R},\mathbb{C}^n)$. In particular, we show that the Jordan chains of $\Delta$ generate the elementary solutions of \eqref{eq:LinearDDE}. This concept of elementary solutions extends the notion of elementary solutions for linear autonomous DDEs, see \cite[Corollary II.1.1.4]{Kaashoek1992}. Moreover, it generalizes the concept of Floquet solutions discussed in \cref{sec:applications}. For alternative approaches, compare with \cite[Lemma 8.1.2]{Hale1993}, \cite[Section XIII.4]{Diekmann1995}, and \cite[Section 11.2]{Kaashoek2022}.

The following result can be proven in two ways. First, a direct approach verifies that \eqref{eq:x(t)Floquet} satisfies \eqref{eq:LinearDDE} by utilizing the Jordan chain structure for $\Delta$ given in \eqref{eq:Jordanchainexplicit}. However, we present a proof that builds on the natural connection between $U, \mathcal{A}$ and $\Delta$.

\begin{corollary} \label{cor:FloquetsolDDE}
The elementary solutions of \eqref{eq:LinearDDE} are given by
    \begin{equation} \label{eq:x(t)Floquet}
        x(t) = e^{\sigma(t-s)} \sum_{l=0}^{k-1} \frac{(t-s)^l}{l!}q_{k-l-1}(t), \quad \forall t \geq s,
    \end{equation}
where $\{{q}_{0},\dots,{q}_{k-1} \}$ is a Jordan chain of $\Delta$ at a Floquet exponent $\sigma$.
\end{corollary}

\begin{proof}
Let $\varphi_i \in \mathcal{N}((\sigma I - \mathcal{A})^k)$ for $i=0,\dots,k-1$ be given. Since $x_t^i = U(t,s)\varphi_i(s)$, we know that $x^i(t) = (U(t,s)\varphi_i(s))(0)$ is a solution of \eqref{eq:LinearDDE} for $t \geq s$. To find $x^i$, recall from \cite[Theorem 5]{Article2} that
\begin{equation*}
    U(t,s)\varphi_i(s) = 
    \begin{cases}
        e^{\sigma (t-s)} \varphi_0(t), \quad  &i=0, \\
        e^{\sigma (t-s)} \varphi_i(t) - \sum_{l=1}^{i}\frac{(s-t)^l}{l!}U(t,s)\varphi_{i-l}(s), \quad &i=1,\dots,k-1.
    \end{cases}
\end{equation*}
An inductive argument, where we changed the order of summation and used the binomial formula, shows that
\begin{equation*}
    U(t,s)\varphi_i(s) = e^{\sigma (t-s)} \bigg(\varphi_i(t) - \sum_{l=1}^{i} \frac{(t-s)^{l}}{l!}\varphi_{i-l}(t) \bigg).
\end{equation*}
Now, \cref{thm:eigenfunctionsDDE} tells us that $\varphi_i(t)(\theta) = e^{\sigma \theta} \sum_{l=0}^{i} \frac{\theta^l}{l!} q_{k-l-1}(t+\theta)$ for all $t \in \mathbb{R}$ and $\theta \in [-h,0]$, and recalling that $x^i(t) = (U(t,s)\varphi_i(s))(0)$ for $t \geq s$ proves \eqref{eq:x(t)Floquet}.
\end{proof}

\begin{remark} \label{remark:jordanchainpractice}
To explicitly construct in practice a Jordan chain $\{{q}_{0},\dots,{q}_{k-1} \}$ for $\Delta$ at $\sigma$, to obtain for example the elementary solutions \eqref{eq:x(t)Floquet}, notice that the set of solutions $\{q_0,\dots,q_{k-1}\}$ of the system
\begin{equation} \label{eq:Jordanchainexplicit}
    \sum_{l=0}^i \frac{1}{l!}\Delta^{(l)}(\sigma)  q_{i-l} = 0, \quad \forall i=0,\dots,k-1,
\end{equation}
form a Jordan chain of length $k$, see \cite[Section 7.4]{Hale1993} or \cite[Section IV.5]{Diekmann1995}. Here, $z \mapsto \Delta^{(l)} (z)$ denotes the $l$th order derivative of the map $z \mapsto \Delta(z)$. A key advantage of \eqref{eq:Jordanchainexplicit} is that $\Delta^{(l)}(\sigma)$ becomes a bounded linear operator for any integer $l \geq 1$. Hence, \eqref{eq:Jordanchainexplicit} can be regarded as a linear system consisting of $kn$ (if $\sigma \in \mathbb{R}$) or $2kn$ (if $\sigma \notin \mathbb{R}$) real scalar-valued periodic linear (inhomogeneous) DDEs. Such systems can be solved by, for example, the well-known orthogonal collocation method \cite{Engelborghs2002a,Engelborghs2001}, that is already implemented in the \verb|MATLAB| continuation package \verb|DDE-BifTool| \cite{Engelborghs2002,Sieber2014} or the freely available \verb|Julia| continuation packages \verb|BifurcationKit| \cite{Veltz2020} and \verb|PeriodicNormalizationDDEs| \cite{Bosschaert2024c} for DDEs (consisting of finitely many discrete delays). \hfill $\lozenge$ 
\end{remark}

\begin{remark} \label{remark:autonomousDDE}
We next illustrate that the results from \cite[Section II.1.1]{Kaashoek1992} on characteristic matrices for \eqref{intro:LDDE} can be obtained from the construction presented here for \eqref{intro:LDDET} in two different ways. First, if the case $T=0$ is understood as the limit $T \downarrow 0$ in our construction, then for $l \in \{0,1\}$ the space $C_T^l(\mathbb{R},E)$ reduces to the space of $E$-valued $C^l$-smooth functions defined on $\mathbb{R}$ that are invariant under arbitrarily small translations. This limiting space consists precisely of constant functions, and hence we may canonically identify $C_0^l(\mathbb{R},E) = C_0(\mathbb{R},E) \cong E$. Second, we can restrict any (linear) operator defined (on a subspace) of the periodic functions $C_T(\mathbb{R},E)$ towards the constant functions $C_0(\mathbb{R},E)$ and identify this space again with $E$. Using one of these approaches and their associated identification, the operator $L \in \mathcal{L}(X,\mathbb{C}^n)$ becomes autonomous and takes the form
\begin{equation*}
    L\varphi = \int_0^h d\zeta(\theta)\varphi(-\theta).
\end{equation*}
The generator $A$ of the $\mathcal{C}_0$-semigroup $\{T(t)\}_{t \geq 0}$ on $X$, defined by $T(t) \coloneqq U(t,0)$ for all $t \geq 0$, has domain $\mathcal{D}(A) = \{ \varphi \in C^1([-h,0],\mathbb{C}^n) :  \varphi'(0) = L \varphi \}$, action $A \varphi = \varphi',$ and is a closed linear operator, see \cite[Theorem I.1.4]{Engel2000}. Using this identification, the closable linear operator $D$ from \cref{lemma:operatorlemmaDDE} reduces to $\mathcal{D}(D) = C^1(\mathbb{R},\mathbb{C}^n)$ with action $D\varphi = \varphi'$ as the derivative with respect to time vanishes. Hence, $D$ becomes a closed linear operator and there holds: $\mathcal{D}(\mathcal{A}) = \mathcal{D}(A)$ and $ \mathcal{D}(\hat{\mathcal{A}}) = \mathcal{D}(\hat{A})$, with action $\mathcal{A} \varphi = A \varphi$ and $ \hat{\mathcal{A}} \varphi = \hat{A} \varphi$, where
\begin{equation*}
        \mathcal{D}(\hat{A})  = \bigg \{
    \begin{pmatrix}
        c \\
        \varphi
    \end{pmatrix}
    \in \mathbb{C}^n \oplus X : \varphi \in \mathcal{D}(D), \ c = \varphi(0) \bigg \}, \quad \hat{A} 
    \begin{pmatrix}
        c \\
        \varphi
    \end{pmatrix}
    =
    \begin{pmatrix}
        K \varphi \\
        D\varphi
    \end{pmatrix}.
\end{equation*}
The bounded linear operator $M$ and closed linear operator $K$ now become bounded linear operators on $X$ given by $M \varphi = \varphi(0)$ and $K \varphi = L \varphi$, as $\dot{\varphi}(t)(0) = 0$ since $t \mapsto \varphi(t)(0)$ is constant. Thus, the operator-valued function $\Delta$ reduces to a matrix-valued function $\Delta : \mathbb{C} \to \mathcal{L}(\mathbb{C}^n)$ of the form
\begin{equation*}
    \Delta(z) = zI - \int_0^h d\zeta(\theta)e^{-z \theta},
\end{equation*}
which is exactly the characteristic matrix found in \cite[Section II.1.1]{Kaashoek1992}. With a similar argumentation, one can also verify that the functions $E$ and $F$, and the resolvent $R(z,\mathcal{A})$ reduce to the expressions found in \cite[Section II.1.1]{Kaashoek1992}. This reduction also holds for the results in the upcoming applications, see \cite[Section II.1.3]{Kaashoek1992} for iDDEs and \cite[Section 4]{Hupkes2008} for MFDEs.
\hfill $\lozenge$
\end{remark}

\begin{remark} \label{remark:charoperatorTh}
The previous remark demonstrated for \eqref{intro:LDDE} that the characteristic \emph{operator} reduces to a characteristic \emph{matrix} under a suitable identification of function spaces. When the period $T$ is rationally related to the maximal delay $h$ in \eqref{intro:LDDET}, that is $T/h \in \mathbb{Q}$, a similar reduction can be carried out. As recalled in \cref{sec:introduction}, several other characteristic matrices have been constructed in this setting. A natural question is how these characteristic matrices relate to the reduced characteristic operator introduced in this section. We expect that these matrices are somehow connected to $\Delta(z)$ by restricting its domain to a particular finite-dimensional subspace of $C_T(\mathbb{R},\mathbb{C}^n)$, although this remains an interesting topic for future research. \hfill $\lozenge$
\end{remark}

\subsection{Delay differential equations with infinite delay} \label{subsec:infinitedelay}

For delay differential equations with infinite delay (iDDE), it is not obvious how to choose a state space $X$, since the `natural candidate' $C(\mathbb{R}_{-},\mathbb{C}^n)$ is not a Banach space with respect to the standard associated supremum norm. Indeed, the choice of state space has been the topic of extensive study in the literature, see for example the articles \cite{Murakami1989,Hino1991,Hale1978,Hino1991} that identify criteria which a state space should satisfy in order to allow for classical results on existence, uniqueness and (in)stability of solutions. In \cite{Diekmann2012}, the authors choose the Banach space $X := C_\rho(\mathbb{R}_{-},\mathbb{C}^n)$ with the weighted norm $\| \cdot \|_\rho$ defined by
\begin{equation} \label{eq:statespaceinfdelay}
    C_\rho(\mathbb{R}_{-},\mathbb{C}^n) \coloneqq \bigg\{ \varphi \in C(\mathbb{R_{-}},\mathbb{C}^n) : \lim_{\theta \to -\infty} e^{\rho \theta} | \varphi(\theta) | = 0 \bigg\},  \quad \| \varphi \|_\rho \coloneqq \sup_{\theta \leq 0} e^{\rho \theta} | \varphi(\theta) |
\end{equation}
for some fixed $\rho > 0$. With this choice of state space, the authors of \cite{Diekmann2012} proved the principle of linearized stability, and argue that the center manifold and Hopf bifurcation theorem can also be formulated and proved (along the same lines of \cite{Diekmann1995} or \cite{Magal2009}). All this suggests that $X$ is a suitable choice of state space for our characteristic operator framework, and we shall adopt this space throughout the remainder of this section. For more general state spaces to which our results also apply, we refer to \cref{remark:statespaceinfdelay} at the end of this subsection. Regardless of the chosen state space, the solution operator $U(\cdot,s)$ cannot in general be expected to become eventually compact (as was the case in \cref{subsec:classicalDDEs}), since the state necessarily contains a part of the initial function defined on $\mathbb{R}_{-}$. In this context, it is therefore noteworthy that the characteristic operator framework allows us to prove that the Floquet exponents are isolated and of finite type in a given right half-plane in $\mathbb{C}$ containing the origin, see \cref{thm:eigenfunctionsDDEinfinite}. 

To state and prove our results, we will assume that $L \in C_T(\mathbb{R},\mathcal{L}(X,\mathbb{C}^n))$. A vector-valued version of the Riesz representation theorem tells us that for every $t \in \mathbb{R}$ the operator $L(t)$ can be represented uniquely as
\begin{equation*}
    L(t) = \int_0^\infty d_2 \mu(t,\theta) \varphi(-\theta).
\end{equation*}
Here, the map $\mu : \mathbb{R} \times \mathbb{R}_{+} \to \mathbb{C}^{n \times n}$ is a matrix-valued function, and $\mu(\cdot,\theta)$ is continuous and $T$-periodic for all $\theta \in \mathbb{R}_{+}$. Moreover, $\mu(t,\cdot)$ is a Borel measure satisfying $\|\mu(t,\cdot)\|_\rho \coloneqq \int_0^\infty e^{-\rho s} d |\mu(t,\cdot)|(s) < \infty$, where $|\mu(t,\cdot)|$ denotes the total variation of $\mu(t,\cdot)$ for all $t \in \mathbb{R}$. With this choice of state space, the initial value problem
\begin{equation} \label{eq:LinearDDEinfinite}
    \begin{dcases}
    \dot{x}(t) = L(t)x_t, \quad &t \geq s, \\
    x_s = \varphi, \quad &\varphi \in X,    
    \end{dcases}    
\end{equation}
is well-posed on $X$, and the family of operators $U \coloneqq \{U(t,s)\}_{t \geq s}$ defined through $x_t = U(t, s)\varphi$ for all $t \geq s$ forms a strongly continuous forward evolutionary system on $X$. The associated (generalized) generator of $U$ (at time $t \in \mathbb{R}$) is given by the closed linear operator $A(t) : \mathcal{D}(A(t)) \subseteq X \to X$ with domain and action
\begin{equation*}
    \mathcal{D}(A(t)) = \{ \varphi \in C_\rho^1(\mathbb{R}_{-},\mathbb{C}^n) : \varphi'(0) = L(t)\varphi \}, \quad A(t)\varphi = \varphi'.
\end{equation*}
To this operator, we associate the linear operator $\mathcal{A} : \mathcal{D}(\mathcal{A}) \subseteq C_T(\mathbb{R},X) \to C_T(\mathbb{R},X)$ defined by 
\begin{equation} \label{eq:curlya_idde}
     \mathcal{D}(\mathcal{A}) = \{ \varphi \in C_T^1(\mathbb{R},X) : \varphi(t) \in \mathcal{D}(A(t)) \mbox{ for all $t \in \mathbb{R}$} \}, \quad (\mathcal{A}\varphi)(t) = A(t)\varphi(t) - \dot{\varphi}(t). 
\end{equation}
Due to the lack of eventual compactness of $U(\cdot,s)$, it turns out that we can only characterize the spectrum of $\mathcal{A}$ in terms of $\Delta$ in the right half-plane $\mathbb{C}_{-\rho} \coloneqq \{ z \in \mathbb{C} : \Re(z) > -\rho \}$,
see \cite[Section II.1.3]{Kaashoek1992} and \cite[Section 4]{Diekmann2012} for the same issue occurring in autonomous iDDEs. 

\begin{lemma}
Let $D : \mathcal{D}(D) \to C_T(\mathbb{R},X)$ be the operator defined by
\begin{equation*}
    \mathcal{D}(D) = \{ \varphi \in C_T^1(\mathbb{R},X) : \varphi(t)' \in X \mbox{ for all $t \in \mathbb{R}$} \}, \quad (D\varphi)(t) = \varphi(t)' - \dot{\varphi}(t).
\end{equation*}
Then $D$ is a closable linear operator satisfying \ref{hyp:iota},\ref{hyp:restriction}, \ref{hyp:domainrangeD} and \ref{hyp:domaincurlyA} with $\Omega = \mathbb{C}_{-\rho}$. Moreover, define the linear operators $K : \mathcal{D}(K) \subseteq C_T(\mathbb{R},X) \to C_T(\mathbb{R},\mathbb{C}^n)$ and $M : C_T(\mathbb{R},X) \to C_T(\mathbb{R},\mathbb{C}^n)$ by
\begin{equation*}
    \mathcal{D}(K) = C_T^1(\mathbb{R},X), \qquad (K\varphi)(t) = L(t) \varphi(t) - \dot{\varphi}(t)(0), \qquad (M\varphi)(t) = \varphi(t)(0), 
\end{equation*}
then $K$ satisfies \ref{hyp:K}, and $\mathcal{A}$ defined in \eqref{eq:curlya_idde} is the first operator associated to $D, K$ and $M$.  
\end{lemma}
\begin{proof}
The proof is analogous to the proof of \cref{lemma:operatorlemmaDDE} until we have to solve the resolvent equation \eqref{eq:resolventD01} for $\varphi$ with right-hand side $\psi \in C_T^1(\mathbb{R},X)$. Equation \eqref{eq:resolventD02} tells us that the (formal) solution $\varphi$ reads
\begin{equation*}
\varphi(t)(\theta) = \int_\theta^0 e^{z(\theta - s)} \psi(t+\theta -s)(s) ds, 
\end{equation*}
but we will prove that this function belongs to $\mathcal{D}(D_0)$ if and only if $\Re(z) > -\rho$, so that $\Omega = \mathbb{C}_{-\rho}$ is nonempty. To do this, we first show that $\varphi$ is an element of the state space $X = C_\rho(\mathbb{R}_-, \mathbb{C}^n)$ if and only if $\Re(z) > -\rho$. Clearly, $\varphi$ is continuous so it remains to show that $e^{\rho \theta}|\varphi(t)(\theta)| \to 0$ as $\theta \to -\infty$ for all $t \in \mathbb{R}$. Let $t \in \mathbb{R}, \varepsilon > 0$ and $z \in \mathbb{C}_{-\rho}$ be given. Since $\psi$ takes values in $X$, we know that $e^{\rho s}\psi(\tau)(s) \to 0$ as $s \to -\infty$ for all $\tau \in \mathbb{R}$ and so
\begin{equation*}
e^{\rho s} | \psi(t+\theta-s)(s)| \leq \max_{\tau \in [0,T]} e^{\rho s} |\psi(\tau)(s)| \to 0,
\end{equation*}
as $s \to - \infty$ due to compactness of $[0,T]$ and $T$-periodicity of $\psi$ in the first component. Hence, there exists an $M_1 < 0$ such that $e^{\rho s}|\psi(t+\theta-s)(s)| \leq (\Re(z) + \rho)\varepsilon /2$ for all $s \leq M_1$. Moreover, since $e^{(\Re(z) + \rho)\theta} \to 0$ as $\theta \to - \infty$, there exists an $M_2 < 0$ such that $e^{(\Re(z) + \rho)\theta} \leq e^{-M_1(\Re(z) + \rho)} \varepsilon/(2 \|\psi\|_\infty) $ for all $\theta \leq M_2$. Hence, for all $\theta \leq M = \min\{M_1,M_2\}$ there holds
\begin{align*}
    e^{\rho \theta} |\varphi(t)(\theta)| &\leq \int_\theta^M e^{(\Re(z) + \rho)(\theta-s)} e^{\rho s} |\psi(t+\theta-s)(s)| ds + \int_M^0 e^{(\Re(z) + \rho)(\theta-s)} e^{\rho s} |\psi(t+\theta-s)(s)| ds \\
    &\leq \frac{(\Re(z) + \rho) \varepsilon}{2} \int_\theta^M e^{(\Re(z) + \rho)(\theta-s)} ds + \| \psi \|_\infty \int_M^0 e^{(\Re(z) + \rho)(\theta-s)} ds \leq \frac{\varepsilon}{2} + \frac{\varepsilon}{2} = \varepsilon,
\end{align*}
where we computed both integrals explicitly and used previous estimates to obtain the result. We conclude that $\varphi \in C_T(\mathbb{R},X)$. Computing the derivatives $\dot{\varphi}(t)(\theta)$ and ${\varphi}(t)'(\theta)$ as in \eqref{eq:derivativesresolventDDE}, we see that $\varphi$ is in $C_T^1(\mathbb{R},X)$ and $\varphi(t)' \in X$ for all $t \in \mathbb{R}$ as $\psi \in C_T^1(\mathbb{R},X)$, and therefore $\varphi \in \mathcal{D}(D_0)$. To prove the converse, assume that $\Re(z) \leq - \rho$ for some fixed $z \in \mathbb{C}$, and consider a nonzero function $\psi \in C_T^1(\mathbb{R},X)$ satisfying the translation property $\psi(t)(\theta) = \psi(t+\theta)(0)$ for all $t \in \mathbb{R}$ and $\theta \in [-h,0]$. Since $e^{\rho \theta} \to 0$ as $\theta \to -\infty$ and $e^{(\Re(z) + \rho) \theta} \geq 1$ for $\theta \leq 0$, there exists an $M_3 < 0$ such that $e^{\rho \theta}(e^{\Re(z) \theta} - 1) > 1/2$ for all $\theta \leq M_3$. Hence, an application of the reverse triangle inequality tells us that
\begin{equation*}
    e^{\rho \theta} |\varphi(t)(\theta)| \geq \frac{e^{\rho \theta}(e^{\Re(z) \theta} - 1)}{|z|}|\psi(t+\theta)(0)|  > \frac{1}{2|z|}|\psi(t+\theta)(0)|, \quad \forall t \in \mathbb{R}, \ \theta \leq M_3,
\end{equation*}
which proves that $e^{\rho \theta} |\varphi(\cdot)(\theta)| \nrightarrow 0$ as $\theta \to - \infty$. We conclude that $\varphi \in \mathcal{D}(D_0)$ if and only if $\Re(z) > - \rho$. Hence, $zI-D_0$ has a densely defined bounded linear inverse for every $z \in \mathbb{C}_{-\rho}$ given by \eqref{eq:resolventD0} since
\begin{equation*}
    \|R(z,D_0)\varphi\|_\infty \leq \frac{1}{\Re(z) + \rho} \| \varphi \|_\infty, \quad \forall \varphi \in \mathcal{R}(zI-D_0).
\end{equation*}
The remaining part of the proof can be completed as performed in the proof of \cref{lemma:operatorlemmaDDE}.
\end{proof}

The second operator $\hat{\mathcal{A}}$ associated with $D, K$ and $M$ now takes the form 
\begin{equation*}
    \mathcal{D}(\hat{\mathcal{A}}) = \bigg \{ \begin{pmatrix}
    q\\
    \varphi
    \end{pmatrix}
    \in \hat{C}_T^1(\mathbb{R},X) : \varphi \in \mathcal{D}(D), \ q = M\varphi \bigg \}, \quad 
    \hat{\mathcal{A}}
    \begin{pmatrix}
    q\\
    \varphi
    \end{pmatrix}
    =
    \begin{pmatrix}
    K\varphi \\
    D\varphi
    \end{pmatrix}.
\end{equation*}

\begin{theorem} \label{thm:charmatrixDDEinfinite}
The holomorphic operator-valued function $\Delta : \mathbb{C} \to L(C_T(\mathbb{R},\mathbb{C}^n))$ with values $\Delta(z) : C_T^1(\mathbb{R},\mathbb{C}^n) \to C_T(\mathbb{R},\mathbb{C}^n)$ given by
\begin{equation*}
    (\Delta(z)q)(t) = \dot{q}(t) + z q(t) - \int_0^\infty d_2 \mu(t,\theta) e^{-z \theta} q(t-\theta),
\end{equation*}
is a $\frac{2\pi i}{T}\mathbb{Z}$-similar closed characteristic operator for $\hat{\mathcal{A}}$ on $\mathbb{C}_{-\rho}$ and the equivalence is given by
\begin{equation*}
    \begin{pmatrix}
        \Delta(z) & 0 \\
        0 & I 
    \end{pmatrix}
    =
    F(z)(z I - \hat{\mathcal{A}})E(z), \quad \forall z \in \mathbb{C}_{-\rho},
\end{equation*}
where $E : \mathbb{C}_{-\rho} \to L(\hat{C}_T(\mathbb{R},X),\hat{C}_T(\mathbb{R},X)_{\hat{\mathcal{A}}})$ and $F : \mathbb{C}_{-\rho} \to L(\hat{C}_T(\mathbb{R},X))$ are holomorphic bounded operator-valued functions such that $E(z) \in \mathcal{L}(C_T^1(\mathbb{R},\mathbb{C}^n) \oplus \mathcal{R}(zI-D_0),\mathcal{D}(\hat{\mathcal{A}}))$ and $F(z) \in \mathcal{L}(C_T(\mathbb{R},\mathbb{C}^n) \oplus \mathcal{R}(zI-D_0))$ are bijective mappings, given by
\begin{alignat*}{2}
    E(z)
    \begin{pmatrix}
    q\\
    \varphi
    \end{pmatrix}
    &=
    \begin{pmatrix}
    q\\
    \psi
    \end{pmatrix}, \quad
        \psi(t)(\theta) &&= e^{z \theta}q(t+\theta)+ \int_{\theta}^0 e^{z(\theta-s)}\varphi(t+\theta-s)(s) ds, \\
    F(z)
    \begin{pmatrix}
    q\\
    \varphi
    \end{pmatrix}
    &=
    \begin{pmatrix}
    q + \phi\\
    \varphi
    \end{pmatrix}, \quad
        \phi(t) &&= \int_0^\infty d_2 \mu(t,\theta) \int_{-\theta}^0 e^{-z(\theta+s)}\varphi(t-\theta-s)(s) ds.
\end{alignat*}
\end{theorem}
\begin{proof}
The proof is similar to the proof of \cref{thm:charmatrixDDE}. We must only verify that the map $\theta \mapsto e^{z\theta}q(t+\theta)$ is in $X$ and that $\Delta(z)$ is closed for all $z \in \mathbb{C}_{-\rho}, t \in \mathbb{R}$ and $q \in C_T^1(\mathbb{R},\mathbb{C}^n)$. To prove the first claim, note that $\theta \mapsto e^{z\theta}q(t+\theta)$ is clearly continuous and that $e^{\rho \theta} |e^{z\theta}q(t+\theta)| \leq e^{(\Re(z) + \rho)\theta} \|q\|_\infty \to 0$ as $\theta \to - \infty$ since $z \in \mathbb{C}_{-\rho}$. To prove the second claim, note that the linear operator $N(z)$ from the proof of \cref{thm:charmatrixDDE} is still bounded since $\|N(z)q\|_\infty \leq (|z| + \|L\|_\infty) \|q\|_\infty$ for all $z \in \mathbb{C}_{-\rho}$.
\end{proof}

The next key result shows that the spectrum of $\mathcal{A}$ in the given right half-plane $\mathbb{C}_{-\rho}$ with $\rho > 0$, can be characterised in terms of the characteristic values of the operator $\Delta$. In particular, the spectrum of $\mathcal{A}$ in $\mathbb{C}_{-\rho}$ consists solely of isolated eigenvalues of finite type. This is remarkable, given that for iDDEs the associated solution operator $U(\cdot,s)$ of \eqref{eq:LinearDDEinfinite} is not eventually compact.

\begin{theorem} \label{thm:eigenfunctionsDDEinfinite}
The spectrum of $\mathcal{A}$ with real part greater than $-\rho$ consists solely of isolated eigenvalues of finite type:
\begin{equation*}
    \sigma(\mathcal{A}) \cap \mathbb{C}_{-\rho} =  \{ \sigma \in \mathbb{C}_{-\rho} :  \sigma \textit{ is a characteristic value of } \Delta \}  =  \{\sigma \in \mathbb{C}_{-\rho} :  \sigma \mbox{ is a Floquet exponent} \}.
\end{equation*}
For $\sigma \in \sigma(\mathcal{A}) \cap \mathbb{C}_{-\rho}$, there holds $m_g(\sigma,\mathcal{A}) = m_g(\Delta(\sigma)), m_a(\sigma,\mathcal{A}) = m_a(\Delta(\sigma)), k(\sigma,\mathcal{A}) = k(\Delta(\sigma))$ and $r(\sigma,\mathcal{A}) = r(\Delta(\sigma))$, and the partial multiplicities of $\sigma$ considered as an eigenvalue of $\mathcal{A}$ are equal to the zero multiplicities of $\sigma$ considered as a characteristic value of $\Delta$. Furthermore, if $\{ \{{q}_{i,0},\dots,{q}_{i,k_i-1} \} : i=1,\dots,p\}$ is a canonical system of Jordan chains of $\Delta$ at $\sigma$, then $\{ \{\varphi_{i,0},\dots,\varphi_{i,k_i-1} \} : i=1,\dots,p \}$ is a canonical system of (generalized) eigenvectors of $\mathcal{A}$ at $\sigma$, where
\begin{equation*}
    \varphi_{i,k}(t)(\theta) = e^{\sigma \theta} \sum_{l=0}^k \frac{\theta^l}{l!}q_{i,k-l}(t+\theta), \quad \forall k = 0,\dots,k_{i}-1, \ t \in \mathbb{R}, \ \theta \in \mathbb{R}_{-}.
\end{equation*}
\end{theorem}
\begin{proof}
Similar to the proof of \cref{lemma:characinvertible}, one can show that the characteristic operator from \cref{thm:charmatrixDDEinfinite} satisfies $\rho(\Delta(z)) \neq \emptyset$ by taking the real part of $\mu \in \mathbb{C} \setminus (z + \frac{2 \pi i}{T}\mathbb{Z})$ for $z \in \mathbb{C}_{-\rho}$ in the proof of \cref{lemma:characinvertible} sufficiently large. Similarly, one can prove that \ref{hyp:rangeequal} holds. The remaining part of the proof is now analogous to the proof of \cref{thm:eigenfunctionsDDE}.
\end{proof}

For completeness, we next state the form of the resolvent operator $R(z, \mathcal{A})$, the multiplicity theorem, and a characterisation of the elementary solutions. The proofs of these results are omitted because they are similar to those of \cref{cor:resolventDDE}, \cref{cor:multiplicityDDE} and \Cref{cor:FloquetsolDDE}.

\begin{corollary}
If $z \in \mathbb{C}_{-\rho}$ is not a Floquet exponent, then the resolvent of $\mathcal{A}$ at $z$ has the following representation
\begin{equation*}
    (R(z,\mathcal{A}) \varphi)(t)(\theta) = e^{z \theta} \varphi_z(t+\theta) + \int_\theta^0 e^{z(\theta-s)} \varphi(t+\theta-s)(s) ds, \quad \forall \varphi \in \mathcal{R}(zI-\mathcal{A}) \supseteq C_T^1(\mathbb{R},X),
\end{equation*}
where $\varphi_z$ is given by
\begin{equation*}
    \varphi_z = \Delta(z)^{-1} \bigg( t \mapsto \varphi(t)(0) + \int_0^\infty d_2 \mu(t,\theta)\int_{-\theta}^{0} e^{-z(\theta + s)} \varphi(t -\theta - s)(s) ds \bigg).
\end{equation*}
\end{corollary}

\begin{corollary}
If $\sigma \in \mathbb{C}_{-\rho}$ is a Floquet exponent such that $k(\sigma,\mathcal{A}) = r(\Delta(\sigma))$, then the multiplicity theorem $\dim \mathcal{N}((\sigma I - \mathcal{A})^{r(\Delta(\sigma))}) =  m_a(\Delta(\sigma))$ holds.
\end{corollary}

\begin{corollary} \label{cor:FloquetsolDDEinfinite}
    The elementary solutions of \eqref{eq:LinearDDEinfinite} are given by
    \begin{equation*}
        x(t) = e^{\sigma(t-s)} \sum_{l=0}^{k-1} \frac{(t-s)^l}{l!}q_{k-l-1}(t), \quad \forall t \geq s,
    \end{equation*}
where $\{ {q}_{0},\dots,{q}_{k-1} \}$ is a Jordan chain of $\Delta$ at a Floquet exponent $\sigma \in \mathbb{C}_{-\rho}$.
\end{corollary}

In practice, one can explicitly construct a Jordan chain $\{q_0,\dots,q_{k-1}\}$ for $\Delta$ at $\sigma \in \mathbb{C}_{-\rho}$ by using the same technique as described in \cref{remark:jordanchainpractice}. The only difference is that $\Delta(\sigma)q_0 = 0$ is now a periodic linear homogeneous iDDE and the remaining $k-1$ equations in \eqref{eq:Jordanchainexplicit} to obtain $q_1,\dots,q_{k-1}$ are periodic linear inhomogeneous iDDEs. The mentioned orthogonal collocation method is still suited to solve this particular linear system of differential equations.

\begin{remark} \label{remark:statespaceinfdelay}
As discussed in the introduction of this subsection, there are several potential choices of state spaces $X$ for iDDEs. There are two other natural state spaces on which our results can be applied:

1) Instead of working on the exponentially weighted space $X =  C_\rho(\mathbb{R}_{-}, \mathbb{C}^n)$ from \eqref{eq:statespaceinfdelay}, one can instead choose a weighted function space with a more general weight function, see for example \cite[Section 2]{Faria2002}. More precisely, let $w : \mathbb{R}_{-} \to (0,1]$ be a nondecreasing continuous function satisfying
\begin{equation*}
    w(0) = 1, \quad \lim_{\theta \to -\infty} w(\theta) = 0, \quad \lim_{u \uparrow 0} \frac{w(\theta+u)}{w(\theta)} = 1, \quad \forall \theta \in \mathbb{R}_{-},
\end{equation*}
where in the third equality, the limit is uniform in $\theta$. We can then consider the $w$-weighted Banach space $C_w(\mathbb{R}_{-},\mathbb{C}^n)$ with weighted norm $\| \cdot \|_w$ defined by
\begin{equation*}
    C_w(\mathbb{R}_{-},\mathbb{C}^n) \coloneqq \{ \varphi \in C(\mathbb{R}_-, \mathbb{C}^n) : \lim_{\theta \to -\infty} w(\theta) | \varphi(\theta) | = 0\}, \quad \| \varphi \|_w \coloneqq \sup_{\theta \leq 0} w(\theta) | \varphi(\theta) |.
\end{equation*}
In this setting, we can construct a characteristic operator on $\Omega = \mathbb{C}_\beta = \{ z \in \mathbb{C} : \Re(z) > \beta\}$, where
\begin{equation} \label{eq:Lambdabeta}
    \beta = \inf_{z \in \Lambda} \Re(z), \quad \Lambda = \bigg \{ z \in \mathbb{C} : \lim_{\theta \to - \infty} e^{\Re(z) \theta}  w(\theta) = 0 \bigg\},
\end{equation}
and note that $\beta = - \rho$ if $w(\theta) = e^{\rho \theta}$ for all $\theta \in \mathbb{R}_{-}$, as expected.

2) An alternative popular choice of state space is a weighted $L^p$-space, see for example \cite[Section II.1.3]{Kaashoek1992} or \cite[Section 1]{Naito1976}. This state space is constructed as follows: let $w : \mathbb{R}_{-} \to (0,1]$ be a measurable nondecreasing function such that $\| w \|_{L^1}$ is finite, and consider for every $p \geq 1$ the Banach space $L_w^p(\mathbb{R}_{-},\mathbb{C}^n)$ of measurable functions $\varphi: \mathbb{R}_{-} \to \mathbb{C}^n$ such that for any given $k \geq 0$, the restricted function $\varphi |_{[-k,0]}$ is continuous with norm
\begin{equation*}
    \| \varphi \|_{L_w^p} \coloneqq \sup_{k \geq 0} \bigg( \sup_{\theta \in [-k,0]} | \varphi(\theta) |^p + \int_{-\infty}^0 | \varphi(\theta) |^p w(\theta) d\theta  \bigg)^{1/p} < \infty.
\end{equation*}
In this setting $\beta$ is given as in \eqref{eq:Lambdabeta}, where now $\Lambda = \{ z \in \mathbb{C} : \int_{-\infty}^0 e^{p\Re(z)\theta} w(\theta) < \infty \}.$ \hfill $\lozenge$
\end{remark}

\subsection{Mixed functional differential equations} \label{subsec:MFDEs}

In the previous two examples, the rate of change of the system depended on the current and past states of the system. One can generalize these retarded equations in the sense that the rate of change of a system is allowed to depend on future states as well. Such differential equations arise, for example, naturally in the study of travelling wave solutions to differential equations posed on lattices, see for example \cite{Hupkes2012,Hupkes2010,Wu1997} and the references therein. For these so-called \emph{mixed functional differential equations} (MFDEs), we choose the state space $X$ as the Banach space $C([r_{-},r_{+}],\mathbb{C}^n)$, equipped with the supremum norm $\| \cdot \|_\infty$, where the (maximal) retarded and advanced arguments $r_{-} < 0 < r_{+}$ are assumed to be finite. Furthermore, let us consider a function $L \in C_T(\mathbb{R}, \mathcal{L}(X,\mathbb{C}^n))$. A vector-valued version of the Riesz representation theorem tells us that for every $t \in \mathbb{R}$ the operator $L(t)$ can be represented uniquely as
\begin{equation} \label{eq:L(t)RieszMFDE}
    L(t)\varphi = \int_{r_{-}}^{r_{+}} d_2 \eta(t,\theta) \varphi(\theta),
\end{equation}
and notice the subtle difference of a minus sign in \eqref{eq:L(t)RieszMFDE} compared to \eqref{eq:L(t)RieszDDE}. The map $\eta : \mathbb{R} \times [r_{-},r_{+}] \to \mathbb{C}^{n \times n}$ is a matrix-valued function, and $\eta(\cdot,\theta)$ is continuous and $T$-periodic for all $\theta \in [r_{-},r_{+}]$. Moreover, $\eta(t,\cdot)$ is of bounded variation, right-continuous on the open interval $(r_{-},r_{+})$ and normalized by the requirement $\eta(\cdot,0) = 0$.

In the setting of MFDEs, it turns out that the initial value problem
\begin{equation} \label{eq:LinearMFDE}
    \begin{dcases}
    \dot{x}(t) = L(t)x_t, \quad &t \geq s, \\
    x_s = \varphi, \quad &\varphi \in X.   
    \end{dcases}    
\end{equation}
is in general ill-posed, see for example the simple illustrative example by H\"arterich et al. \cite{Harterich2002}. Despite the ill-posedness of \eqref{eq:LinearMFDE}, the general theory for MFDEs has been developed extensively over the last two decades, see for example \cite{MalletParet1999,MALLETPARET2005,Rustichini1989,Hupkes2006}. We especially mention the work by Hupkes and Verduyn Lunel on periodic center manifolds for periodic MFDEs \cite{Hupkes2008}, as we will solve an open question regarding the spectral structure generated by periodic MFDEs by an application of the characteristic operator, see \cref{remark:MFDEspectral} for further details.

Since the initial value problem \eqref{eq:LinearMFDE} is in general ill-posed, we cannot construct an associated forward evolutionary system $U$, and therefore the Floquet multipliers cannot be defined via the monodromy operator $U(s+T,s)$. However, we can define the Floquet exponents in terms of $\mathcal{A}$ and $\Delta$, as we have also done in the previous two subsections. We emphasize that this definition is natural in the light of previous definitions of eigenvalues for autonomous MFDEs, which are based on a conceptually similar construction, see \cite{MalletParet1999,MALLETPARET2005,Hupkes2006} for further details. Concretely, we define for any $t\in \mathbb{R}$ the closed linear operator $A(t) : \mathcal{D}(A(t)) \subseteq X \to X$ by
\begin{equation*} 
    \mathcal{D}(A(t)) = \{ \varphi \in C^1([r_{-},r_{+}],\mathbb{C}^n) : \varphi '(0) = L(t) \varphi \}, \quad A(t)\varphi = \varphi'.
\end{equation*}
To this operator, we associate the linear operator $\mathcal{A} : \mathcal{D}(\mathcal{A}) \subseteq C_T(\mathbb{R},X) \to C_T(\mathbb{R},X)$ defined by 
\begin{equation} \label{eq:curlya_mfde}
     \mathcal{D}(\mathcal{A}) = \{ \varphi \in C_T^1(\mathbb{R},X) : \varphi(t) \in \mathcal{D}(A(t)) \mbox{ for all $t \in \mathbb{R}$} \}, \quad (\mathcal{A}\varphi)(t) = A(t)\varphi(t) - \dot{\varphi}(t). 
\end{equation}
The proof of the following result is similar to \cref{lemma:operatorlemmaDDE} and therefore omitted.

\begin{lemma}
Let $D : \mathcal{D}(D) \to C_T(\mathbb{R},X)$ be the operator defined by
\begin{equation*}
    \mathcal{D}(D) = \{ \varphi \in C_T^1(\mathbb{R},X) : \varphi(t)' \in X \mbox{ for all $t \in \mathbb{R}$} \}, \quad (D\varphi)(t) = \varphi(t)' - \dot{\varphi}(t).
\end{equation*}
Then $D$ is a closable linear operator satisfying \ref{hyp:iota}, \ref{hyp:restriction}, \ref{hyp:domainrangeD} and \ref{hyp:domaincurlyA} with $\Omega = \mathbb{C}$. Moreover, define the linear operators $K : \mathcal{D}(K) \subseteq C_T(\mathbb{R},X) \to C_T(\mathbb{R},\mathbb{C}^n)$ and $M : C_T(\mathbb{R},X) \to C_T(\mathbb{R},\mathbb{C}^n)$ by
\begin{equation*}
    \mathcal{D}(K) = C_T^1(\mathbb{R},X), \qquad (K\varphi)(t) = L(t) \varphi(t) - \dot{\varphi}(t)(0), \qquad (M\varphi)(t) = \varphi(t)(0), 
\end{equation*}
then $K$ satisfies \ref{hyp:K}, and $\mathcal{A}$ defined in \eqref{eq:curlya_mfde} is the first operator associated to $D, K$ and $M$.
\end{lemma}
The second operator $\hat{\mathcal{A}}$ associated with $D, K$ and $M$ now takes the form 
\begin{equation*}
    \mathcal{D}(\hat{\mathcal{A}}) = \bigg \{ \begin{pmatrix}
    q\\
    \varphi
    \end{pmatrix}
    \in \hat{C}_T^1(\mathbb{R},X) : \varphi \in \mathcal{D}(D), \ q = M\varphi \bigg \}, \quad 
    \hat{\mathcal{A}}
    \begin{pmatrix}
    q\\
    \varphi
    \end{pmatrix}
    =
    \begin{pmatrix}
    K\varphi \\
    D\varphi
    \end{pmatrix}.
\end{equation*}
The proof of the following result is similar to that of \cref{thm:charmatrixDDE} and therefore omitted.

\begin{theorem}
The holomorphic operator-valued function $\Delta : \mathbb{C} \to L(C_T(\mathbb{R},\mathbb{C}^n))$ with values $\Delta(z) : C_T^1(\mathbb{R},\mathbb{C}^n) \to C_T(\mathbb{R},\mathbb{C}^n)$ given by
\begin{equation*}
    (\Delta(z)q)(t) = \dot{q}(t) + z q(t) - \int_{r_{-}}^{r_{+}} d_2 \eta(t,\theta) e^{z \theta} q(t+\theta),
\end{equation*}
is a $\frac{2\pi i}{T}\mathbb{Z}$-similar closed characteristic operator for $\hat{\mathcal{A}}$ on $\mathbb{C}$ and the equivalence is given by
\begin{equation*}
    \begin{pmatrix}
        \Delta(z) & 0 \\
        0 & I 
    \end{pmatrix}
    =
    F(z)(z I - \hat{\mathcal{A}})E(z), \quad \forall z \in \mathbb{C},
\end{equation*}
where $E : \mathbb{C} \to L(\hat{C}_T(\mathbb{R},X),\hat{C}_T(\mathbb{R},X)_{\hat{\mathcal{A}}})$ and $F : \mathbb{C} \to L(\hat{C}_T(\mathbb{R},X))$ are holomorphic bounded operator-valued functions such that $E(z) : C_T^1(\mathbb{R},\mathbb{C}^n) \oplus \mathcal{R}(zI-D_0) \to \mathcal{D}(\hat{\mathcal{A}})$ and $F(z) \in \mathcal{L}(C_T(\mathbb{R},\mathbb{C}^n) \oplus \mathcal{R}(zI-D_0))$ are bijective mappings, given by
\begin{alignat*}{2}
    E(z)
    \begin{pmatrix}
    q\\
    \varphi
    \end{pmatrix}
    &=
    \begin{pmatrix}
    q\\
    \psi
    \end{pmatrix}, \quad
        \psi(t)(\theta) &&= e^{z \theta}q(t+\theta)+ \int_{\theta}^0 e^{z(\theta-s)}\varphi(t+\theta-s)(s) ds, \\
    F(z)
    \begin{pmatrix}
    q\\
    \varphi
    \end{pmatrix}
    &=
    \begin{pmatrix}
    q + \phi\\
    \varphi
    \end{pmatrix}, \quad
        \phi(t) &&= \int_{r_{-}}^{r_{+}} d_2\eta(t,\theta) \int_{\theta}^0 e^{z(\theta-s)}\varphi(t+\theta-s)(s) ds.
\end{alignat*}
\end{theorem}

\begin{theorem} \label{thm:eigenfunctionsMFDE}
If $\rho(\Delta) \neq \emptyset$, then the spectrum of $\mathcal{A}$ consists solely of isolated eigenvalues of finite type:
\begin{equation*}
    \sigma(\mathcal{A}) =  \{ \sigma \in \mathbb{C} : \sigma \textit{ is a characteristic value of } \Delta \} = \{\sigma \in \mathbb{C} : \sigma \mbox{ is a Floquet exponent} \}.
\end{equation*}
For $\sigma \in \sigma(\mathcal{A})$, there holds $m_g(\sigma,\mathcal{A}) = m_g(\Delta(\sigma)), m_a(\sigma,\mathcal{A}) = m_a(\Delta(\sigma)),  k(\sigma,\mathcal{A}) = k(\Delta(\sigma))$ and $r(\sigma,\mathcal{A}) = r(\Delta(\sigma))$, and the partial multiplicities of $\sigma$ considered as an eigenvalue of $\mathcal{A}$ are equal to the zero multiplicities of $\sigma$ considered as a characteristic value of $\Delta$. Furthermore, if $\{ \{{q}_{i,0},\dots,{q}_{i,k_i-1} \} : i=1,\dots,p\}$ is a canonical system of Jordan chains of $\Delta$ at $\sigma$, then $\{ \{\varphi_{i,0},\dots,\varphi_{i,k_i-1} \} : i=1,\dots,p \}$ is a canonical system of (generalized) eigenvectors of $\mathcal{A}$ at $\sigma$, where
\begin{equation*}
    \varphi_{i,k}(t)(\theta) = e^{\sigma \theta} \sum_{l=0}^k \frac{\theta^l}{l!}q_{i,k-l}(t+\theta), \quad \forall k = 0,\dots,k_{i}-1, \  t \in \mathbb{R}, \ \theta \in [r_{-},r_{+}].
\end{equation*}
\end{theorem}
\begin{proof}
Analogous to the proof of \cref{lemma:characinvertible}, one can verify that \ref{hyp:rangeequal} holds, provided that $\rho(\Delta) \neq \emptyset$ and $\rho(\Delta(z)) \neq \emptyset$ for all $z \in \mathbb{C}$. The only difference is that the operator norm of $C(z)$ from \cref{lemma:characinvertible} is bounded above by $ \max \{e^{r_{-} \Re(z)},e^{r_{+} \Re (z)} \} \|L\|_\infty < \infty$. The remaining parts of the proof can be completed as in \cref{lemma:characinvertible} and \cref{thm:eigenfunctionsDDE}.
\end{proof}

\begin{remark} \label{remark:rhoDeltaMFDE}
In contrast to the results for classical DDEs (\cref{subsec:classicalDDEs}) and iDDEs (\cref{subsec:infinitedelay}), the proof of \cref{thm:eigenfunctionsMFDE} requires the additional assumption that $\rho(\Delta) \neq \emptyset$. If one attempts to prove, along the same lines of \cref{lemma:characinvertible} that $\rho(\Delta) \neq \emptyset$, the difficulty arises in establishing the existence of a $z_0 \in \mathbb{C}$ such that $\|C(z_0)S(-z_0)^{-1} \| < 1$. A direct computation shows that this inequality is satisfied if $\min\{-r_{-},r_{+} \} < 1/(e\|L\|_\infty)$, meaning that either the retarded or advanced argument, or the operator norm $\|L\|_\infty$, is small enough. According to \cite[Lemma 4.3.2]{Hupkes2008}, there might be the possibility of $\rho(\Delta) = \emptyset$ for MFDEs, but necessary and (alternative) sufficient conditions remain a topic for further research. \hfill $\lozenge$
\end{remark}

The proofs of the following two results are similar to those of \cref{cor:resolventDDE} and \cref{cor:multiplicityDDE}, and therefore omitted.
\begin{corollary}
If $\rho(\Delta) \neq \emptyset$ and $z \in \mathbb{C}$ is not a Floquet exponent, then the resolvent of $\mathcal{A}$ at $z$ has the following representation
\begin{equation*}
    (R(z,\mathcal{A}) \varphi)(t)(\theta) = e^{z \theta} \varphi_z(t+\theta) + \int_\theta^0 e^{z(\theta-s)} \varphi(t+\theta-s)(s) ds, \quad \forall \varphi \in \mathcal{R}(zI-\mathcal{A}) \supseteq C_T^1(\mathbb{R},X),
\end{equation*}
where $\varphi_z$ is given by
\begin{equation*}
    \varphi_z = \Delta(z)^{-1} \bigg( t \mapsto \varphi(t)(0) + \int_{r_{-}}^{r_{+}} d_2 \eta(t,\theta)\int_{\theta}^{0} e^{z(\theta - s)} \varphi(t + \theta - s)(s) ds \bigg).
\end{equation*}
\end{corollary}

\begin{corollary}
Suppose that $\rho(\Delta) \neq \emptyset$. If $\sigma \in \mathbb{C}$ is a Floquet exponent such that $k(\sigma,\mathcal{A}) = r(\Delta(\sigma))$, then the multiplicity theorem $\dim \mathcal{N}((\sigma I - \mathcal{A})^{r(\Delta(\sigma))}) =  m_a(\Delta(\sigma))$ holds.
\end{corollary}

Since for MFDEs the initial value problem \eqref{eq:LinearMFDE} is in general ill-posed, and there is no solution operator $U$ available, we cannot use the tools from the proof of \cref{cor:FloquetsolDDE} to obtain the elementary solutions for MFDEs. However, if $\rho(\Delta) \neq \emptyset$ and $\{{q}_{0},\dots,{q}_{k-1} \}$ is a Jordan chain of $\Delta$ at a characteristic value $\sigma$, then one can verify (direct approach of the proof of \cref{cor:FloquetsolDDE}) that the elementary solutions $x$ of the form
\begin{equation*}
    x(t) = e^{\sigma(t-s)} \sum_{l=0}^{k-1} \frac{(t-s)^l}{l!}q_{k-l-1}(t), \quad \forall t \geq s.
\end{equation*}
satisfy \eqref{eq:LinearMFDE}. To explicitly construct in practice a Jordan chain $\{q_0,\dots,q_{k-1}\}$ for $\Delta$ at $\sigma$, one can use the same technique as described in \cref{remark:jordanchainpractice}. The only minor difference is that $\Delta(\sigma)q_0 = 0$ is now a periodic linear homogeneous MFDE and the remaining $k-1$ equations in \eqref{eq:Jordanchainexplicit} to obtain $q_1,\dots,q_{k-1}$ are periodic linear inhomogeneous MFDEs. The mentioned orthogonal collocation method is still suited to solve this particular linear system of differential equations.

\begin{remark} \label{remark:MFDEspectral}
In \cite{Hupkes2008}, Hupkes and Verduyn Lunel state and prove the periodic center manifold theorem for periodic MFDEs. To do so, they need the additional assumption (called (\textcolor{red}{HF}) in \cite{Hupkes2008}) that there exists a spectral gap for the Floquet exponents around the imaginary axis. Although they establish the existence of such a spectral gap for various specific subclasses of MFDEs (discrete delays and advances that are rationally related to the period, see \cite[Lemma 3.3 and Corollary 3.4-3.5]{Hupkes2008}), the general question whether all periodic MFDEs have such a spectral gap remained open in \cite{Hupkes2008}. This question is now (partially) answered in \cref{thm:eigenfunctionsMFDE}: if $\rho(\Delta) \neq \emptyset$, then this theorem states that all Floquet exponents of periodic linear MFDEs are isolated and of finite type, meaning that such a spectral gap exists. 
\hfill $\lozenge$
\end{remark}

\section{Conclusion and outlook} \label{sec:conclusion}
This manuscript considers characteristic operators for periodic linear evolutionary systems, thereby building upon and extending the work by Kaashoek and Verduyn Lunel \cite{Kaashoek1992} on characteristic matrices for autonomous linear evolutionary systems.
We applied our results to three significant subclasses of periodic FDEs: classical delay differential equations (DDEs, \cref{subsec:classicalDDEs}), delay differential equations with infinite delay (iDDEs, \cref{subsec:infinitedelay}) and mixed functional differential equations (MFDEs, \cref{subsec:MFDEs}). For this last class of equations, we resolved an open problem concerning the discrete spectral structure of the Floquet exponents (\cref{remark:MFDEspectral}). We conjecture that our results can also be used to yield spectral information of various other classes of periodic linear differential equations, including DDEs on $L^p$ spaces (with applications in control theory), abstract DDEs (with applications in mathematical neuroscience and PDEs with time delay), neutral FDEs (with applications in electrical and mechanical engineering), hyperbolic PDEs (with various applications in the analysis of waves) and Lotka-McKendrick-von Förster models (with applications in population dynamics). Characteristic matrices for the autonomous counterparts of the aforementioned classes of differential equations are provided in \cite[Section II]{Kaashoek1992}. Nevertheless, verifying the hypotheses of the general scheme from \cref{sec:construction} for the mentioned periodic models under consideration is far from straightforward and involves numerous technical challenges. For instance, consider the difference in length and complexity between the proofs of periodic classical DDEs (\cref{subsec:classicalDDEs}) and that of autonomous classical DDEs (\cite[Section II.1]{Kaashoek1992}). In some cases, it may even be necessary to relax or invoke alternative hypotheses to achieve comparable results.

When applying the abstract results to various classes of (functional) differential equations (\cref{sec:applications}), we constructed a characteristic operator $\Delta : \Omega \to L(\mathcal{F}_T(\mathbb{R},\mathbb{C}^n))$ for a given linear operator $\mathcal{A} \in L(\mathcal{F}_T(\mathbb{R},X))$ with $X = \mathcal{F}(I,\mathbb{C}^n)$, where $I \subseteq \mathbb{R}$ is a closed interval. For classical DDEs, it was observed in \cite[Section 3]{Article2} that the closure of $\mathcal{A}$ is the generator of an evolution semigroup $\mathcal{U}$ on $\mathcal{F}_T(\mathbb{R}, X)$ with $\mathcal{F} = C$. The choice of a $T$-periodic function space $\mathcal{F}_T(\mathbb{R},X)$ is not standard as evolution semigroups are in general considered on $C_0(\mathbb{R},X)$, the space of all $X$-valued continuous functions on $\mathbb{R}$ that vanish at infinity, or $L^p(\mathbb{R},X)$ due to their connections with the existence and regularity of non-autonomous Cauchy problems, see \cite{Engel2000,Chicone1997,Latushkin1996,Latushkin2004,Arendt2009,Nickel1997} for more information. However, the choice of a $T$-periodic function space seems now to be very natural regarding its applicability to characteristic operators. A promising direction of further research is to formulate and prove general results regarding the (spectral) relationship between $\mathcal{A}$, $\mathcal{U}$ and an associated characteristic operator $\Delta$, including spectral mapping theorems, Fredholm properties, and related aspects. In particular, it would be interesting to investigate the properties of the adjoints $\mathcal{A}^\star$, $\mathcal{U}^\star$ and $\Delta^\star$ in terms of the periodic pairings $\langle \cdot, \cdot \rangle_T$ from \cite[Section 3.3]{Article2} and \cite[Section 3]{Bosschaert2025} since these operators have, for example, applications in (numerical) bifurcation analysis \cite{Bosschaert2025,Article2,Kuznetsov2005,Witte2013,Witte2014}. We believe that this evolution semigroup approach could also offer valuable insights into the theory of series expansions, completeness theorems and small solutions for various classes of periodic evolution equations, see \cite[Section 8.3]{Hale1993}, \cite[Section 11]{Kaashoek2022} and \cite[Section 6.2]{Frasson2003} for applications towards time-delayed models when the period is a multiple of the delay, recall also \cref{remark:charoperatorTh} as a direction for further research in this spirit.

We finalize this paper with four open problems regarding our results:
\begin{enumerate}
    \item In the three examples from \cref{sec:applications}, it became clear that $\sigma(\mathcal{A}) \cap \Omega$ consists solely of isolated eigenvalues of finite type, as a direct consequence of \cref{lemma:compactspectra}. This spectral property aligns perfectly with the spectral structure of the characteristic matrix from Kaashoek and Verduyn Lunel \cite{Kaashoek1992} for autonomous linear evolution equations, see also \cref{remark:Yfinitedim}. This naturally raises the question of whether or not a characteristic operator $\Delta$ for $\mathcal{A}$ on $\Omega$ always exhibits this spectral structure. We conjecture that this is not the case as $\Delta(z)$ does not necessarily admit a compact resolvent for all $z \in \Omega$. However, a simple counterexample remains elusive.

    \item Resolving the preceding problem requires extending the spectral theory for closable linear operators (\cref{subsec:spectralprop}) and closable operator-valued functions (\cref{subsec:equivalence}) beyond eigenvalues and isolated spectral points. As noted in the discussion following \cref{lemma:spectraequalclosure}, this extension must account for the approximate point spectrum. Furthermore, a comprehensive theory necessitates a detailed analysis of the compression, continuous, residual, and essential spectrum for closable linear operators, alongside their associated (semi-)Fredholm properties and their relationship to the corresponding spectra of the closure as detailed in \cref{lemma:spectraequalclosure}, \cref{lemma:mgakineq} and \cref{lemma:orderpole}.
    
    \item In the three examples from \cref{sec:applications}, 
    the holomorphic operator-valued function $\Delta$ was always closed. However, we could not prove this directly from the general formula \eqref{eq:Delta(z)2}. An interesting topic for further research is to determine necessary and sufficient conditions on the operators $M, D, K, Q(z)$ and $\iota$ under which $\Delta(z)$ is (non-)closed, and whether this property is (in)dependent of $z \in \Omega$.
    
    \item We formulated and proved all our results on a $T$-periodic function space $\mathcal{F}_T(\mathbb{R},Z)$. The natural question arises if our results also hold in a $(T,z)$-periodic function space $\mathcal{F}_T(\mathbb{R},Z)_z$, with $(T,z) \in \mathbb{R}_{+} \times S^1$, consisting of \emph{Floquet-periodic} functions $f$ satisfying $f(t+T) = zf(t)$ for all $t \in \mathbb{R}$. It turns out that such Floquet-periodic functions appear naturally in the study of codimension one \cite{Kuznetsov2005,Bosschaert2025} and two \cite{Witte2013,Witte2014} bifurcations of limit cycles in (in)finite-dimensional dynamical systems. In particular, it would be interesting to see how the results on periodic spectral patterns from \cref{subsec:periodicspectral} should be generalized towards Floquet-periodic functions.
\end{enumerate}

\section*{Acknowledgements}
The authors thank Stein Meereboer (Radboud University) and Vitorio Courtens (Hasselt University) for helpful discussions and suggestions.

\appendix

\section{Additional notes on closed and closable linear operators} \label{appendix:closedclosable}
In this section of the appendix, we recall and prove some basic results related to closed and closable linear operators.

\begin{lemma}[{\cite[Lemma III.2.4]{Engel2000}} or {\cite[Problem 5.6]{Kato1995}}] \label{lemma:closedsum}
    Let $X$ and $Y$ be Banach spaces, $A : \mathcal{D}(A) \subseteq X \to Y$ a closed linear operator and $B : \mathcal{D}(B) \subseteq X \to Y$ a bounded linear operator with $\mathcal{D}(B) \supseteq \mathcal{D}(A)$. Then $A + B : \mathcal{D}(A) \subseteq X \to Y$ is closed.
\end{lemma}

\begin{lemma}[{\cite[Proposition B.2]{Engel2000}}] \label{lemma:compositionclosed}
    Let $X,Y$ and $Z$ be Banach spaces, $A : \mathcal{D}(A) \subseteq Y \to Z$ a closed linear operator and $B : X \to Y$ a bounded linear operator. Then $AB : \mathcal{D}(AB) \subseteq X \to Z$ with domain $\mathcal{D}(AB) \coloneqq \{\varphi \in X : B \varphi \in \mathcal{D}(A) \}$ is closed.
\end{lemma}

\begin{lemma}[{\cite[Proposition B.4]{Engel2000}}] \label{lemma:closabledef}
Let $X$ and $Y$ be Banach spaces. A linear operator $A : \mathcal{D}(A) \subseteq X \to Y$ is closable if and only if for every sequence $(\varphi_m)_m$ in $\mathcal{D}(A)$ converging in norm to zero and $(A \varphi_m)_m$ converges in norm to some $\psi$ in $Y,$ one has $\psi = 0$.
\end{lemma}

\begin{lemma} \label{lemma:closablesum}
   Let $X$ and $Y$ be Banach spaces, $A : \mathcal{D}(A) \subseteq X \to Y$ a closable linear operator and $B : \mathcal{D}(B) \subseteq X \to Y$ a bounded linear operator. Then $B$ and $A + B : \mathcal{D}(A) \cap \mathcal{D}(B) \subseteq X \to Y$ are closable.
\end{lemma}
\begin{proof}
To prove the first assertion, let $(\varphi_m)_m$ be a sequence in $\mathcal{D}(B)$ converging in norm to zero and suppose that $(B\varphi_m)_m$ converges in norm to some $\psi \in Y$. Then $\|\psi\| \leq \| \psi - B \varphi_m \| + \|B\| \|\varphi_m\| \to 0$ as $m \to \infty$, which shows that $\psi = 0$ and thus $B$ is closable.

To prove the second assertion, let $(\varphi_m)_m$ be a sequence in $\mathcal{D}(A + B) = \mathcal{D}(A) \cap \mathcal{D}(B)$ converging in norm to zero and suppose that $((A+B)\varphi_m)_m$ converges in norm to some $\psi \in Y$. Then $\|A \varphi_m - \psi \| \leq \|(A+B)\varphi_m - \psi \| + \|B \| \|\varphi_m\| \to 0$ as $m \to \infty$, which shows that $A\varphi_m \to \psi$ as $m \to \infty$. But, $A$ is closable and thus $\psi = 0$, which proves the claim.
\end{proof}

\begin{lemma} \label{lemma:compoclosable}
    Let $X,Y$ and $Z$ be Banach spaces, $A : \mathcal{D}(A) \subseteq Y \to Z$ a closable linear operator and $B : \mathcal{D}(B) \subseteq X \to Y$ a bounded linear operator. Then $AB : \mathcal{D}(AB) \subseteq X \to Z$ with domain $\mathcal{D}(AB) \coloneqq \{\varphi \in \mathcal{D}(B) : B \varphi \in \mathcal{D}(A) \}$ is closable.
\end{lemma}
\begin{proof}
Let $(\varphi_m)_m$ be a sequence in $\mathcal{D}(AB)$ converging in norm to zero and suppose that $(AB\varphi_m)_m$ converges in norm to some $\psi \in Z$. Then $\psi_m = B\varphi_m \to 0$ as $m \to \infty$ since $B$ is bounded. Hence, $AB \varphi_m = A \psi_m \to \psi$ as $m \to \infty$. Since $A$ is closable, we have that $\psi = 0$, which proves the claim.
\end{proof}

\section{Non-closedness and necessity of the resolvent set definition} \label{sec:nonclosed}
In this section of the appendix, we verify that the operators $D, D_0, \mathcal{A}$ and $\hat{\mathcal{A}}$ appearing in \cref{subsec:classicalDDEs} are not closed and that the ranges of $zI-D_0$ and $zI-\mathcal{A}$ are not the full space $C_T(\mathbb{R},X)$ when $T > 0$. Recall that a detailed discussion of these operators for $T = 0$ is provided in \cref{remark:autonomousDDE}. Moreover, the results from this section also apply to the operators $D, D_0, \mathcal{A}$ and $\hat{\mathcal{A}}$ appearing in \cref{subsec:infinitedelay} and \cref{subsec:MFDEs}, but some parts of the proofs might be slightly different.

\begin{proposition} \label{prop:DAnotclosed}
The operators $D, D_0, \mathcal{A}$ and $\hat{\mathcal{A}}$ in the setting of \cref{subsec:classicalDDEs} are not closed.
\end{proposition}
\begin{proof}
To prove that $D$ is not closed, let $(\psi_m)_m$ be a sequence in $C_T^1(\mathbb{R},\mathbb{C}^n)$ converging uniformly to some $\psi \in C_T(\mathbb{R},\mathbb{C}^n) \setminus C_T^1(\mathbb{R},\mathbb{C}^n)$. Define the sequence $(\varphi_m)_m$ in $\mathcal{D}(D)$ by $\varphi_m(t)(\theta) = \psi_m(t+\theta)$ and let $\varphi(t)(\theta) = \psi(t+\theta)$. Then $\| \varphi_m - \varphi \|_\infty = \| \psi_m - \psi \|_{\infty} \to 0$ and $D\varphi_m = 0$ as $m \to \infty$, but $\varphi \notin \mathcal{D}(D)$ since $\varphi$ is not continuously differentiable in the first variable. We conclude that $D$ is not closed. 

To prove that $D_0$ is not closed, let $(\psi_m)_m$ be a sequence in $C_T^1(\mathbb{R},\mathbb{C}^n)$ converging uniformly to some $\psi \in C_T(\mathbb{R},\mathbb{C}^n) \setminus C_T^1(\mathbb{R},\mathbb{C}^n)$. Define the sequence $(\varphi_m)_m$ in $\mathcal{D}(D_0)$ by $\varphi_m(t)(\theta) = \theta \psi_m(t+\theta)$ and let $\varphi(t)(\theta) = \theta \psi(t+\theta)$. Then $\| \varphi_m - \varphi \|_\infty \leq h \| \psi_m - \psi \|_{\infty} \to 0$ as $m \to \infty$ and we obtain $(D_0\varphi_m)(t)(\theta) = \psi_m(t+\theta)$. Hence, $(D_0\varphi_m)_m$ converges in norm to $t \mapsto \psi(t+\cdot)$. However, $\varphi$ is not continuously differentiable in the first variable and so $\varphi \notin \mathcal{D}(D_0)$, which proves that $D_0$ is not closed.

Before we prove that $\mathcal{A}$ and $\hat{\mathcal{A}}$ are not closed, let us first show that the linear operator $\hat{\mathcal{A}_0} : \mathcal{D}(\hat{\mathcal{A}_0}) \subseteq \hat{C}_T(\mathbb{R},X) \to \hat{C}_T(\mathbb{R},X)$ defined by
\begin{equation*}
    \mathcal{D}(\hat{\mathcal{A}_0}) \coloneqq \mathcal{D}(\hat{\mathcal{A}}), \quad \hat{\mathcal{A}_0}
    \begin{pmatrix}
    q\\
    \varphi
    \end{pmatrix}
    \coloneqq
    \begin{pmatrix}
    -GM\varphi \\
    D\varphi
    \end{pmatrix},
\end{equation*}
is not closed, where we recall the definition of $G$ and $M$ from the proof of \cref{lemma:operatorlemmaDDE}. To do this, let $(\psi_m)_m$ be a sequence in $C_T^1(\mathbb{R},\mathbb{C}^n)$ converging uniformly to some $\psi \in C_T(\mathbb{R},\mathbb{C}^n) \setminus C_T^1(\mathbb{R},\mathbb{C}^n)$. Define the sequence $(\varphi_m)_m$ in $\mathcal{D}(D)$ by $\varphi_m(t)(\theta) = \theta^2 \psi_m(t+\theta)$ and let $\varphi(t)(\theta) = \theta^2 \psi(t+\theta)$. Then $((0,\varphi_m))_m$ is a sequence in $\mathcal{D}(\hat{\mathcal{A}_0})$ and $\| (0,\varphi_m) - (0,\varphi) \|_\infty \leq h^2 \| \psi_m - \psi \|_\infty \to 0$ as $m \to \infty$ and $(\hat{\mathcal{A}_0}(0,\varphi_m))(t)(\theta) = (0, 2 \theta \psi_m(t+\theta))$. Hence, $(\hat{\mathcal{A}_0}(0,\varphi_m))_m$ converges in norm to $(0, t \mapsto (\theta \mapsto 2 \theta \psi(t+\theta)))$. However, $\varphi$ is not continuously differentiable in the first variable and so $\varphi \notin \mathcal{D}(D)$. Hence, $(0,\varphi) \notin \mathcal{D}(\hat{\mathcal{A}_0})$, i.e. $\hat{\mathcal{A}_0}$ is not closed. To prove that $\hat{\mathcal{A}}$ is not closed, consider the perturbation $\hat{L} : \hat{C}_T(\mathbb{R},X) \to \hat{C}_T(\mathbb{R},X)$ by 
\begin{equation*}
    \hat{L}
    \begin{pmatrix}
    q \\
    \varphi
    \end{pmatrix}
    \coloneqq
    \begin{pmatrix}
    M_L \varphi \\
    0
    \end{pmatrix}
\end{equation*}
and note that $\hat{L} \in \mathcal{L}(\hat{C}_T(\mathbb{R},X))$ as $L \in C_T(\mathbb{R},\mathcal{L}(X,\mathbb{C}^n))$. We observe that $\hat{\mathcal{A}} = \hat{\mathcal{A}_0} + \hat{L}$. If $\hat{\mathcal{A}}$ were closed, then $\hat{\mathcal{A}} - \hat{L} = \hat{\mathcal{A}_0}$ would be closed by \cref{lemma:closedsum}, which contradicts our previous finding. Thus, $\hat{\mathcal{A}}$ is not closed.

Let us now prove that $\mathcal{A}$ is not closed. Choose a sequence $(\psi_m)_m$ in $C_T^1(\mathbb{R}, X)$ converging uniformly to $\psi \in C_T(\mathbb{R}, X)$ such that the function $\Psi \in C_T(\mathbb{R},X)$ defined by
\begin{equation} \label{eq:Psiappendix}
    \Psi(t)(\theta) \coloneqq \int_0^\theta \psi(t+\theta-s)(s) ds
\end{equation}
is not in $\mathcal{D}(D)$. Note that such a sequence $(\psi_m)_m$ exists as shown above for the non-closedness of the linear operator $D$. Consider now a sequence $(q_m)_m$ in $C_T^1(\mathbb{R}, \mathbb{C}^n)$ satisfying the periodic inhomogeneous DDE
\begin{equation} \label{eq:dotqnDDE}
    \dot{q}_m(t) = L(t)[q_m]_t + H_m(t), \quad H_m(t) \coloneqq L(t)\bigg [\theta \mapsto \int_0^\theta \psi_m(t+\theta-s)(s) ds \bigg],
\end{equation}
and define the sequence $(\varphi_m)_m$ in $\mathcal{D}(D)$ by $\varphi_m(t)(\theta) = q_m(t+\theta) + \int_0^\theta \psi_m(t+\theta-s)(s) ds$. As $\dot{\varphi}_m(t)(0) = L(t)\varphi_m(t)$ due to \eqref{eq:dotqnDDE}, the sequence $(\varphi_m)_m$ is in $\mathcal{D}(\mathcal{A})$ and satisfies $\mathcal{A} \varphi_m = \psi_m$ for all $m \in \mathbb{N}$. Clearly, $H_m \to H$ uniformly since $\mathcal{A}\varphi_m = \psi_m \to \psi$ as $m \to \infty$, where $H(t) = L(t)[\theta \mapsto \int_0^\theta \psi(t+\theta-s)(s) ds]$. By the continuous dependence of solutions on the inhomogeneity, the sequence $(q_m)_m$ converges in $C_T^1(\mathbb{R},\mathbb{C}^n)$ to a map $q$ satisfying $\dot{q}(t) = L(t)q_t + H(t)$. Consequently, $\varphi_m \to \varphi$ uniformly as $m \to \infty$, where $\varphi(t)(\theta) = q(t+\theta) + \Psi(t)(\theta)$. Suppose by way of contradiction that $\varphi \in \mathcal{D}(D)$. Since $q \in C_T^1(\mathbb{R},\mathbb{C}^n)$, we obtain that $\Psi = \varphi - q \in \mathcal{D}(D)$, which is a contradiction due to the assumption on $\Psi$. Thus $\varphi \notin \mathcal{D}(D)$, and therefore $\varphi \notin \mathcal{D}(\mathcal{A})$, from which we conclude that $\mathcal{A}$ is not closed.
\end{proof}

\begin{proposition} \label{prop:rangenotequal}
In the setting of \cref{subsec:classicalDDEs}, the following statements hold:
\begin{enumerate}
    \item $\mathcal{R}(zI-D_0) = \{ \varphi \in C_T(\mathbb{R},X) : \mathcal{K}_z\varphi \in C_T(\mathbb{R},X) \}$, where $\mathcal{K}_z \in \mathcal{L}(C_T(\mathbb{R},X))$ is defined by
    \begin{equation*}
        (\mathcal{K}_z \varphi)(t)(\theta) \coloneqq \int_\theta^0 e^{z(\theta -s)}\varphi(t + \theta - s)(s) ds, \quad \forall z \in \mathbb{C}.
    \end{equation*}
    \item $C_T^1(\mathbb{R},X) \neq \mathcal{R}(zI-D_0) \neq C_T(\mathbb{R},X)$ for all $z \in \mathbb{C}$.
    \item $C_T^1(\mathbb{R},X) \neq \mathcal{R}(zI-\mathcal{A}) \neq C_T(\mathbb{R},X)$ for all $z \in \rho(\mathcal{A})$.
\end{enumerate}
\begin{proof}
Recall from \cref{lemma:operatorlemmaDDE} that $\rho(D_0) = \Omega = \mathbb{C}$. The first statement follows from \eqref{eq:resolventD02} and \eqref{eq:derivativesresolventDDE}. Note that $\mathcal{K}_z$ is bounded as $\|\mathcal{K}_z \varphi \| \leq M_z \|\varphi\|$ for all $\varphi \in C_T(\mathbb{R},X)$, where $M_z$ is defined in \eqref{eq:resolventD0Mz}.

Let us now prove the second statement. According to \cref{lemma:rangeinclusion}, it remains to show that $C_T^1(\mathbb{R},X) \neq \mathcal{R}(D_0) \neq C_T(\mathbb{R},X)$. Fix a $c \in \mathbb{R}$, consider a function $W\in C_T(\mathbb{R},\mathbb{C}^n) \setminus C_T^1(\mathbb{R},\mathbb{C}^n)$ and define the function $\varphi_c \in C_T(\mathbb{R},X)$ by $\varphi_c(t)(\theta) \coloneqq W(t+c\theta)$. Then a straightforward computation shows that
\begin{equation*}
    (\mathcal{K}_0 \varphi_c)(t)(\theta) =
    \begin{dcases}
        \theta W(t+\theta), \quad &c = 1, \\
        \frac{1}{c-1}\int_{t+\theta}^{t+c\theta}W(s) ds, \quad &c \neq 1.
    \end{dcases}
\end{equation*}
Hence, $\mathcal{K}_0 \varphi_1 \notin C_T^1(\mathbb{R},X)$ and $\mathcal{K}_0 \varphi_c \in C_T^1(\mathbb{R},X)$ for $c \neq 1$, while $\varphi_c \in C_T(\mathbb{R},X) \setminus C_T^1(\mathbb{R},X)$ for all $c \in \mathbb{R}$. This completes the proof due to the characterization provided in the first statement.

The third statement follows now from \cref{cor:spectralrelations} since $\Omega = \mathbb{C}$.
\end{proof}
\end{proposition}

\bibliographystyle{siamplain}
\bibliography{references}

@Book{Iooss1999,
  author    = {G{\'{e}}rard Iooss and Moritz Adelmeyer},
  publisher = {World Scientific},
  title     = {Topics in Bifurcation Theory and Applications},
  year      = {1999},
  month     = {jan},
  doi       = {10.1142/3990},
}

@Article{Wu1974,
  author    = {Wu, M.},
  journal   = {IEEE Transactions on Automatic Control},
  title     = {A note on stability of linear time-varying systems},
  year      = {1974},
  issn      = {0018-9286},
  month     = apr,
  number    = {2},
  pages     = {162--162},
  volume    = {19},
  doi       = {10.1109/tac.1974.1100529},
  publisher = {Institute of Electrical and Electronics Engineers (IEEE)},
}

@Book{Coppel1978,
  author    = {Coppel, W. A.},
  publisher = {Springer Berlin Heidelberg},
  title     = {Dichotomies in Stability Theory},
  year      = {1978},
  isbn      = {9783540359760},
  doi       = {10.1007/bfb0067780},
  issn      = {1617-9692},
  journal   = {Lecture Notes in Mathematics},
}

@Article{Iooss1988,
  author    = {G Iooss},
  journal   = {Journal of Differential Equations},
  title     = {Global characterization of the normal form for a vector field near a closed orbit},
  year      = {1988},
  month     = {nov},
  number    = {1},
  pages     = {47--76},
  volume    = {76},
  comment   = {One of the Lemmas in 'Numerical Periodic Normalization paper of Yuri' uses a technique in here},
  doi       = {10.1016/0022-0396(88)90063-0},
  publisher = {Elsevier {BV}},
}

@Article{Kuznetsov2005,
  author    = {{\mbox {Yu}}. A. Kuznetsov and W. Govaerts and E. J. Doedel and A. Dhooge},
  journal   = {{SIAM} Journal on Numerical Analysis},
  title     = {Numerical Periodic Normalization for Codim 1 Bifurcations of Limit Cycles},
  year      = {2005},
  month     = {jan},
  number    = {4},
  pages     = {1407--1435},
  volume    = {43},
  doi       = {10.1137/040611306},
  publisher = {Society for Industrial {\&} Applied Mathematics ({SIAM})},
}

@PhdThesis{Lyapunov1892,
  author  = {Aleksandr Mikhailovich {L}yapunov},
  school  = {Kharkov},
  title   = {Stability of Motion},
  year    = {1892, English translation, Academic Press, 1966.},
}

@Article{Clement1988,
  author    = {Ph. Cl{\'{e}}ment and O. Diekmann and M. Gyllenberg and H. J. A. M. Heijmans and H. R. Thieme},
  journal   = {Proceedings of the Royal Society of Edinburgh: Section A Mathematics},
  title     = {Perturbation theory for dual semigroups {II}. Time-dependent perturbations in the sun-reflexive case},
  year      = {1988},
  number    = {1-2},
  pages     = {145--172},
  volume    = {109},
  doi       = {10.1017/s0308210500026731},
  publisher = {Cambridge University Press ({CUP})},
}

@Article{Bosschaert2024a,
  author    = {Bosschaert, M. M. and Kuznetsov, {\mbox{Yu}}. A.},
  journal   = {SIAM Journal on Applied Dynamical Systems},
  title     = {Bifurcation Analysis of {B}ogdanov–{T}akens Bifurcations in Delay Differential Equations},
  year      = {2024},
  issn      = {1536-0040},
  month     = jan,
  number    = {1},
  pages     = {553--591},
  volume    = {23},
  doi       = {10.1137/22m1527532},
  publisher = {Society for Industrial & Applied Mathematics (SIAM)},
}

@MastersThesis{Janssens2010,
  author  = {Sebastiaan G. Janssens},
  school  = {Utrecht University},
  title   = {On a Normalization Technique for Codimension Two Bifurcations of Equilibria of Delay Differential Equations},
  year    = {2010},
  comment = {Calculation of critical normal form coefficients with examples and introduction to DDE theory.},
  url     = {http://dspace.library.uu.nl/handle/1874/312252},
}

@Article{Bosschaert2020,
  author    = {Maikel M. Bosschaert and Sebastiaan G. Janssens and \mbox{Yu. A.} Kuznetsov},
  journal   = {{SIAM} Journal on Applied Dynamical Systems},
  title     = {Switching to Nonhyperbolic Cycles from Codimension Two Bifurcations of Equilibria of Delay Differential Equations},
  year      = {2020},
  month     = {jan},
  number    = {1},
  pages     = {252--303},
  volume    = {19},
  doi       = {10.1137/19m1243993},
  publisher = {Society for Industrial {\&} Applied Mathematics ({SIAM})},
}

@Book{Kato1995,
  author    = {Kato, Tosio},
  publisher = {Springer Berlin Heidelberg},
  title     = {Perturbation Theory for Linear Operators},
  year      = {1995},
  isbn      = {9783642662829},
  doi       = {10.1007/978-3-642-66282-9},
  issn      = {1431-0821},
  journal   = {Classics in Mathematics},
}

@Article{Gohberg1978,
  author    = {Gohberg, I. C and Kaashoek, M. A and Lay, D. C},
  journal   = {Journal of Functional Analysis},
  title     = {Equivalence, linearization, and decomposition of holomorphic operator functions},
  year      = {1978},
  issn      = {0022-1236},
  month     = apr,
  number    = {1},
  pages     = {102--144},
  volume    = {28},
  doi       = {10.1016/0022-1236(78)90081-2},
  publisher = {Elsevier BV},
}

@Article{Engstroem2017,
  author    = {Engström, Christian and Torshage, Axel},
  journal   = {Integral Equations and Operator Theory},
  title     = {On Equivalence and Linearization of Operator Matrix Functions with Unbounded Entries},
  year      = {2017},
  issn      = {1420-8989},
  month     = nov,
  number    = {4},
  pages     = {465--492},
  volume    = {89},
  doi       = {10.1007/s00020-017-2415-5},
  publisher = {Springer Science and Business Media LLC},
}

@Article{Magal2009,
  author    = {Magal, Pierre and Ruan, Shigui},
  journal   = {Memoirs of the American Mathematical Society},
  title     = {Center manifolds for semilinear equations with non-dense domain and applications to Hopf bifurcation in age structured models},
  year      = {2009},
  issn      = {1947-6221},
  number    = {951},
  pages     = {1-80},
  volume    = {202},
  doi       = {10.1090/s0065-9266-09-00568-7},
  publisher = {American Mathematical Society (AMS)},
}

@Article{Frasson2003,
  author    = {Miguel V. S. Frasson and Sjoerd M. {Verduyn Lunel}},
  journal   = {Integral Equations and Operator Theory},
  title     = {Large Time Behaviour of Linear Functional Differential Equations},
  year      = {2003},
  month     = {sep},
  number    = {1},
  pages     = {91--121},
  volume    = {47},
  doi       = {10.1007/s00020-003-1155-x},
  publisher = {Springer Science and Business Media {LLC}},
}

@Article{Kaashoek1992,
  author    = {M. A. Kaashoek and S. M. {Verduyn Lunel}},
  journal   = {Transactions of the American Mathematical Society},
  title     = {Characteristic matrices and spectral properties of evolutionary systems},
  year      = {1992},
  month     = {feb},
  number    = {2},
  pages     = {479--517},
  volume    = {334},
  doi       = {10.1090/s0002-9947-1992-1155350-0},
  publisher = {American Mathematical Society ({AMS})},
}

@Book{Engel2000,
  author    = {Klaus-Jochen Engel and Rainer Nagel},
  publisher = {Springer-Verlag},
  title     = {One-Parameter Semigroups for Linear Evolution Equations},
  year      = {2000},
  doi       = {10.1007/b97696},
}

@Book{Hale1993,
  author    = {Jack K. Hale and Sjoerd M. {Verduyn Lunel}},
  publisher = {Springer New York},
  title     = {Introduction to Functional Differential Equations},
  year      = {1993},
  doi       = {10.1007/978-1-4612-4342-7},
}

@Book{Diekmann1995,
  author    = {Odo Diekmann and Sjoerd M. {Verduyn Lunel} and Stephan A. van Gils and Hanns-Otto Walther},
  publisher = {Springer New York},
  title     = {Delay Equations},
  year      = {1995},
  doi       = {10.1007/978-1-4612-4206-2},
}

@Article{Harterich2002,
  author    = {Jorg H\"arterich and Bjorn Sandstede and Arnd Scheel},
  journal   = {Indiana University Mathematics Journal},
  title     = {Exponential dichotomies for linear non-autonomous functional differential equations of mixed type},
  year      = {2002},
  number    = {5},
  pages     = {1081--1110},
  volume    = {51},
  doi       = {10.1512/iumj.2002.51.2188},
  publisher = {Indiana University Mathematics Journal},
}

@Book{Kaashoek2022,
  author    = {Marinus A. Kaashoek and Sjoerd M. {Verduyn Lunel}},
  publisher = {Springer International Publishing},
  title     = {Completeness Theorems and Characteristic Matrix Functions},
  year      = {2022},
  doi       = {10.1007/978-3-031-04508-0},
}

@Book{Gohberg1990,
  author    = {Israel C. Gohberg and Seymour Goldberg and Marinus A. Kaashoek},
  publisher = {Birkhäuser Basel},
  title     = {Classes of Linear Operators Vol. {I}},
  year      = {1990},
  doi       = {10.1007/978-3-0348-7509-7},
}

@Article{Article2,
  author    = {Lentjes, Bram and Spek, Len and Bosschaert, Maikel M. and Kuznetsov, {\mbox{Yu}}. A.},
  journal   = {Journal of Differential Equations},
  title     = {Periodic normal forms for bifurcations of limit cycles in {DDE}s},
  year      = {2025},
  issn      = {0022-0396},
  month     = apr,
  pages     = {631--694},
  volume    = {423},
  doi       = {10.1016/j.jde.2025.01.064},
  publisher = {Elsevier BV},
}

@Book{Taylor1986,
  author    = {A. E. Taylor and D. C. Lay},
  publisher = {Krieger},
  title     = {Introduction to Functional Analysis},
  year      = {1986},
  isbn      = {9780898749519},
}

@Article{MalletParet1999,
  author    = {Mallet-Paret, John},
  journal   = {Journal of Dynamics and Differential Equations},
  title     = {The {F}redholm Alternative for Functional Differential Equations of Mixed Type},
  year      = {1999},
  issn      = {1040-7294},
  number    = {1},
  pages     = {1--47},
  volume    = {11},
  doi       = {10.1023/a:1021889401235},
  publisher = {Springer Science and Business Media LLC},
}

@Article{Article1,
  author    = {Lentjes, Bram and Spek, Len and Bosschaert, Maikel M. and Kuznetsov, {\mbox{Yu}} A.},
  journal   = {Journal of Dynamics and Differential Equations},
  title     = {Periodic Center Manifolds for DDEs in the Light of Suns and Stars},
  year      = {2023},
  issn      = {1572-9222},
  month     = aug,
  number    = {1},
  pages     = {815--858},
  volume    = {37},
  doi       = {10.1007/s10884-023-10289-9},
  publisher = {Springer Science and Business Media LLC},
}

@Article{Wu1997,
  author    = {Wu, Jianhong and Zou, Xingfu},
  journal   = {Journal of Differential Equations},
  title     = {Asymptotic and Periodic Boundary Value Problems of Mixed {FDE}s and Wave Solutions of Lattice Differential Equations},
  year      = {1997},
  issn      = {0022-0396},
  month     = apr,
  number    = {2},
  pages     = {315--357},
  volume    = {135},
  doi       = {10.1006/jdeq.1996.3232},
  publisher = {Elsevier BV},
}

@Article{Faria2002,
  author    = {Faria, Teresa and Huang, Wenzhang and Wu, Jianhong},
  journal   = {SIAM Journal on Mathematical Analysis},
  title     = {Smoothness of Center Manifolds for Maps and Formal Adjoints for Semilinear {FDE}s in General {B}anach Spaces},
  year      = {2002},
  issn      = {1095-7154},
  month     = jan,
  number    = {1},
  pages     = {173--203},
  volume    = {34},
  doi       = {10.1137/s0036141001384971},
  publisher = {Society for Industrial & Applied Mathematics (SIAM)},
}

@Misc{Veltz2020,
  author      = {Veltz, Romain},
  month       = Jul,
  title       = {{BifurcationKit.jl}},
  year        = {2020},
  hal_id      = {hal-02902346},
  hal_version = {v1},
  institution = {{Inria Sophia-Antipolis}},
  url         = {https://hal.archives-ouvertes.fr/hal-02902346},
}

@InBook{Krauskopf2022,
  author    = {Krauskopf, Bernd and Sieber, Jan},
  pages     = {195--245},
  publisher = {Springer International Publishing},
  title     = {Bifurcation Analysis of Systems With Delays: Methods and Their Use in Applications},
  year      = {2022},
  isbn      = {9783031011290},
  month     = sep,
  booktitle = {Controlling Delayed Dynamics},
  doi       = {10.1007/978-3-031-01129-0\_7},
  issn      = {2309-3706},
}

@Article{Chicone1997,
  author    = {Chicone, C. and Latushkin, Y.},
  journal   = {Journal of Differential Equations},
  title     = {Center Manifolds for Infinite Dimensional Nonautonomous Differential Equations},
  year      = {1997},
  issn      = {0022-0396},
  month     = dec,
  number    = {2},
  pages     = {356--399},
  volume    = {141},
  doi       = {10.1006/jdeq.1997.3343},
  publisher = {Elsevier BV},
}

@Article{Latushkin2004,
  author    = {Latushkin, Yuri and Tomilov, Yuri},
  journal   = {Illinois Journal of Mathematics},
  title     = {Fredholm properties of evolution semigroups},
  year      = {2004},
  issn      = {0019-2082},
  month     = jul,
  number    = {3},
  pages     = {999-1020},
  volume    = {48},
  doi       = {10.1215/ijm/1258131066},
  publisher = {Duke University Press},
}

@Article{Latushkin1996,
  author    = {Latushkin, Y. and Montgomery-Smith, S. and Randolph, T.},
  journal   = {Journal of Differential Equations},
  title     = {Evolutionary Semigroups and Dichotomy of Linear Skew-Product Flows on Locally Compact Spaces with {B}anach Fibers},
  year      = {1996},
  issn      = {0022-0396},
  month     = feb,
  number    = {1},
  pages     = {73--116},
  volume    = {125},
  doi       = {10.1006/jdeq.1996.0025},
  publisher = {Elsevier BV},
}

@Article{Nickel1997,
  author    = {Nickel, Gregor},
  journal   = {Abstract and Applied Analysis},
  title     = {Evolution semigroups for nonautonomous {C}auchy problems},
  year      = {1997},
  issn      = {1085-3375},
  number    = {1–2},
  pages     = {73--95},
  volume    = {2},
  doi       = {10.1155/s1085337597000286},
  publisher = {Hindawi Limited},
}

@Article{Arendt2009,
  author    = {Arendt, Wolfgang and J. Rabier, Patrick},
  journal   = {Communications on Pure \& Applied Analysis},
  title     = {Linear evolution operators on spaces of periodic functions},
  year      = {2009},
  issn      = {1553-5258},
  number    = {1},
  pages     = {5--36},
  volume    = {8},
  doi       = {10.3934/cpaa.2009.8.5},
  publisher = {American Institute of Mathematical Sciences (AIMS)},
}

@Article{Rustichini1989,
  author    = {Rustichini, Aldo},
  journal   = {Journal of Dynamics and Differential Equations},
  title     = {Functional differential equations of mixed type: The linear autonomous case},
  year      = {1989},
  issn      = {1572-9222},
  month     = apr,
  number    = {2},
  pages     = {121--143},
  volume    = {1},
  doi       = {10.1007/bf01047828},
  publisher = {Springer Science and Business Media LLC},
}

@InProceedings{MALLETPARET2005,
  author    = {Mallet-Paret, John and {Verduyn Lunel}, Sjoerd Michiel},
  booktitle = {Equadiff 2003},
  title     = {MIXED-TYPE FUNCTIONAL DIFFERENTIAL EQUATIONS, HOLOMORPHIC FACTORIZATION, AND APPLICATIONS},
  year      = {2005},
  month     = feb,
  publisher = {World Scientific},
  doi       = {10.1142/9789812702067\_0007},
}

@Book{Harm2008,
  author    = {Harm Bart and André C. M. Ran and Israel C. Gohberg and Marinus A. Kaashoek},
  publisher = {Birkhäuser Basel},
  title     = {Factorization of matrix and operator functions: the state space method},
  year      = {2008},
  isbn      = {9783764382681},
  volume    = {Operator Theory: Advances and Applications 178},
  doi       = {10.1007/978-3-7643-8268-1},
  journal   = {Operator Theory: Advances and Applications},
}

@Book{Mortad2022,
  author    = {Mortad, Mohammed Hichem},
  publisher = {Springer International Publishing},
  title     = {Counterexamples in Operator Theory},
  year      = {2022},
  isbn      = {9783030978143},
  doi       = {10.1007/978-3-030-97814-3},
}

@Article{Naito1976,
  author    = {Naito, Toshiki},
  journal   = {Journal of Differential Equations},
  title     = {On autonomous linear functional differential equations with infinite retardations},
  year      = {1976},
  issn      = {0022-0396},
  month     = jul,
  number    = {2},
  pages     = {297--315},
  volume    = {21},
  doi       = {10.1016/0022-0396(76)90124-8},
  publisher = {Elsevier BV},
}

@Article{Dhooge2003,
  author    = {A. Dhooge and W. Govaerts and {\mbox{Yu}}. A. Kuznetsov},
  journal   = {{ACM} Transactions on Mathematical Software},
  title     = {{MATCONT}: A {MATLAB} package for numerical bifurcation analysis of {ODE}s},
  year      = {2003},
  month     = {jun},
  number    = {2},
  pages     = {141--164},
  volume    = {29},
  doi       = {10.1145/779359.779362},
  publisher = {Association for Computing Machinery ({ACM})},
}

@Article{Dhooge2008,
  author    = {Dhooge, A. and Govaerts, W. and Kuznetsov, {\mbox{Yu}}. A. and Meijer, H. G. E. and Sautois, B.},
  journal   = {Mathematical and Computer Modelling of Dynamical Systems},
  title     = {New features of the software {MATCONT} for bifurcation analysis of dynamical systems},
  year      = {2008},
  issn      = {1744-5051},
  month     = apr,
  number    = {2},
  pages     = {147--175},
  volume    = {14},
  doi       = {10.1080/13873950701742754},
  publisher = {Informa UK Limited},
}

@Misc{Lentjes2026,
  author    = {Lentjes, Bram and Dani\"els, Seppe and Follon, Meinder and Kuznetsov, {\mbox{Yu}}. A.},
  title     = {Center Manifolds and Normal Forms for Nonlinearly Periodically Forced {DDE}s},
  year      = {2026},
  copyright = {Creative Commons Attribution 4.0 International},
  doi       = {10.48550/ARXIV.2601.18918},
  publisher = {arXiv},
}

@Misc{Bosschaert2025,
  author    = {Bosschaert, M. M. and Lentjes, B. and Spek, L. and Kuznetsov, {\mbox{Yu}}. A.},
  title     = {Numerical Periodic Normalization at Codim 1 Bifurcations of Limit Cycles in {DDE}s},
  year      = {2025},
  copyright = {Creative Commons Attribution 4.0 International},
  doi       = {10.48550/ARXIV.2505.19786},
  publisher = {arXiv},
}

@Article{Witte2013,
  author    = {V. {De Witte} and F. {Della Rossa} and W. Govaerts and  \mbox {Yu}. A. Kuznetsov},
  journal   = {{SIAM} Journal on Applied Dynamical Systems},
  title     = {Numerical Periodic Normalization for Codim 2 Bifurcations of Limit Cycles: computational Formulas, Numerical Implementation, and Examples},
  year      = {2013},
  month     = {jan},
  number    = {2},
  pages     = {722--788},
  volume    = {12},
  doi       = {10.1137/120874904},
  publisher = {Society for Industrial {\&} Applied Mathematics ({SIAM})},
}

@Book{Kuznetsov2023a,
  author    = {Kuznetsov, {\mbox {Yu}}. A.},
  publisher = {Springer, Cham},
  title     = {Elements of {A}pplied {B}ifurcation {T}heory},
  year      = {2023},
  edition   = {4th},
  isbn      = {978-3-031-22006-7; 978-3-031-22007-4},
  series    = {Applied Mathematical Sciences},
  volume    = {112},
  doi       = {10.1007/978-3-031-22007-4},
  pages     = {xxvi+703},
}

@Article{LentjesCMODE,
  author    = {Lentjes, Bram and Windmolders, Mattias and Kuznetsov, {\mbox{Yu}}. A.},
  journal   = {International Journal of Bifurcation and Chaos},
  title     = {Periodic Center Manifolds for Nonhyperbolic Limit Cycles in {ODE}s},
  year      = {2023},
  issn      = {1793-6551},
  month     = dec,
  number    = {15},
  pages     = {1-29},
  volume    = {33},
  doi       = {10.1142/s0218127423501845},
  publisher = {World Scientific Pub Co Pte Ltd},
}

@Article{Witte2014,
  author    = {V. {De Witte} and W. Govaerts and  \mbox {Yu}. A. Kuznetsov and H. G. E. Meijer},
  journal   = {Physica D: Nonlinear Phenomena},
  title     = {Analysis of bifurcations of limit cycles with {L}yapunov exponents and numerical normal forms},
  year      = {2014},
  month     = {feb},
  pages     = {126--141},
  volume    = {269},
  doi       = {10.1016/j.physd.2013.12.002},
  publisher = {Elsevier {BV}},
}

@Article{Hupkes2006,
  author    = {Hupkes, H. J. and {Verduyn Lunel}, S. M.},
  journal   = {Journal of Dynamics and Differential Equations},
  title     = {Center Manifold Theory for Functional Differential Equations of Mixed Type},
  year      = {2006},
  issn      = {1572-9222},
  month     = dec,
  number    = {2},
  pages     = {497--560},
  volume    = {19},
  doi       = {10.1007/s10884-006-9055-9},
  publisher = {Springer Science and Business Media LLC},
}

@Article{Hupkes2008,
  author    = {H. J. Hupkes and S. M. {Verduyn Lunel}},
  journal   = {Journal of Differential Equations},
  title     = {Center manifolds for periodic functional differential equations of mixed type},
  year      = {2008},
  month     = {sep},
  number    = {6},
  pages     = {1526--1565},
  volume    = {245},
  doi       = {10.1016/j.jde.2008.02.026},
  publisher = {Elsevier {BV}},
}

@Article{Engelborghs2001,
  author    = {K. Engelborghs and T. Luzyanina and K. J. Hout and D. Roose},
  journal   = {{SIAM} Journal on Scientific Computing},
  title     = {Collocation Methods for the Computation of Periodic Solutions of Delay Differential Equations},
  year      = {2001},
  month     = {jan},
  number    = {5},
  pages     = {1593--1609},
  volume    = {22},
  doi       = {10.1137/s1064827599363381},
  publisher = {Society for Industrial {\&} Applied Mathematics ({SIAM})},
}

@Book{Hale2009,
  author    = {Hale, Jack K.},
  publisher = {Dover Publications},
  title     = {Ordinary differential equations},
  year      = {2009},
  isbn      = {9780486472119},
}

@Article{Just2000,
  author    = {Just, Wolfram},
  journal   = {Physica D: Nonlinear Phenomena},
  title     = {On the eigenvalue spectrum for time-delayed {F}loquet problems},
  year      = {2000},
  issn      = {0167-2789},
  month     = aug,
  number    = {1–2},
  pages     = {153--165},
  volume    = {142},
  doi       = {10.1016/s0167-2789(00)00051-8},
  publisher = {Elsevier BV},
}

@InCollection{Lunel2001,
  author    = {Sjoerd M. {Verduyn Lunel}},
  booktitle = {Systems, Approximation, Singular Integral Operators, and Related Topics},
  publisher = {Birkhäuser Basel},
  title     = {Spectral theory for delay equations},
  year      = {2001},
  pages     = {465--507},
  doi       = {10.1007/978-3-0348-8362-7\_19},
}

@Article{Skubachevskii2006,
  author    = {Skubachevskii, Alexander L. and Walther, Hans-Otto},
  journal   = {Journal of Dynamics and Differential Equations},
  title     = {On the {F}loquet Multipliers of Periodic Solutions to Non-linear Functional Differential Equations},
  year      = {2006},
  issn      = {1572-9222},
  month     = apr,
  number    = {2},
  pages     = {257--355},
  volume    = {18},
  doi       = {10.1007/s10884-006-9006-5},
  publisher = {Springer Science and Business Media LLC},
}

@Article{Szalai2006,
  author    = {R{\'{o}}bert Szalai and G{\'{a}}bor St{\'{e}}p{\'{a}}n and S. John Hogan},
  journal   = {{SIAM} Journal on Scientific Computing},
  title     = {Continuation of Bifurcations in Periodic Delay-Differential Equations Using Characteristic Matrices},
  year      = {2006},
  month     = {jan},
  number    = {4},
  pages     = {1301--1317},
  volume    = {28},
  doi       = {10.1137/040618709},
  publisher = {Society for Industrial {\&} Applied Mathematics ({SIAM})},
}

@Article{Sieber2011,
  author    = {Jan Sieber and Robert Szalai},
  journal   = {{SIAM} Journal on Applied Dynamical Systems},
  title     = {Characteristic Matrices for Linear Periodic Delay Differential Equations},
  year      = {2011},
  month     = {jan},
  number    = {1},
  pages     = {129--147},
  volume    = {10},
  doi       = {10.1137/100796455},
  publisher = {Society for Industrial {\&} Applied Mathematics ({SIAM})},
}

@Article{Wolff2022a,
  author    = {Babette A. J. de Wolff},
  journal   = {Dynamical Systems},
  title     = {Characteristic matrix functions for delay differential equations with symmetry},
  year      = {2022},
  month     = {oct},
  number    = {1},
  pages     = {30--51},
  volume    = {38},
  doi       = {10.1080/14689367.2022.2132136},
  publisher = {Informa {UK} Limited},
}

@Article{Pyragas1992,
  author    = {Pyragas, K.},
  journal   = {Physics Letters A},
  title     = {Continuous control of chaos by self-controlling feedback},
  year      = {1992},
  issn      = {0375-9601},
  month     = nov,
  number    = {6},
  pages     = {421--428},
  volume    = {170},
  doi       = {10.1016/0375-9601(92)90745-8},
  publisher = {Elsevier BV},
}

@Misc{Bosschaert2024c,
  author       = {Maikel M. Bosschaert and Bram Lentjes and Len Spek and \mbox {Yu}. A. Kuznetsov},
  howpublished = {GitHub},
  title        = {{P}eriodic{N}ormalization{DDE}s},
  year         = {2024},
  booktitle    = {GitHub},
  publisher    = {GitHub repository},
  url          = {https://github.com/mmbosschaert/PeriodicNormalizationDDEs.git},
}

@Article{Engelborghs2002,
  author    = {K. Engelborghs and T. Luzyanina and D. Roose},
  journal   = {{ACM} Transactions on Mathematical Software},
  title     = {Numerical bifurcation analysis of delay differential equations using {DDE}-{B}if{T}ool},
  year      = {2002},
  month     = {mar},
  number    = {1},
  pages     = {1--21},
  volume    = {28},
  doi       = {10.1145/513001.513002},
  publisher = {Association for Computing Machinery ({ACM})},
}

@Misc{Sieber2014,
  author    = {Sieber, Jan and Engelborghs, Koen and Luzyanina, Tatyana and Samaey, Giovanni and Roose, Dirk},
  title     = {{DDE}-{B}if{T}ool Manual - Bifurcation analysis of delay differential equations},
  year      = {2014},
  copyright = {arXiv.org perpetual, non-exclusive license},
  doi       = {10.48550/ARXIV.1406.7144},
  publisher = {arXiv},
}

@Article{Engelborghs2002a,
  author    = {K. Engelborghs and E. J. Doedel},
  journal   = {Numerische Mathematik},
  title     = {Stability of piecewise polynomial collocation for computing periodic solutions of delay differential equations},
  year      = {2002},
  month     = {jun},
  number    = {4},
  pages     = {627--648},
  volume    = {91},
  doi       = {10.1007/s002110100313},
  publisher = {Springer Science and Business Media {LLC}},
}

@Article{Gohberg1971,
  author    = {Gohberg, I. C. and Sigal, E. I.},
  journal   = {Mathematics of the USSR-Sbornik},
  title     = {AN OPERATOR GENERALIZATION OF THE LOGARITHMIC RESIDUE THEOREM AND THE THEOREM OF {R}OUCHé},
  year      = {1971},
  issn      = {0025-5734},
  month     = apr,
  number    = {4},
  pages     = {603--625},
  volume    = {13},
  doi       = {10.1070/sm1971v013n04abeh003702},
  publisher = {IOP Publishing},
}

@Book{Hino1991,
  author    = {Hino, Yoshiyuki and Murakami, Satoru and Naito, Toshiki},
  publisher = {Springer Berlin Heidelberg},
  title     = {Functional Differential Equations with Infinite Delay},
  year      = {1991},
  isbn      = {9783540473886},
  doi       = {10.1007/bfb0084432},
  issn      = {1617-9692},
  journal   = {Lecture Notes in Mathematics},
}

@Article{Hale1978,
  author  = {Hale, Jack Kenneth and Kato, Junji},
  journal = {Funkcial Ekvac},
  title   = {Phase space for retarded equations with infinite delay},
  year    = {1978},
  pages   = {11-41},
  volume  = {21},
}

@Article{Diekmann2012,
  author    = {Diekmann, Odo and Gyllenberg, Mats},
  journal   = {Journal of Differential Equations},
  title     = {Equations with infinite delay: Blending the abstract and the concrete},
  year      = {2012},
  issn      = {0022-0396},
  month     = jan,
  number    = {2},
  pages     = {819--851},
  volume    = {252},
  doi       = {10.1016/j.jde.2011.09.038},
  publisher = {Elsevier BV},
}

@Article{Murakami1989,
  author  = {Murakami, Satoru and Naito, Toshiki},
  journal = {Funkcialaj Ekvacioj},
  title   = {Fading Memory Spaces and Stability Properties for Functional Differential Equations with Infinite Delay},
  year    = {1989},
  pages   = {91-105},
  volume  = {32},
}

@Article{Hupkes2010,
  author    = {Hupkes, Hermen Jan and Sandstede, Björn},
  journal   = {SIAM Journal on Applied Dynamical Systems},
  title     = {Traveling Pulse Solutions for the Discrete {F}itz{H}ugh–{N}agumo System},
  year      = {2010},
  issn      = {1536-0040},
  month     = jan,
  number    = {3},
  pages     = {827--882},
  volume    = {9},
  doi       = {10.1137/090771740},
  publisher = {Society for Industrial & Applied Mathematics (SIAM)},
}

@Article{Hupkes2012,
  author    = {Hupkes, H. and Sandstede, B.},
  journal   = {Transactions of the American Mathematical Society},
  title     = {Stability of pulse solutions for the discrete {F}itz{H}ugh–{N}agumo system},
  year      = {2012},
  issn      = {1088-6850},
  month     = jul,
  number    = {1},
  pages     = {251--301},
  volume    = {365},
  doi       = {10.1090/s0002-9947-2012-05567-x},
  publisher = {American Mathematical Society (AMS)},
}

\end{sloppypar}
\end{document}